\documentclass[11pt]{article}
\usepackage{amsmath,amssymb}
\usepackage{graphicx}
\usepackage{amsfonts}

\usepackage{amsthm}
\newtheorem{thm}{Theorem}[section]

\newtheorem{corollary}[thm]{Corollary}
\newtheorem{define}[thm]{Definition}

\usepackage[USenglish]{babel} 
\usepackage[T1]{fontenc}
\usepackage[ansinew]{inputenc}
\usepackage{mathrsfs}
\usepackage{stmaryrd} 
\usepackage{epstopdf} 
\usepackage{verbatim}
\usepackage{MnSymbol}  
\usepackage{accents}   
\usepackage{bbm}			 
\usepackage{color}     
\usepackage{authblk}   

\numberwithin{equation}{section}
\numberwithin{figure}{section}

\DeclareMathOperator{\tr}{tr}
\DeclareMathOperator{\divergence}{div}

\title{Variational integrators for stochastic dissipative Hamiltonian systems}
\date{}
\author[1,2]{Michael Kraus\thanks{\texttt{michael.kraus@ipp.mpg.de}}}
\author[1]{Tomasz M. Tyranowski\thanks{\texttt{tomasz.tyranowski@ipp.mpg.de}}}

\affil[1]{\small Max-Planck-Institut f\"ur Plasmaphysik \authorcr Boltzmannstra{\ss}e 2, 85748 Garching, Germany}
\affil[2]{\small Technische Universit\"{a}t M\"{u}nchen, Zentrum Mathematik \authorcr Boltzmannstra{\ss}e 3, 85748 Garching, Germany}

\begin{document}

\maketitle

\begin{abstract}
Variational integrators are derived for structure-preserving simulation of stochastic forced Hamiltonian systems. The derivation is based on a stochastic discrete Hamiltonian which approximates a type-II stochastic generating function for the stochastic flow of the Hamiltonian system. The generating function is obtained by introducing an appropriate stochastic action functional and considering a stochastic generalization of the deterministic Lagrange-d'Alembert principle. Our approach presents a general methodology to derive new structure-preserving numerical schemes. The resulting integrators satisfy a discrete version of the stochastic Lagrange-d'Alembert principle, and in the presence of symmetries, they also satisfy a discrete counterpart of Noether's theorem. Furthermore, mean-square and weak Lagrange-d'Alembert Runge-Kutta methods are proposed and tested numerically to demonstrate their superior long-time numerical stability and energy behavior compared to non-geometric methods. The Vlasov-Fokker-Planck equation is considered as one of the numerical test cases, and a new geometric approach to collisional kinetic plasmas is presented.
\end{abstract}

\section{Introduction}
\label{sec:intro}
Stochastic differential equations (SDEs) play an important role in modeling dynamical systems subject to internal or external random fluctuations. Standard references include \cite{ArnoldSDE}, \cite{IkedaWatanabe1989}, \cite{KloedenPlatenSDE}, \cite{Kunita1997}, \cite{MilsteinBook}, \cite{ProtterStochastic}. Within this class of problems, we are interested in stochastic forced Hamiltonian systems, which take the form 

\begin{align}
\label{eq: Stochastic dissipative Hamiltonian system}
d_t q &= \phantom{-}\frac{\partial H}{\partial p}dt + \sum_{i=1}^m\frac{\partial h_i}{\partial p}\circ dW^i(t), \nonumber \\
d_t p &= \bigg[-\frac{\partial H}{\partial q} + F(q,p) \bigg] dt + \sum_{i=1}^m \bigg[-\frac{\partial h_i}{\partial q}+f_i(q,p)\bigg]\circ dW^i(t),
\end{align}

\noindent
where $H=H(q,p)$ and $h_i=h_i(q,p)$ for $i=1,\ldots,m$ are the Hamiltonian functions, $F=F(q,p)$ and $f_i=f_i(q,p)$ are the forcing terms, $W(t)=(W^1(t),\ldots,W^m(t))$ is the standard $m$-dimensional Wiener process, and $\circ$ denotes Stratonovich integration. We use $d_t$ to denote the stochastic differential of stochastic processes (other than the Wiener process $W(t)$) to avoid confusion with the exterior derivative $d$ of differential forms. The system \eqref{eq: Stochastic dissipative Hamiltonian system} can be formally regarded as a classical forced Hamiltonian system with the randomized Hamiltonian given by $\widehat H(q,p,t) = H(q,p) + \sum_{i=1}^m h_i(q,p)\circ \dot W^i(t)$, and the randomized forcing given by $\widehat F(q,p,t) = F(q,p) + \sum_{i=1}^m f_i(q,p)\circ \dot W^i(t)$, where $H(q,p)$ and $F(q,p)$ are the deterministic Hamiltonian and forcing, respectively, and  $h_i(q,p)$, $f_i(q,p)$ represent the intensity of the noise. Equation~\eqref{eq: Stochastic dissipative Hamiltonian system} is a generalization of stochastic Hamiltonian systems considered in \cite{Bismut}, \cite{HolmTyranowskiGalerkin}, \cite{LaCa-Or2008}, and \cite{MilsteinRepin}. Such systems can be used to model, e.g., mechanical systems with uncertainty, or error, assumed to arise from random forcing, limited precision of experimental measurements, or unresolved physical processes on which the Hamiltonian of the deterministic system might otherwise depend. Applications arise in many models in physics, chemistry, and biology. Particular examples include molecular dynamics (see, e.g., \cite{Beard2000}, \cite{Izaguirre2001}, \cite{Lacasta2004}, \cite{Skeel1999}), dissipative particle dynamics (see, e.g., \cite{Ripoll2001}), investigations of the dispersion of passive tracers in turbulent flows (see, e.g., \cite{Sawford2001}, \cite{Thomson1987}), energy localization in thermal equilibrium (see, e.g., \cite{Reigada1999}), lattice dynamics in strongly anharmonic crystals (see, e.g., \cite{Gornostyrev1996}), description of noise induced transport in stochastic ratchets (see, e.g., \cite{Landa1998}), and collisional kinetic plasmas (\cite{Kleiber2011}, \cite{Sonnendrucker2015}).

As occurs for other SDEs, most Hamiltonian SDEs cannot be solved analytically and one must resort to numerical simulations to obtain approximate solutions. In principle, general purpose stochastic numerical schemes for SDEs can be applied to stochastic Hamiltonian systems. However, as for their deterministic counterparts, stochastic Hamiltonian systems possess several important geometric features: in the case of systems without forcing, their phase space flows (almost surely) preserve the \emph{symplectic} structure (\cite{Bismut}, \cite{MilsteinRepin2001}, \cite{MilsteinRepin});  when the forcing terms are present, then the solutions also satisfy the stochastic \emph{Lagrange-d'Alembert principle}, as will be shown in Section~\ref{sec:Lagrange-d'Alembert principle for stochastic forced Hamiltonian systems}, and in some special cases the phase space flow may be \emph{conformally symplectic} (see \cite{BouRabeeOwhadiBoltzmannGibbs}, \cite{Hong2017}, \cite{MilsteinQuasiSymplectic}). When simulating these systems numerically, it is therefore advisable that the numerical scheme also preserves such geometric features. Geometric integration of deterministic Hamiltonian systems has been thoroughly studied (see \cite{HLWGeometric}, \cite{McLachlanQuispel}, \cite{SanzSerna} and the references therein) and symplectic integrators have been shown to demonstrate superior performance in long-time simulations of Hamiltonian systems without forcing, compared to non-symplectic methods; so it is natural to pursue a similar approach for stochastic Hamiltonian systems. This is a relatively recent pursuit. Stochastic symplectic integrators are discussed in \cite{Anmarkrud2017}, \cite{AntonWeak2014}, \cite{AntonGlobalError2013}, \cite{Anton2013}, \cite{Burrage2012}, \cite{Burrage2014}, \cite{AntonHighOrder2014}, \cite{Hong2015}, \cite{MaDing2012}, \cite{MaDing2015}, \cite{MilsteinRepin2001}, \cite{MilsteinRepin}, \cite{Misawa2010}, \cite{SunWang2016}, \cite{WangPHD}, \cite{Wang2014}, \cite{Wang2017}, \cite{Zhou2017}.

Long-time accuracy and near preservation of the Hamiltonian by symplectic integrators applied to deterministic Hamiltonian systems have been rigorously studied using the so-called backward error analysis (see, e.g., \cite{HLWGeometric} and the references therein). To the best of our knowledge, such general rigorous results have not yet been proved for stochastic Hamiltonian systems, but backward error analysis for SDEs is currently an active area of research. Modified SDEs associated with some particular numerical schemes are considered in \cite{AbdulleCohen2012}, \cite{Debussche2012}, \cite{Deng2017}, \cite{Shardlow2006}, \cite{Wang2016}, and \cite{Zygalakis2011}. Backward error analysis for the Langevin equation with additive noise is studied for several integrators in \cite{AbdulleVilmart2015}, \cite{Kopec2014}, and \cite{Kopec2015}. Recently, backward error analysis for a weak symplectic scheme applied to a stochastic Hamiltonian system has been presented in \cite{Anton2019}. Asymptotic preservation of large deviation principles by stochastic symplectic methods is investigated in \cite{ChenHongLDP}. The numerical evidence and partial theoretical results to date are promising and suggest that stochastic geometric integrators indeed possess the property of very accurately capturing the evolution of the Hamiltonian $H$ over long time intervals.

An important class of geometric integrators are \emph{variational integrators}. This type of numerical schemes is based on discrete variational principles and provides a natural framework for the discretization of Lagrangian systems, including forced, dissipative, or constrained ones. These methods have the advantage that they are symplectic when applied to systems without forcing, and in the presence of a symmetry, they satisfy a discrete version of Noether's theorem. For an overview of variational integration for deterministic systems see \cite{MarsdenWestVarInt}; see also \cite{HallLeokSpectral}, \cite{JaySPARK}, \cite{KaneMarsden2000}, \cite{LeokShingel}, \cite{LeokZhang}, \cite{OberBlobaum2016}, \cite{OberBlobaum2015}, \cite{RowleyMarsden}, \cite{TyranowskiDesbrunLinearLagrangians}, \cite{VankerschaverLeok}. Variational integrators were introduced in the context of finite-dimensional mechanical systems, but were later generalized to Lagrangian field theories (see \cite{MarsdenPatrickShkoller}) and applied in many computations, for example in elasticity, electrodynamics, or fluid dynamics; see \cite{LewAVI}, \cite{Pavlov}, \cite{SternDesbrun}, \cite{TyranowskiDesbrunRAMVI}.

Stochastic variational integrators were first introduced in \cite{BouRabeeSVI} and further studied in \cite{BouRabeeConstrainedSVI}. However, those integrators were restricted to the special case when the Hamiltonian functions $h_i=h_i(q)$ were independent of $p$, and only low-order Runge-Kutta types of discretization were considered. Stochastic discrete Hamiltonian variational integrators applicable to a general class of Hamiltonian systems were proposed in \cite{HolmTyranowskiGalerkin} by generalizing the variational principle for deterministic systems introduced in \cite{LeokZhang} and applying a Galerkin type of discretization; see also \cite{HolmTyranowskiSolitons}. In the present work we extend the ideas put forth in \cite{HolmTyranowskiGalerkin} to forced systems of the form \eqref{eq: Stochastic dissipative Hamiltonian system} and propose the corresponding Lagrange-d'Alembert variational integrators.

When the forcing terms in Eq.~\eqref{eq: Stochastic dissipative Hamiltonian system} are linear functions of the momentum variable $p$, then the stochastic flow of the system is conformally symplectic (see \cite{MilsteinQuasiSymplectic} and Section~\ref{sec:Conformal symplecticity and phase space volume}). Stochastic conformally symplectic integrators for such systems were proposed in \cite{BouRabeeOwhadiBoltzmannGibbs}, \cite{BouRabeeOwhadi2010}, and \cite{Hong2017}. Quasi-symplectic integrators were introduced in \cite{MilsteinQuasiSymplectic} and further studied in \cite{MilsteinErgodic}. These ideas are very interesting, but at present seem to be limited only to systems that exhibit a very special form, that is, systems with separable Hamiltonians, linear forcing terms, and additive noise. The stochastic Lagrange-d'Alembert variational integrators introduced in Section~\ref{sec:Stochastic Lagrange-d'Alembert variational integrators} are applicable to the general class of systems of the form \eqref{eq: Stochastic dissipative Hamiltonian system} and preserve their underlying variational structure.

\paragraph{Main content}
The main content of the remainder of this paper is, as follows. 
\begin{description}
\item
In Section~\ref{sec:Lagrange-d'Alembert principle for stochastic forced Hamiltonian systems} we introduce a stochastic Lagrange-d'Alembert principle and a stochastic generating function suitable for considering stochastic forced Hamiltonian systems, and we discuss their properties. 
\item
In Section~\ref{sec:Stochastic Lagrange-d'Alembert variational integrators} we present a general framework for constructing stochastic Lagrange-d'Alembert variational integrators, prove the discrete stochastic Lagrange-d'Alembert principle, propose mean-square and weak stochastic Lagrange-d'Alembert Runge-Kutta methods, and present several particularly interesting examples of low-stage schemes. We also discuss connections with the idea of quasi-symplectic integrators. 
\item
In Section~\ref{sec:Numerical experiments} we present the results of our numerical tests, which verify the excellent long-time performance of our integrators compared to some popular non-geometric methods. In particular, as one of the test cases we consider the Vlasov-Fokker-Planck equation, which is used as a model for collisional kinetic plasmas.
\item
Section~\ref{sec:Summary} contains the summary of our work.
\end{description}

\section{Lagrange-d'Alembert principle for stochastic forced Hamiltonian systems}
\label{sec:Lagrange-d'Alembert principle for stochastic forced Hamiltonian systems}

The stochastic variational integrators proposed in \cite{BouRabeeSVI} and \cite{BouRabeeConstrainedSVI} were formulated for dynamical systems which are described by a Lagrangian and which are subject to noise whose magnitude depends only on the position $q$. Therefore, these integrators can be extended to \eqref{eq: Stochastic dissipative Hamiltonian system} only if the Hamiltonian functions $h_i=h_i(q)$ are independent of $p$ and the Hamiltonian $H$ is non-degenerate (i.e., the associated Legendre transform is invertible). However, in the case of general $h_i=h_i(q,p)$ the paths $q(t)$ of the system become almost surely nowhere differentiable, which poses a difficulty in interpreting the meaning of the corresponding Lagrangian. To avoid these kind of issues, in \cite{HolmTyranowskiGalerkin} an action functional based on a phase space Lagrangian was introduced, and variational integrators for unforced Hamiltonian systems were constructed. In the present work we extend the approach taken in \cite{HolmTyranowskiGalerkin} to include forced Hamiltonian systems. To begin, in the next section, we will introduce an appropriate stochastic action functional and show that it can be used to define a type-II generating function for the stochastic flow of the system \eqref{eq: Stochastic dissipative Hamiltonian system}.

\subsection{Stochastic Lagrange-d'Alembert principle}
\label{sec:Stochastic Lagrange-d'Alembert principle}

Let the Hamiltonian functions $H: T^*Q \longrightarrow \mathbb{R}$ and $h_i: T^*Q \longrightarrow \mathbb{R}$ for $i=1,\ldots,m$ be defined on the cotangent bundle $T^*Q$ of the configuration manifold $Q$, and let $(q,p)$ denote the canonical coordinates on $T^*Q$. The Hamiltonian forces $F: T^*Q \longrightarrow T^*Q$ and $f_i: T^*Q \longrightarrow T^*Q$ for $i=1,\ldots,m$ are fiber-preserving mappings with the coordinate representations $F(q,p)= (q, F(q,p))$ and $f_i(q,p)= (q, f_i(q,p))$, respectively, where by a slight abuse of notation we use the same symbol to denote the force and its local representation. For simplicity, in this work we assume that the configuration manifold has a vector space structure, $Q \cong \mathbb{R}^N$, so that $T^*Q = Q \times Q^* \cong \mathbb{R}^N \times \mathbb{R}^N$ and $TQ = Q \times Q \cong \mathbb{R}^N \times \mathbb{R}^N$. In this case, the natural pairing between one-forms and vectors can be identified with the scalar product on $\mathbb{R}^N$, that is, $\langle (q,p), (q,\dot q) \rangle = p\cdot \dot q$, where $(q,\dot q)$ denotes the coordinates on $TQ$. Let $(\Omega, \mathcal{F},\mathbb{P})$ be the probability space with the filtration $\{\mathcal{F}_t\}_{t \geq 0}$, and let $W(t)=(W^1(t),\ldots,W^m(t))$ denote a standard $m$-dimensional Wiener process on that probability space (such that $W(t)$ is $\mathcal{F}_t$-measurable). We will assume that the Hamiltonian functions and the forcing terms are sufficiently smooth and satisfy all the necessary conditions for the existence and uniqueness of solutions to \eqref{eq: Stochastic dissipative Hamiltonian system}, and their extendability to a given time interval $[t_a, t_b]$ with $t_b > t_a \geq 0$. One possible set of such assumptions can be formulated by considering the It\^o form of~\eqref{eq: Stochastic dissipative Hamiltonian system},

\begin{equation}
\label{eq: Ito form of the stochastic Hamiltonian system}
d_t z = A(z) dt + B(z)dW(t),
\end{equation}

\noindent
with $z=(q,p)$ and

\begin{align}
\label{eq: Ito coefficients}
A(z) =
\begin{pmatrix}
 \phantom{-}\frac{\partial H}{\partial p}+\frac{1}{2}\sum_{i=1}^m\Big[ \frac{\partial^2 h_i}{\partial p \partial q} \frac{\partial h_i}{\partial p} + \frac{\partial^2 h_i}{\partial p^2} \Big( f_i -\frac{\partial h_i}{\partial q} \Big) \Big ]  \\
-\frac{\partial H}{\partial q}+F+\frac{1}{2}\sum_{i=1}^m\Big[ \Big(\frac{\partial^2 h_i}{\partial q \partial p} - \frac{\partial f_i}{\partial p} \Big) \Big(\frac{\partial h_i}{\partial p}-f_i \Big) - \Big( \frac{\partial^2 h_i}{\partial q^2}-\frac{\partial f_i}{\partial q} \Big) \frac{\partial h_i}{\partial p} \Big ]
\end{pmatrix},
\qquad\quad B(z) =
\begin{pmatrix}
 \phantom{-}\big(\frac{\partial h}{\partial p}\big)^T \\
-\big(\frac{\partial h}{\partial q}\big)^T+f
\end{pmatrix},
\end{align}

\noindent
where $\partial^2 h_i/\partial q^2$, $\partial^2 h_i/\partial p^2$, and $\partial^2 h_i/\partial q \partial p$ denote the Hessian matrices of $h_i$, whereas $\partial h/\partial q$, $\partial h/\partial p$, $\partial f_i/\partial q$, and $\partial f_i/\partial p$ denote the Jacobian matrices of $h=(h_1,\ldots,h_m)$ and $f_i$, respectively, and the $n \times m$ forcing matrix $f$ is defined as $f=(f_1, \ldots, f_m)$. For simplicity and clarity of the exposition, throughout this paper we assume that (see \cite{ArnoldSDE}, \cite{IkedaWatanabe1989}, \cite{KloedenPlatenSDE}, \cite{Kunita1997})

\begin{itemize}
	\item[(H1)] $H$ and $h_i$ for $i=1,\ldots,m$ are $C^2$ functions of their arguments,
	\item[(H2)] $F$ and $f_i$ for $i=1,\ldots,m$ are $C^1$ functions of their arguments,
	\item[(H3)] $A$ and $B$ are globally Lipschitz.
\end{itemize}

\noindent
These assumptions are sufficient for our purposes, but could be relaxed if necessary. Define the space

\begin{equation}
\label{eq:Definition of the q space}
C([t_a, t_b]) = \big\{ (q,p):\Omega \times [t_a, t_b] \longrightarrow T^*Q  \, \big| \, \text{$q$, $p$ are almost surely continuous $\mathcal{F}_t$-adapted semimartingales} \big\}.
\end{equation}

\noindent
Since we assume $T^*Q \cong \mathbb{R}^N \times \mathbb{R}^N$, the space $C([t_a, t_b])$ is a vector space (see \cite{ProtterStochastic}). Therefore, we can identify the tangent space $TC([t_a, t_b]) \cong C([t_a, t_b])\times C([t_a, t_b])$. We can now define the following stochastic action functional, $\mathcal{B}: \Omega \times C([t_a, t_b]) \longrightarrow \mathbb{R}$,

\begin{equation}
\label{eq:Stochastic action functional}
\mathcal{B}\big[q(\cdot),p(\cdot) \big] = p(t_b)q(t_b) - \int_{t_a}^{t_b} \Big[ p\circ d_t q - H\big(q(t),p(t)\big)\,dt - \sum_{i=1}^m h_i\big(q(t),p(t)\big)\circ dW^i(t)\Big],
\end{equation}

\noindent
where $\circ$ denotes Stratonovich integration, and we have omitted writing the elementary events $\omega \in \Omega$ as arguments of functions, following the standard convention in stochastic analysis. For a given curve $\big( q(t), p(t) \big)$ in $T^*Q$ and its arbitrary variation $\big(\delta q(t), \delta p(t) \big)$, we define the corresponding variation of the action functional as

\begin{equation}
\label{eq:Definition of the variation of the action functional}
\delta \mathcal{B}\big[q(\cdot), p(\cdot) \big] \equiv \frac{d}{d\epsilon} \bigg|_{\epsilon=0}\mathcal{B}\big[q(\cdot)+\epsilon \delta q(\cdot), p(\cdot)+\epsilon \delta p(\cdot) \big].
\end{equation}

\begin{thm}[{\bf Stochastic Lagrange-d'Alembert Principle in Phase Space}]
\label{thm:Stochastic Lagrange-d'Alembert Principle}
Suppose that $H(q,p)$, $F(q,p)$, and $h_i(q,p)$, $f_i(q,p)$ for $i=1,\ldots,m$ satisfy conditions (H1)-(H3). If the curve $\big( q(t), p(t) \big)$ in $T^*Q$ satisfies the stochastic forced Hamiltonian system \eqref{eq: Stochastic dissipative Hamiltonian system} for $t\in [t_a,t_b]$, where $t_b \geq t_a >0$, then it also satisfies the integral equation

\begin{equation}
\label{eq:Stochastic Lagrange-d'Alembert Principle}
\delta \mathcal{B}\big[q(\cdot), p(\cdot) \big] - \int_{t_a}^{t_b} F\big(q(t),p(t)\big)\cdot \delta q(t)\,dt - \sum_{i=1}^m \int_{t_a}^{t_b} f_i\big(q(t),p(t)\big)\cdot \delta q(t)\circ dW^i(t)=0\,,
\end{equation} 

\noindent
almost surely for all variations $\big(\delta q(\cdot), \delta p(\cdot) \big) \in C([t_a, t_b])$ such that almost surely $\delta q(t_a)=0$ and $\delta p(t_b)=0$.
\end{thm}

\begin{proof}
Let the curve $\big( q(t), p(t) \big)$ in $T^*Q$ satisfy \eqref{eq: Stochastic dissipative Hamiltonian system} for $t\in [t_a,t_b]$. It then follows that the stochastic processes $q(t)$ and $p(t)$ are almost surely continuous, $\mathcal{F}_t$-adapted semimartingales, that is, $\big( q(\cdot), p(\cdot) \big) \in C([t_a, t_b])$ (see \cite{ArnoldSDE}, \cite{ProtterStochastic}). We calculate the variation \eqref{eq:Definition of the variation of the action functional} as

\begin{align}
\label{eq:Calculating delta B}
\delta \mathcal{B}\big[q(\cdot), p(\cdot) \big] &= p(t_b)\delta q(t_b) - \int_{t_a}^{t_b} p(t)\circ d_t \delta q(t) - \int_{t_a}^{t_b} \delta p(t)\circ d_t q(t) \nonumber \\
                                       &\phantom{=} + \int_{t_a}^{t_b} \bigg[ \frac{\partial H}{\partial q}\big(q(t),p(t)\big)\,\delta q(t) + \frac{\partial H}{\partial p}\big(q(t),p(t)\big)\,\delta p(t) \bigg]\,dt \nonumber \\
                                       &\phantom{=} + \sum_{i=1}^m \int_{t_a}^{t_b} \bigg[ \frac{\partial h_i}{\partial q}\big(q(t),p(t)\big)\,\delta q(t) + \frac{\partial h_i}{\partial p}\big(q(t),p(t)\big)\,\delta p(t) \bigg]\circ dW^i(t),
\end{align}

\noindent
where we have used the end point condition, $\delta p(t_b)=0$. Since the Hamiltonians are $C^2$ and the processes $q(t)$, $p(t)$ are almost surely continuous, in the last two lines we have used a dominated convergence argument to interchange differentiation with respect to $\epsilon$ and integration with respect to $t$ and $W(t)$. Upon applying the integration by parts formula for semimartingales (see \cite{ProtterStochastic}), we find

\begin{equation}
\label{eq:Integration by parts for semimartingales}
\int_{t_a}^{t_b} p(t)\circ d_t \delta q(t) = p(t_b)\delta q(t_b) -  p(t_a)\delta q(t_a) - \int_{t_a}^{t_b} \delta q(t)\circ d_t p(t).
\end{equation}

\noindent
Substituting and rearranging terms produces,

\begin{align}
\label{eq:Calculating delta B continued}
\delta \mathcal{B}\big[q(\cdot), p(\cdot) \big] &= \int_{t_a}^{t_b} \delta q(t) \bigg[\circ d_t p(t) + \frac{\partial H}{\partial q}\big(q(t),p(t)\big)\,dt + \sum_{i=1}^m \frac{\partial h_i}{\partial q}\big(q(t),p(t)\big)\circ dW^i(t) \bigg] \nonumber \\
&- \int_{t_a}^{t_b} \delta p(t) \bigg[\circ d_t q(t) - \frac{\partial H}{\partial p}\big(q(t),p(t)\big)\,dt - \sum_{i=1}^m \frac{\partial h_i}{\partial p}\big(q(t),p(t)\big)\circ dW^i(t) \bigg],
\end{align}

\noindent
where we have used $\delta q(t_a)=0$. Therefore, we have

\begin{align}
\label{eq: Calculating the Lagrange-d'Alembert terms}
&\delta \mathcal{B}\big[q(\cdot), p(\cdot) \big] - \int_{t_a}^{t_b} F\big(q(t),p(t)\big)\cdot \delta q(t)\,dt - \sum_{i=1}^m \int_{t_a}^{t_b} f_i\big(q(t),p(t)\big)\cdot \delta q(t)\circ dW^i(t) \nonumber \\
&= \underbrace{\int_{t_a}^{t_b} \delta q(t) \Bigg[\circ d_t p(t) + \bigg(\frac{\partial H}{\partial q}\big(q(t),p(t)\big)-F\big(q(t),p(t)\big) \bigg)\,dt + \sum_{i=1}^m \bigg( \frac{\partial h_i}{\partial q}\big(q(t),p(t)\big)-f_i\big(q(t),p(t)\big) \bigg)\circ dW^i(t) \Bigg]}_{A} \nonumber \\
&- \underbrace{\int_{t_a}^{t_b} \delta p(t) \bigg[\circ d_t q(t) - \frac{\partial H}{\partial p}\big(q(t),p(t)\big)\,dt - \sum_{i=1}^m \frac{\partial h_i}{\partial p}\big(q(t),p(t)\big)\circ dW^i(t) \bigg]}_{B}.
\end{align}

\noindent
Since $\big( q(t), p(t) \big)$ satisfy \eqref{eq: Stochastic dissipative Hamiltonian system}, then by definition we have that almost surely for all $t \in [t_a,t_b]$,

\begin{equation}
\label{eq: Integral form of the solution of the stochastic Hamiltonian system}
q(t) = q(t_a) + \underbrace{\int_{t_a}^t \frac{\partial H}{\partial p}(q(s),p(s)) \,ds}_{M_0(t)} + \sum_{i=1}^m \underbrace{\int_{t_a}^t \frac{\partial h_i}{\partial p}(q(s),p(s)) \circ dW^i(s)}_{M_i(t)},
\end{equation}

\noindent
that is, $q(t)$ can be represented as the sum of the semi-martingales $M_i(t)$ for $i=0,\ldots,m$, where the sample paths of the process $M_0(t)$ are almost surely continuously differentiable. Let us calculate

\begin{align}
\label{eq: One of the integrals in delta B}
\int_{t_a}^{t_b} \delta p(t)\circ d_t q(t) &= \int_{t_a}^{t_b} \delta p(t)\circ d_t\Big(q(t_a)+M_0(t)+\sum_{i=1}^m M_i(t)\Big) \nonumber \\
                                        &= \int_{t_a}^{t_b} \delta p(t)\circ d_t M_0(t) + \sum_{i=1}^m \int_{t_a}^{t_b} \delta p(t)\circ d_t M_i(t) \nonumber \\
																				&= \int_{t_a}^{t_b} \delta p(t)\frac{\partial H}{\partial p}(q(t),p(t)) \,dt + \sum_{i=1}^m \int_{t_a}^{t_b} \delta p(t) \frac{\partial h_i}{\partial p}(q(t),p(t)) \circ dW^i(t),
\end{align}

\noindent
where in the last equality we have used the standard property of the Riemann-Stieltjes integral for the first term, as $M_0(t)$ is almost surely differentiable, and the associativity property of the Stratonovich integral for the second term (see \cite{ProtterStochastic}, \cite{IkedaWatanabe1989}). Substituting \eqref{eq: One of the integrals in delta B} in the term $B$ of \eqref{eq: Calculating the Lagrange-d'Alembert terms}, we show that $B=0$. By a similar argument we also prove that $A=0$. Therefore, the left-hand side of \eqref{eq: Calculating the Lagrange-d'Alembert terms} is equal to zero, almost surely.\\
\end{proof}

\noindent
{\bf Remark:} It is natural to expect that the converse theorem, that is, if $\big( q(\cdot), p(\cdot) \big)$ satisfy the integral principle \eqref{eq:Stochastic Lagrange-d'Alembert Principle}, then the curve $\big( q(t), p(t) \big)$ is a solution to \eqref{eq: Stochastic dissipative Hamiltonian system}, should also hold, although a larger class of variations $(\delta q, \delta p)$ may be necessary. Variants of such a theorem for systems without forcing have been proved in L\'azaro-Cam\'i \& Ortega \cite{LaCa-Or2008} and Bou-Rabee \& Owhadi \cite{BouRabeeSVI}. We leave this as an open question. Here, we will use the action functional \eqref{eq:Stochastic action functional} and the Lagrange-d'Alembert principle \eqref{eq:Stochastic Lagrange-d'Alembert Principle} to construct numerical schemes, and we will directly verify that these numerical schemes converge to solutions of \eqref{eq: Stochastic dissipative Hamiltonian system}.\\

\subsection{Stochastic type-II generating function and forcing}
\label{sec:Stochastic type-II generating function and forcing}

When the functions $H(q,p)$, $F(q,p)$, $h_i(q,p)$, and $f_i(q,p)$ satisfy standard measurability and regularity conditions (e.g., (H1)-(H3)), then the system \eqref{eq: Stochastic dissipative Hamiltonian system} possesses a pathwise unique stochastic flow $F_{t,t_0}: \Omega \times T^*Q \longrightarrow T^*Q$. It can be proved that for fixed $t,t_0$ this flow is mean-square differentiable with respect to the $q$, $p$ arguments, and is also almost surely a diffeomorphism (see \cite{ArnoldSDE}, \cite{IkedaWatanabe1989}, \cite{KloedenPlatenSDE}, \cite{Kunita1997}). We will show below that the action functional \eqref{eq:Stochastic action functional} can be used to construct a type II generating function for $F_{t,t_0}$. Let $(\bar q(t), \bar p(t))$ be a particular solution of \eqref{eq: Stochastic dissipative Hamiltonian system} on $[t_a,t_b]$. Suppose that for almost all $\omega \in \Omega$ there is an open neighborhood $\mathcal{U}(\omega)\subset Q$ of $\bar q(\omega,t_a)$, an open neighborhood $\mathcal{V}(\omega)\subset Q^*$ of $\bar p(\omega,t_b)$, and an open neighborhood $\mathcal{W}(\omega)\subset T^*Q$ of the curve $(\bar q(\omega,t), \bar p(\omega,t))$ such that for all $q_a \in \mathcal{U}(\omega)$ and $p_b \in \mathcal{V}(\omega)$ there exists a pathwise unique solution $(\bar q(\omega,t; q_a, p_b), \bar p(\omega,t; q_a, p_b))$ of \eqref{eq: Stochastic dissipative Hamiltonian system} which satisfies $\bar q(\omega,t_a; q_a, p_b)=q_a$, $\bar p(\omega,t_b; q_a, p_b)=p_b$, and $(\bar q(\omega,t; q_a, p_b), \bar p(\omega,t; q_a, p_b)) \in \mathcal{W}(\omega)$ for $t_a\leq t \leq t_b$. (As in the deterministic case, for $t_b$ sufficiently close to $t_a$ one can argue that such neighborhoods exist; see \cite{MarsdenRatiuSymmetry}.) Define the function $S:\mathcal{Y} \longrightarrow \mathbb{R}$ as

\begin{equation}
\label{eq:Definition of the generating function}
S(q_a,p_b) = \mathcal{B}\big[\bar q(\cdot;q_a,p_b), \bar p(\cdot; q_a,p_b) \big],
\end{equation}

\noindent
where the domain $\mathcal{Y} \subset \Omega \times Q \times Q^*$ is given by $\mathcal{Y} = \bigcup\limits_{\omega \in \Omega} \{\omega \} \times \mathcal{U}(\omega) \times \mathcal{V}(\omega)$. Define further the two functions $F^\pm:\mathcal{Y} \longrightarrow \mathbb{R}^N$ as

\begin{align}
\label{eq:Definition of the exact forcing}
F^-(q_a,p_b) &= \int_{t_a}^{t_b} \bigg( \frac{\partial \bar q(t;q_a,p_b)}{\partial q_a} \bigg)^T \Big[ F\big(\bar q(t;q_a,p_b), \bar p(t; q_a,p_b)\big)\,dt + \sum_{i=1}^m f_i\big(\bar q(t;q_a,p_b), \bar p(t; q_a,p_b)\big)\circ dW^i(t) \Big], \nonumber \\
F^+(q_a,p_b) &= \int_{t_a}^{t_b} \bigg( \frac{\partial \bar q(t;q_a,p_b)}{\partial p_b} \bigg)^T \Big[ F\big(\bar q(t;q_a,p_b), \bar p(t; q_a,p_b)\big)\,dt + \sum_{i=1}^m f_i\big(\bar q(t;q_a,p_b), \bar p(t; q_a,p_b)\big)\circ dW^i(t) \Big].
\end{align}

\noindent
Below we prove that the functions $S$ and $F^\pm$ generate\footnote{A generating function for the transformation $(q_a,p_a)\longrightarrow (q_b,p_b)$ is a function of one of the variables $(q_a,p_a)$ and one of the variables $(q_b,p_b)$. Therefore, there are four basic types of generating functions: $S=S_1(q_a, q_b)$, $S=S_2(q_a, p_b)$, $S=S_3(p_a, q_b)$, and $S=S_4(p_a, p_b)$. In this work we use the type-II generating function $S=S_2(q_a, p_b)$.} the stochastic flow $F_{t_b,t_a}$. 

\begin{thm}
\label{thm:S and F generate the Hamiltonian flow}
The function $S(q_a,p_b)$ is a type-II stochastic generating function and the functions $F^\pm(q_a,p_b)$ are type-II stochastic exact discrete forces for the stochastic mapping $F_{t_b,t_a}$, that is, $F_{t_b,t_a}:(q_a,p_a)\longrightarrow (q_b,p_b)$ is implicitly given by the equations

\begin{equation}
\label{eq:Equations generating the flow of the Hamiltonian system}
q_b = D_2 S(q_a,p_b)-F^+(q_a,p_b), \qquad\qquad p_a = D_1 S(q_a,p_b)-F^-(q_a,p_b),
\end{equation}

\noindent
where the derivatives are understood in the mean-square sense.
\end{thm}

\begin{proof}
Under appropriate regularity assumptions on the Hamiltonians and forces (e.g., (H1)-(H3)), the solutions $\bar q(t; q_a, p_b)$ and $\bar p(t; q_a, p_b)$ are mean-square differentiable with respect to the parameters $q_a$ and $p_b$, and the partial derivatives are semimartingales (see \cite{ArnoldSDE}). We calculate the derivative of $S$ as

\begin{align}
\label{eq:Calculate derivative of S}
\frac{\partial S}{\partial q_a}(q_a,p_b) &= \bigg( \frac{\partial \bar q(t_b)}{\partial q_a} \bigg)^T p_b - \int_{t_a}^{t_b} \bigg( \frac{\partial \bar p(t)}{\partial q_a} \bigg)^T \circ d_t \bar q(t)  - \int_{t_a}^{t_b} d_t \bigg( \frac{\partial \bar q(t)}{\partial q_a} \bigg)^T \circ \bar p(t) \nonumber \\
&\phantom{=} + \int_{t_a}^{t_b} \bigg[ \bigg(\frac{\partial \bar q(t)}{\partial q_a} \bigg)^T \frac{\partial H}{\partial q} \big(\bar q(t), \bar p(t) \big) + \bigg(\frac{\partial \bar p(t)}{\partial q_a}\bigg)^T \frac{\partial H}{\partial p} \big(\bar q(t), \bar p(t) \big) \bigg]\,dt \nonumber \\
&\phantom{=} + \sum_{i=1}^m \int_{t_a}^{t_b} \bigg[ \bigg(\frac{\partial \bar q(t)}{\partial q_a}\bigg)^T \frac{\partial h_i}{\partial q} \big(\bar q(t), \bar p(t) \big) + \bigg(\frac{\partial \bar p(t)}{\partial q_a}\bigg)^T \frac{\partial h_i}{\partial p} \big(\bar q(t), \bar p(t) \big) \bigg]\circ dW^i(t),
\end{align}

\noindent
where for notational convenience we have omitted writing $q_a$ and $p_b$ explicitly as arguments of $\bar q(t)$ and $\bar p(t)$. Applying the integration by parts formula for semimartingales (see \cite{ProtterStochastic}), we find

\begin{equation}
\label{eq:Integration by parts for semimartingales 2}
\int_{t_a}^{t_b} d_t \bigg( \frac{\partial \bar q(t)}{\partial q_a} \bigg)^T \circ \bar p(t) = \bigg( \frac{\partial \bar q(t_b)}{\partial q_a} \bigg)^T p_b - \bar p(t_a) - \int_{t_a}^{t_b} \bigg( \frac{\partial \bar q(t)}{\partial q_a} \bigg)^T \circ d_t \bar p(t), 
\end{equation}

\noindent
where the left-hand side integral is understood as a column vector with the components given by

\begin{equation}
\label{eq:Components of the Stratonovich integral}
\sum_{j=1}^N\int_{t_a}^{t_b} \bar p^j(t) \circ d_t \frac{\partial \bar q^j(t)}{\partial q^i_a},
\end{equation}

\noindent
for each $i=1,\ldots,N$. Substituting and rearranging terms, we obtain

\begin{align}
\label{eq:Calculate derivative of S continued}
\frac{\partial S}{\partial q_a}(q_a,p_b) = \bar p(t_a) &+ \int_{t_a}^{t_b} \bigg(\frac{\partial \bar q(t)}{\partial q_a}\bigg)^T \bigg[\circ d_t \bar p +  \frac{\partial H}{\partial q} \big(\bar q(t), \bar p(t) \big)\,dt + \sum_{i=1}^m \frac{\partial h_i}{\partial q} \big(\bar q(t), \bar p(t) \big)\circ dW^i(t) \bigg] \nonumber \\
&+ \int_{t_a}^{t_b} \bigg( \frac{\partial \bar p(t)}{\partial q_a} \bigg)^T \bigg[\circ d_t \bar q -  \frac{\partial H}{\partial p} \big(\bar q(t), \bar p(t) \big)\,dt - \sum_{i=1}^m \frac{\partial h_i}{\partial p} \big(\bar q(t), \bar p(t) \big)\circ dW^i(t) \bigg] 
\nonumber\\&= \bar p(t_a) + \int_{t_a}^{t_b} \bigg( \frac{\partial \bar q(t)}{\partial q_a} \bigg)^T \Big[ F\big(\bar q(t), \bar p(t)\big)\,dt + \sum_{i=1}^m f_i\big(\bar q(t), \bar p(t)\big)\circ dW^i(t) \Big], \nonumber\\
&= \bar p(t_a) + F^-(q_a,p_b),
\end{align}

\noindent
since $(\bar q(t), \bar p(t))$ is a solution of \eqref{eq: Stochastic dissipative Hamiltonian system}. After performing similar manipulations for $\partial S / \partial p_b (q_a,p_b)$, together we obtain the result

\begin{equation}
\label{eq:Equations generating the flow of the Hamiltonian system in the proof}
\bar q(t_b) = D_2 S(q_a,p_b)-F^+(q_a,p_b), \qquad\qquad \bar p(t_a) = D_1 S(q_a,p_b)-F^-(q_a,p_b).
\end{equation}

\noindent
By definition of the flow, then $F_{t_b,t_a}(q_a, \bar p(t_a)) = (\bar q(t_b), p_b)$.

\end{proof}

\subsection{Noether's theorem for stochastic systems with forcing}
\label{sec:Stochastic forced Noether's theorem}

Let a Lie group $G$ act on $Q$ by the left action $\Phi:G \times Q \longrightarrow Q$. The Lie group $G$ then acts on $TQ$ and $T^*Q$ by the tangent $\Phi^{TQ}:G \times TQ \longrightarrow TQ$ and cotangent $\Phi^{T^*Q}:G \times T^*Q \longrightarrow T^*Q$ lift actions, respectively, given in coordinates by the formulas (see \cite{HolmGMS}, \cite{MarsdenRatiuSymmetry})

\begin{align}
\label{eq:Tangent and cotangent lift actions}
\Phi^{TQ}_g(q,\dot q) & \equiv \Phi^{TQ}\big(g,(q,\dot q)\big) =\bigg( \Phi^i_g(q),\frac{\partial \Phi^i_g}{\partial q^j}(q) \dot q^j \bigg), \nonumber \\
\Phi^{T^*Q}_g(q,p) & \equiv \Phi^{T^*Q}\big(g,(q,p)\big) =\bigg( \Phi^i_g(q),p_j \frac{\partial \Phi^j_{g^{-1}}}{\partial q^i}\big(\Phi_g(q)\big) \bigg),
\end{align}

\noindent
where $i,j=1,\ldots,N$ and summation is implied over repeated indices. Let $\mathfrak{g}$ denote the Lie algebra of $G$ and $\exp: \mathfrak{g} \longrightarrow G$ the exponential map (see \cite{HolmGMS}, \cite{MarsdenRatiuSymmetry}). Each element $\xi \in \mathfrak{g}$ defines the infinitesimal generators $\xi_Q$, $\xi_{TQ}$, and $\xi_{T^*Q}$, which are vector fields on $Q$, $TQ$, and $T^*Q$, respectively, given by

\begin{align}
\label{eq:Infinitesimal generators}
\xi_Q(q) = \frac{d}{d \lambda} \bigg|_{\lambda =0} \Phi_{\exp \lambda \xi}(q), \qquad \xi_{TQ}(q, \dot q) = \frac{d}{d \lambda} \bigg|_{\lambda =0} \Phi^{TQ}_{\exp \lambda \xi}(q,\dot q), \qquad \xi_{T^*Q}(q, p) = \frac{d}{d \lambda} \bigg|_{\lambda =0} \Phi^{T^*Q}_{\exp \lambda \xi}(q,p).
\end{align}

\noindent
The momentum map $J: T^*Q \longrightarrow \mathfrak{g}^*$ associated with the action $\Phi^{T^*Q}$ is defined as the mapping such that for all $\xi \in \mathfrak{g}$ the function $J_\xi: T^*Q \ni (q,p) \longrightarrow \langle J(q,p),\xi \rangle \in \mathbb{R}$ is the Hamiltonian for the infinitesimal generator $\xi_{T^*Q}$, i.e.,

\begin{align}
\label{eq:Momentum map definition}
\xi^q_{T^*Q} = \frac{\partial J_\xi}{\partial p}, \qquad \xi^p_{T^*Q} = -\frac{\partial J_\xi}{\partial q},
\end{align}

\noindent
where $\xi_{T^*Q}(q,p) = \big(q,p,\xi^q_{T^*Q}(q,p),\xi^p_{T^*Q}(q,p)\big)$. The momentum map $J$ can be explicitly expressed as (see \cite{HolmGMS}, \cite{MarsdenRatiuSymmetry})

\begin{align}
\label{eq:Momentum map formula}
J_\xi(q,p) = p\cdot \xi_Q(q).
\end{align}

Noether's theorem for deterministic Hamiltonian systems relates symmetries of the Hamiltonian to quantities preserved by the flow of the system (see \cite{HolmGMS}, \cite{MarsdenRatiuSymmetry}). When the Hamiltonian system is subject to external forces that are orthogonal to the infinitesimal generators of the symmetry group, then the corresponding momentum maps are still conserved (see \cite{MarsdenWestVarInt}). It turns out that this result carries over to the stochastic case, as well. A stochastic version of Noether's theorem for systems without forcing was proved in \cite{Bismut}, \cite{HolmTyranowskiGalerkin}, and \cite{LaCa-Or2008}. Below we state and provide a proof of Noether's theorem for stochastic forced Hamiltonian systems.

\begin{thm}[{\bf Noether's theorem for stochastic systems with forcing}]
\label{thm:Stochastic forced Noether's theorem}
Suppose that the Hamiltonians $H:T^*Q \longrightarrow \mathbb{R}$ and $h_i:T^*Q \longrightarrow \mathbb{R}$ for $i=1,\ldots,m$ are invariant with respect to the cotangent lift action $\Phi^{T^*Q}:G \times T^*Q \longrightarrow T^*Q$ of the Lie group $G$, that is,

\begin{align}
\label{eq:Invariance of the Hamiltonians}
H \circ \Phi^{T^*Q}_g = H, \qquad \qquad h_i \circ \Phi^{T^*Q}_g = h_i, \qquad \qquad \text{$i=1,\ldots,m$},
\end{align}

\noindent
for all $g \in G$. If the forcing terms are orthogonal to the infinitesimal generators of $G$, that is,

\begin{align}
\label{eq:Orthogonality of the forcing terms}
F(q,p)\cdot \xi_Q(q) = 0, \qquad \qquad f_i(q,p)\cdot \xi_Q(q) = 0, \qquad \qquad \text{$i=1,\ldots,m$},
\end{align}

\noindent
for all $\xi \in \mathfrak{g}$ and $(q,p) \in T^*Q$, then the cotangent lift momentum map $J:T^*Q \longrightarrow \mathfrak{g}^*$ associated with $\Phi^{T^*Q}$ is almost surely preserved along the solutions of the stochastic forced Hamiltonian system~\eqref{eq: Stochastic dissipative Hamiltonian system}.
\end{thm}

\begin{proof}
Equation \eqref{eq:Invariance of the Hamiltonians} implies that the Hamiltonians are infinitesimally invariant with respect to the action of $G$, that is, for all $\xi \in \mathfrak{g}$ we have 

\begin{align}
\label{eq:Infinitesimal invariance of the Hamiltonians}
dH\cdot \xi_{T^*Q} = 0, \qquad \qquad dh\cdot \xi_{T^*Q} = 0,
\end{align}

\noindent
where $dH$ and $dh$ denote differentials with respect to the variables $q$ and $p$. Let $(q(t),p(t))$ be a solution of \eqref{eq: Stochastic dissipative Hamiltonian system} and consider the stochastic process $J_\xi(q(t),p(t))$, where $\xi \in \mathfrak{g}$ is arbitrary. Using the rules of Stratonovich calculus we can calculate the stochastic differential

\begin{align}
\label{eq:Stochastic differential of J}
d_t J_\xi\big(q(t),p(t)\big)& =\frac{\partial J_\xi}{\partial q}(q(t),p(t))\circ d_t q(t) + \frac{\partial J_\xi}{\partial p}(q(t),p(t))\circ d_t p(t) \nonumber \\
                         & =\bigg( -\frac{\partial H}{\partial q} \xi^q_{T^*Q} - \frac{\partial H}{\partial p} \xi^p_{T^*Q} + F\cdot \xi^q_{T^*Q}  \bigg) \,dt +\sum_{i=1}^m \bigg( -\frac{\partial h_i}{\partial q} \xi^q_{T^*Q} - \frac{\partial h_i}{\partial p} \xi^p_{T^*Q} + f_i \cdot \xi^q_{T^*Q}  \bigg)\circ dW^i(t) \nonumber \\
                         & = \big(-dH\cdot \xi_{T^*Q} + F\cdot \xi^q_{T^*Q} \big)\,dt  + \sum_{i=1}^m\big(-dh_i\cdot \xi_{T^*Q} + f_i\cdot \xi^q_{T^*Q}\big)\circ dW^i(t) \nonumber \\
												& = F(q(t),p(t))\cdot \xi_Q (q(t))\,dt  + \sum_{i=1}^m f_i(q(t),p(t))\cdot \xi_Q(q(t))\circ dW^i(t),
\end{align}

\noindent
where we used \eqref{eq: Stochastic dissipative Hamiltonian system}, \eqref{eq:Momentum map definition}, \eqref{eq:Momentum map formula}, and \eqref{eq:Infinitesimal invariance of the Hamiltonians}. Therefore, if \eqref{eq:Orthogonality of the forcing terms} holds, then $J_\xi\big(q(t),p(t)\big) = \text{const}$ almost surely for all $\xi \in \mathfrak{g}$, which completes the proof.\\
\end{proof}

\paragraph{Remark.} When the external forces are not all orthogonal to the infinitesimal generators of the symmetry group, formula \eqref{eq:Stochastic differential of J} provides the rate of change of the momentum map.

\subsection{Conformal symplecticity and phase space volume}
\label{sec:Conformal symplecticity and phase space volume}

The flow $F_{t,t_0}$ for stochastic Hamiltonian systems without forcing almost surely preserves the canonical symplectic two-form 

\begin{equation}
\label{eq:Symplectic form}
\Omega_{T^*Q} = dq \wedge dp = \sum_{i=1}^N dq^i \wedge dp^i, 
\end{equation}

\noindent
that is, $F^*_{t,t_0} \Omega_{T^*Q} = \Omega_{T^*Q}$, where $F^*_{t,t_0}$ denotes the pull-back by the flow $F_{t,t_0}$ (see \cite{MilsteinRepin}, \cite{Bismut}, \cite{LaCa-Or2008}). This property does not hold for the general stochastic forced Hamiltonian system \eqref{eq: Stochastic dissipative Hamiltonian system}. However, for certain choices of the forcing terms, the flow may be \emph{conformally symplectic}, which means that for all $t \geq t_0$ there exists a constant (possibly random) $c_{t,t_0}\in \mathbb{R}$ such that

\begin{equation}
\label{eq:Conformal symplecticity of the Hamiltonian flow}
F^*_{t,t_0} \Omega_{T^*Q} = c_{t,t_0} \, \Omega_{T^*Q}.
\end{equation}

\noindent
Deterministic conformally symplectic systems are considered in \cite{McLachlanConformal}. Conformal symplecticity for the special case of \eqref{eq: Stochastic dissipative Hamiltonian system} with a separable Hamiltonian, an additive noise, and the forcing terms equal to $F(q,p) = -\nu p$ with a real parameter $\nu$, and $f_i(q,p)=0$ for $i=1, \ldots, m$, was considered in \cite{BouRabeeOwhadiBoltzmannGibbs} and \cite{Hong2017}. Below we demonstrate that the property of conformal symplecticity persists for more general cases.

\begin{thm}[{\bf Conformal symplecticity}] 
\label{thm:Conformal symplecticity}
Suppose that $H(q,p)$, $F(q,p)$, and $h_i(q,p)$, $f_i(q,p)$ for $i=1,\ldots,m$ satisfy conditions (H1)-(H3). If the forcing terms have the form

\begin{equation}
\label{eq:Linear forcing terms}
F(q,p) = -\nu_0 p, \qquad\qquad f_i(q,p) = -\nu_i p, \qquad\qquad i=1,\ldots, m,
\end{equation}

\noindent
for real parameters $\nu_i$, then the stochastic flow $F_{t,t_0}$ for \eqref{eq: Stochastic dissipative Hamiltonian system} is almost surely conformally symplectic with the parameter $c_{t,t_0}$ in \eqref{eq:Conformal symplecticity of the Hamiltonian flow} given by

\begin{equation}
\label{eq:Parameter c_t}
c_{t,t_0} = \exp\Big( {-\nu_0(t-t_0) - \sum_{i=1}^{m}\nu_i \big(W^i(t)-W^i(t_0)\big)} \Big)
\end{equation}

\noindent
for all $t\geq t_0$.

\end{thm}

\begin{proof}
For fixed $(q,p)\in T^*Q$, the stochastic process $F_{t,t_0}(q,p)$ satisfies the system \eqref{eq: Stochastic dissipative Hamiltonian system}, which can be written as

\begin{equation}
\label{eq:SDE for the flow}
d_t F_{t,t_0}(q,p) = X\big( F_{t,t_0}(q,p)\big)\,dt + \sum_{i=1}^m Y_i\big( F_{t,t_0}(q,p)\big)\circ dW^i(t),
\end{equation}

\noindent
where $X$ and $Y_i$ are vector fields on $T^*Q$, and are given by, respectively,

\begin{equation}
\label{eq:Vector fields X and Y}
X = \frac{\partial H}{\partial p} \frac{\partial}{\partial q} + \bigg[-\frac{\partial H}{\partial q} + F(q,p) \bigg] \frac{\partial}{\partial p},\qquad\quad Y_i = \frac{\partial h_i}{\partial p} \frac{\partial}{\partial q} + \bigg[-\frac{\partial h_i}{\partial q} + f_i(q,p) \bigg] \frac{\partial}{\partial p}, \qquad\quad i=1,\ldots, m.
\end{equation}

\noindent
Let us calculate the stochastic differential of $F^*_{t,t_0} \Omega_{T^*Q}$. Using the stochastic generalization of the dynamic definition of the Lie derivative (see Theorem~1.2 in \cite{HolmTyranowskiSolitons}), we can write

\begin{equation}
\label{eq:Stochastic differential of the pull-back of the symplectic form}
d_t ( F^*_{t,t_0} \Omega_{T^*Q} ) = F^*_{t,t_0} (\pounds_X \Omega_{T^*Q}) \, dt + \sum_{i=1}^m F^*_{t,t_0} (\pounds_{Y_i} \Omega_{T^*Q})\circ dW^i(t),
\end{equation}

\noindent
where $\pounds_X$ and $\pounds_{Y_i}$ denote the Lie derivatives with respect to the vector fields $X$ and $Y_i$, respectively. Using Cartan's magic formula (see, e.g., \cite{AbrahamMarsdenMTA}) we have that

\begin{equation}
\label{eq:Cartan's magic formula}
\pounds_X \Omega_{T^*Q} = d i_X \Omega_{T^*Q} + i_X d\Omega_{T^*Q} = d i_X \Omega_{T^*Q},
\end{equation}

\noindent
since $d\Omega_{T^*Q}=0$, where $i_X$ denotes the interior product with the vector field $X$. Substituting \eqref{eq:Vector fields X and Y}, \eqref{eq:Linear forcing terms}, and \eqref{eq:Symplectic form}, we obtain

\begin{equation}
\label{eq:Lie derivative with respect to X}
\pounds_X \Omega_{T^*Q} = - \nu_0 \, \Omega_{T^*Q},
\end{equation}

\noindent
since the Hamiltonian function $H$ is $C^2$. In a similar fashion we show that $\pounds_{Y_i} \Omega_{T^*Q} = - \nu_i \Omega_{T^*Q}$. Plugging this in \eqref{eq:Stochastic differential of the pull-back of the symplectic form}, we obtain a stochastic differential equation of the form

\begin{equation}
\label{eq:SDE for symplecticity}
d_t ( F^*_{t,t_0} \Omega_{T^*Q} ) = -\nu_0 (F^*_{t,t_0}\Omega_{T^*Q})\,dt - \sum_{i=1}^m \nu_i (F^*_{t,t_0}\Omega_{T^*Q})\circ dW^i(t).
\end{equation}

\noindent
It is straightforward to verify that the solution of \eqref{eq:SDE for symplecticity} that satisfies the initial condition $F^*_{t_0,t_0} \Omega_{T^*Q} = \Omega_{T^*Q}$ has the form

\begin{equation}
\label{eq:Solution of the conformal symplecticity SDE}
F^*_{t,t_0} \Omega_{T^*Q} = c_{t,t_0} \, \Omega_{T^*Q}
\end{equation}

\noindent
with $c_{t,t_0}$ given by \eqref{eq:Parameter c_t}, which proves the conformal symplecticity of the flow $F^*_{t,t_0}$. It holds almost surely, since the solution of the SDE \eqref{eq:SDE for symplecticity} is pathwise unique (see \cite{ArnoldSDE}, \cite{IkedaWatanabe1989}, \cite{KloedenPlatenSDE}, \cite{Kunita1997}).

\end{proof}

The evolution of stochastic Hamiltonian systems without forcing preserves volumes in phase space, that is, for the standard volume form on $T^*Q$ defined as

\begin{equation}
\label{eq:Volume form}
\mu = dq^1 \wedge \ldots \wedge dq^N \wedge dp^1 \wedge \ldots dp^N
\end{equation}

\noindent
we have that $F^*_{t,t_0} \mu = \mu$. This is a direct consequence of the symplecticity of the flow. Phase space volume preservation does not hold for the general forced system \eqref{eq: Stochastic dissipative Hamiltonian system}, although for certain choices of the forcing terms the flow $F^*_{t,t_0}$ may possess a property similar to \eqref{eq:Conformal symplecticity of the Hamiltonian flow}. Such a property was proved for the special case of \eqref{eq: Stochastic dissipative Hamiltonian system} with a separable Hamiltonian, an additive noise, and the forcing terms equal to $F(q,p) = -\Gamma p$ with a constant $N\times N$ matrix $\Gamma$, and $f_i(q,p)=0$ for $i=1, \ldots, m$ (see \cite{Bismut}, \cite{Hong2017}, \cite{MilsteinErgodic}, \cite{MilsteinMeanSquarePreprint}, \cite{MilsteinRepin2001}, \cite{MilsteinRepin}, \cite{MilsteinQuasiSymplectic}). Below we demonstrate that this property holds also for more general cases.

\begin{thm}[{\bf Phase space volume evolution}] 
\label{thm:Phase space volume evolution}
Suppose that $H(q,p)$, $F(q,p)$, and $h_i(q,p)$, $f_i(q,p)$ for $i=1,\ldots,m$ satisfy conditions (H1)-(H3). If the forcing terms have the form

\begin{equation}
\label{eq:Linear forcing terms for volume form evolution}
F(q,p) = -\Gamma_0 p, \qquad\qquad f_i(q,p) = -\Gamma_i p, \qquad\qquad i=1,\ldots, m,
\end{equation}

\noindent
for constant $N \times N$ matrices $\Gamma_i$, then the phase space volume form $\mu$ for $t\geq t_0$ almost surely evolves according to the formula

\begin{equation}
\label{eq:Volume form evolution}
F^*_{t,t_0} \mu = b_{t,t_0} \, \mu,
\end{equation}

\noindent
where

\begin{equation}
\label{eq:Parameter b_t}
b_{t,t_0} = \exp\Big( {-\tr \Gamma_0 \cdot (t-t_0) - \sum_{i=1}^{m}\tr \Gamma_i \cdot \big(W^i(t)-W^i(t_0)\big)} \Big),
\end{equation}

\noindent
and $F_{t,t_0}$ is the stochastic flow for \eqref{eq: Stochastic dissipative Hamiltonian system}.
\end{thm}

\begin{proof}
This theorem is a special case of, e.g., Lemma~4.3.1 in \cite{Kunita1997}. We briefly outline an alternative geometric proof, analogous to the proof of Theorem~\ref{thm:Conformal symplecticity}. Similar to \eqref{eq:Stochastic differential of the pull-back of the symplectic form}, we can write

\begin{equation}
\label{eq:Stochastic differential of the pull-back of the volume form}
d_t ( F^*_{t,t_0} \mu ) = F^*_{t,t_0} (\pounds_X \mu) \, dt + \sum_{i=1}^m F^*_{t,t_0} (\pounds_{Y_i} \mu)\circ dW^i(t).
\end{equation}

\noindent
Using the property of the divergence operator (see, e.g., \cite{AbrahamMarsdenMTA}), we calculate

\begin{equation}
\label{eq:Lie derivative of the volume form with respect to X}
\pounds_X \mu = (\divergence X)\cdot \, \mu = -(\tr \Gamma_0)\cdot \mu,
\end{equation}

\noindent
where we have used \eqref{eq:Vector fields X and Y} and \eqref{eq:Linear forcing terms for volume form evolution}, and the fact that the Hamiltonian function $H$ is~$C^2$. In a similar way we show that $\pounds_{Y_i} \mu = -(\tr \Gamma_i)\cdot \mu$. Therefore, we obtain the SDE of the form

\begin{equation}
\label{eq:SDE for the evolution of the volume form}
d_t ( F^*_{t,t_0} \mu ) =  -(\tr \Gamma_0)\cdot (F^*_{t,t_0} \mu) \, dt - \sum_{i=1}^m (\tr \Gamma_i)\cdot (F^*_{t,t_0} \mu)\circ dW^i(t).
\end{equation}

\noindent
It is straightforward to verify that the solution that satisfies the initial condition $F^*_{t_0,t_0} \mu = \mu$ is given by \eqref{eq:Volume form evolution} with $b_{t,t_0}$ as in \eqref{eq:Parameter b_t}. The formula \eqref{eq:Volume form evolution} holds almost surely, because the solution of the SDE is pathwise unique (see \cite{ArnoldSDE}, \cite{IkedaWatanabe1989}, \cite{KloedenPlatenSDE}, \cite{Kunita1997}).

\end{proof}

\section{Stochastic Lagrange-d'Alembert variational integrators}
\label{sec:Stochastic Lagrange-d'Alembert variational integrators}

Suppose we would like to solve \eqref{eq: Stochastic dissipative Hamiltonian system} on the interval $[0,T]$ with the initial conditions $(q_0,p_0)\in T^*Q$. Consider the discrete set of times $t_k = k\cdot\Delta t$ for $k=0,1,\ldots,K$, where $\Delta t = T/K$ is the time step. In order to determine the discrete curve $\{(q_k,p_k)\}_{k=0,\ldots,K}$ that approximates the exact solution of \eqref{eq: Stochastic dissipative Hamiltonian system} at times $t_k$ we need to construct an approximation of the exact stochastic flow $F_{t_{k+1},t_k}$ on each interval $[t_k,t_{k+1}]$, so that $(q_{k+1},p_{k+1}) \approx F_{t_{k+1},t_k}(q_k,p_k)$. A numerical method respecting the underlying Lagrange-d'Alembert principle \eqref{eq:Stochastic Lagrange-d'Alembert Principle} can be constructed by approximating the generating function and forcing terms in \eqref{eq:Equations generating the flow of the Hamiltonian system}. Let the discrete Hamiltonian function $H^+_d(q_a,p_b; t_a, t_b)$ be an approximation of the generating function \eqref{eq:Definition of the generating function}, and let the discrete forces $F^\pm_d(q_a,p_b; t_a, t_b)$ be approximations of the forcing terms \eqref{eq:Definition of the exact forcing}. The approximate numerical flow $F^+_{t_{k+1},t_k}:(q_k,p_k)\longrightarrow (q_{k+1},p_{k+1})$ is now generated as in \eqref{eq:Equations generating the flow of the Hamiltonian system in the proof}:

\begin{align}
\label{eq:Equations generating the numerical flow of the Hamiltonian system}
q_{k+1} &= D_2 H^+_d(q_k,p_{k+1}; t_k, t_{k+1})-F^+_d(q_k,p_{k+1}; t_k, t_{k+1}), \nonumber \\
p_k &= D_1 H^+_d(q_k,p_{k+1}; t_k, t_{k+1})-F^-_d(q_k,p_{k+1}; t_k, t_{k+1}).
\end{align}

\noindent
If there is no risk of confusion, we will omit writing the time arguments of $H^+_d$ and $F^\pm_d$. We will refer to the scheme \eqref{eq:Equations generating the numerical flow of the Hamiltonian system} as a stochastic Lagrange-d'Alembert variational integrator.

\subsection{Discrete stochastic Lagrange-d'Alembert principle}
\label{sec:Discrete stochastic Lagrange-d'Alembert principle}

The advantage of the integrator \eqref{eq:Equations generating the numerical flow of the Hamiltonian system} is that it follows from a discrete version of the stochastic Lagrange-d'Alembert principle \eqref{eq:Stochastic Lagrange-d'Alembert Principle}. The discrete Lagrange-d'Alembert principle for deterministic Lagrangian systems was proposed in \cite{KaneMarsden2000}; see also \cite{MarsdenWestVarInt}. Below we generalize it to the stochastic case in the setting of Hamiltonian systems defined on the phase space $T^*Q$. Define the discrete random curve space $C_d$ as

\begin{equation}
\label{eq:Discrete random curve space}
C_d = \Big\{ \big\{ (q_k,p_k) \big\}_{k=0,\ldots,K} \,\big |\, (q_k,p_k):\Omega \longrightarrow T^*Q \text{ are random variables for each $k=0,\ldots,K$} \Big\}.
\end{equation}

\noindent
On that space define the discrete action functional, $\mathcal{B}_d:\Omega \times C_d \longrightarrow \mathbb{R}$,

\begin{equation}
\label{eq:Discrete action functional}
\mathcal{B}_d \big[ \{ (q_k,p_k) \}_{k=0,\ldots,K} \big] = p_K q_K - \sum_{k=0}^{K-1} \big(p_{k+1} q_{k+1} - H^+_d(q_k,p_{k+1}; t_k, t_{k+1}) \big).
\end{equation}

\noindent
Note that $\mathcal{B}_d$ is an approximation of the stochastic action functional \eqref{eq:Stochastic action functional} on the interval $[0,T]$.

\begin{thm}[{\bf Discrete stochastic Lagrange-d'Alembert Principle in Phase Space}]
\label{thm:Discrete stochastic Lagrange-d'Alembert Principle in Phase Space}
Suppose the discrete Hamiltonian $H^+_d$ is almost surely continuously differentiable, and the discrete forces $F^\pm_d$ are almost surely continuous with respect to their arguments. The discrete random curve $\{ (q_k,p_k) \}_{k=0,\ldots,K}$ satisfies the set of equations

\begin{align}
\label{eq:Equations equivalent to the discrete Lagrange-d'Alembert principle}
q_{k} &= D_2 H^+_d(q_{k-1},p_{k}; t_{k-1}, t_{k})-F^+_d(q_{k-1},p_{k}; t_{k-1}, t_{k}), \nonumber \\
p_k &= D_1 H^+_d(q_k,p_{k+1}; t_k, t_{k+1})-F^-_d(q_k,p_{k+1}; t_k, t_{k+1}),
\end{align}

\noindent
almost surely for $k=1,\ldots,K-1$, if and only if it almost surely satisfies the variational equation

\begin{equation}
\label{eq:Discrete stochastic Lagrange-d'Alembert principle}
\delta \mathcal{B}_d - \sum_{k=0}^{K-1} \big( F^-_d(q_k,p_{k+1}; t_k, t_{k+1}) \delta q_k + F^+_d(q_k,p_{k+1}; t_k, t_{k+1}) \delta p_{k+1} \big) = 0
\end{equation}

\noindent
for all variations $\{ (\delta q_k, \delta p_k) \}_{k=0,\ldots,K}$ such that $\delta q_0 =0$ and $\delta p_K =0$ almost surely.
\end{thm}

\begin{proof}
Consider an arbitrary random curve $\{ (q_k,p_k) \}_{k=0,\ldots,K}$. Let us calculate the variation $\delta \mathcal{B}_d$ corresponding to the arbitrary variation $\{ (\delta q_k, \delta p_k) \}_{k=0,\ldots,K}$ with $\delta q_0 =0$ and $\delta p_K =0$ (almost surely). We have

\begin{align}
\label{eq:Calculating delta Bd}
\delta \mathcal{B}_d &= p_K \delta q_K - \sum_{k=0}^{K-1} \big( \delta p_{k+1} q_{k+1} + p_{k+1} \delta q_{k+1} - D_1 H^+_d(q_k,p_{k+1}; t_k, t_{k+1}) \delta q_k - D_2 H^+_d(q_k,p_{k+1}; t_k, t_{k+1}) \delta p_{k+1} \big) \nonumber \\
&= - \sum_{k=0}^{K-1} \big( q_{k+1} - D_2 H^+_d(q_k,p_{k+1}; t_k, t_{k+1}) \big) \delta p_{k+1} - \sum_{k=0}^{K-1} \big( p_k - D_1 H^+_d(q_k,p_{k+1}; t_k, t_{k+1}) \big) \delta q_k,
\end{align}

\noindent
where in the second equality we have shifted the summation index in the $\delta q_{k+1}$ term and used the fact that $\delta q_0=0$. It is now straightforward to see that if the set of equations \eqref{eq:Equations equivalent to the discrete Lagrange-d'Alembert principle} is satisfied, then the variational equation \eqref{eq:Discrete stochastic Lagrange-d'Alembert principle} holds almost surely. Conversely, if the variational equation \eqref{eq:Discrete stochastic Lagrange-d'Alembert principle} holds for all variations $\{ (\delta q_k, \delta p_k) \}_{k=0,\ldots,K}$ with $\delta q_0 =0$ and $\delta p_K =0$, then the set of equations \eqref{eq:Equations equivalent to the discrete Lagrange-d'Alembert principle} has to be satisfied almost surely.

\end{proof}

\subsection{Discrete Noether's theorem for stochastic systems with forcing}
\label{sec:Discrete stochastic forced Noether's theorem}

Another advantage of the integrator \eqref{eq:Equations generating the numerical flow of the Hamiltonian system} is that one can prove a discrete counterpart of Theorem~\ref{thm:Stochastic forced Noether's theorem}. If the discrete system inherits the symmetries of the continuous problem, then the evolution of the momentum maps will be accurately captured by the numerical solution. Discrete Noether's theorem for systems described by a type-II generating function was first proved for deterministic systems in \cite{LeokZhang}, and later generalized to the stochastic case in \cite{HolmTyranowskiGalerkin}. Discrete Noether's theorem for deterministic Lagrangian systems with forcing was first proposed in \cite{MarsdenWestVarInt}. Below we combine these ideas and formulate a version of discrete Noether's theorem applicable to discrete systems described by \eqref{eq:Equations generating the numerical flow of the Hamiltonian system}. Let $R_d: \Omega \times Q \times T^*Q \longrightarrow \mathbb{R}$ be the generalized discrete stochastic Lagrangian defined as

\begin{equation}
\label{eq:Generalized discrete Lagrangian}
R_d(q_k, q_{k+1}, p_{k+1}) = p_{k+1} q_{k+1} - H^+_d(q_k, p_{k+1}).
\end{equation} 

\noindent
Consider the action of the Lie group $G$ on $Q \times T^*Q$ given by

\begin{equation}
\label{eq:Generalized action of the Lie group}
\Phi^{Q \times T^*Q}_g(q_k, q_{k+1}, p_{k+1}) = \big( \Phi_g(q_k), \Phi^{T^*Q}_g(q_{k+1}, p_{k+1})\big).
\end{equation}

\noindent
For any $\xi \in \mathfrak{g}$ the corresponding infinitesimal generator on $Q \times T^*Q$ is then given by

\begin{equation}
\xi_{Q \times T^*Q}(q_k, q_{k+1}, p_{k+1}) = \big(\xi_Q(q_k), \xi_{T^*Q}(q_{k+1},p_{k+1})\big) = \big(\xi_Q(q_k), \xi^q_{T^*Q}(q_{k+1},p_{k+1}), \xi^p_{T^*Q}(q_{k+1},p_{k+1})\big).
\end{equation}

\begin{thm}[{\bf Discrete Noether's theorem for stochastic systems with forcing}]
\label{thm:Discrete stochastic forced Noether's theorem}
Suppose the generalized discrete stochastic Lagrangian $R_d: \Omega \times Q \times T^*Q \longrightarrow \mathbb{R}$ is invariant under the action of the Lie group $G$, that is,

\begin{equation}
\label{eq:Invariance of the generalized Lagrangian}
R_d \Big( \Phi^{Q \times T^*Q}_g(q_k, q_{k+1},p_{k+1}) \Big) = R_d(q_k, q_{k+1}, p_{k+1}), \quad \qquad \text{for all $g\in G$}.
\end{equation}

\noindent
If the discrete forces $F^\pm_d$ satisfy the condition

\begin{equation}
\label{eq:Condition on the discrete forces}
F^-_d(q_k, p_{k+1})\cdot \xi_Q(q_k) + F^+_d(q_k, p_{k+1})\cdot \xi^p_{T^*Q}(q_{k+1},p_{k+1}) = 0
\end{equation}

\noindent
for all $(q_k, q_{k+1},p_{k+1}) \in Q \times T^*Q$, then the cotangent lift momentum map $J$ associated with $\Phi^{T^*Q}$ is almost surely preserved along the solutions of the discrete equations \eqref{eq:Equations generating the numerical flow of the Hamiltonian system}, i.e., a.s. $J(q_{k+1},p_{k+1}) = J(q_k,p_k)$.
\end{thm}

\begin{proof}
Since the generalized discrete Lagrangian $R_d$ is invariant with respect to the actions of $G$, for an arbitrary $\xi \in \mathfrak{g}$ we have

\begin{align}
\label{eq:Derivative of the generalized Lagrangian along the infinitesimal generator 1}
0 &= \frac{d}{d\lambda}\bigg|_{\lambda=0} R_d\Big( \Phi^{Q \times T^*Q}_{\exp \lambda \xi}(q_k, q_{k+1},p_{k+1}) \Big) = dR_d\cdot \xi_{Q \times T^*Q}(q_k, q_{k+1}, p_{k+1}) \nonumber \\
&= -D_1 H^+_d(q_k,p_{k+1})\cdot \xi_Q(q_k) + p_{k+1}\cdot \xi_Q(q_{k+1})  +\big(q_{k+1}- D_1 H^+_d(q_k,p_{k+1})\big)\cdot \xi^p_{T^*Q}(q_{k+1},p_{k+1}),
\end{align}

\noindent
where we have used the fact that $\xi^q_{T^*Q}(q_{k+1},p_{k+1}) = \xi_Q(q_{k+1})$. Assume that $q_k$, $q_{k+1}$, and $p_{k+1}$ satisfy the discrete evolution equation \eqref{eq:Equations generating the numerical flow of the Hamiltonian system}. By substituting \eqref{eq:Equations generating the numerical flow of the Hamiltonian system} in \eqref{eq:Derivative of the generalized Lagrangian along the infinitesimal generator 1}, we obtain

\begin{align}
\label{eq:Derivative of the generalized Lagrangian along the infinitesimal generator 2}
0 &= (-p_k-F^-_d(q_k,p_{k+1}))\cdot \xi_Q(q_k) + p_{k+1}\cdot \xi_Q(q_{k+1})  -F^+_d(q_k,p_{k+1}) \cdot \xi^p_{T^*Q}(q_{k+1},p_{k+1}).
\end{align}

\noindent
This can be rewritten as

\begin{align}
\label{eq:Discrete evolution of the momentum map}
J_\xi(q_{k+1},p_{k+1}) - J_\xi(q_k,p_k) = F^-_d(q_k,p_{k+1})\cdot \xi_Q(q_k) + F^+_d(q_k,p_{k+1}) \cdot \xi^p_{T^*Q}(q_{k+1},p_{k+1}),
\end{align}

\noindent
where we have used the definition of the cotangent lift momentum map \eqref{eq:Momentum map formula}. If the condition \eqref{eq:Condition on the discrete forces} holds, then we have $J_\xi(q_{k+1},p_{k+1}) = J_\xi(q_k,p_k)$. The result holds almost surely, because equation \eqref{eq:Equations generating the numerical flow of the Hamiltonian system} is satisfied almost surely.
\end{proof}

\paragraph{Remark. }When the discrete forces do not satisfy the condition \eqref{eq:Condition on the discrete forces}, equation \eqref{eq:Discrete evolution of the momentum map} provides the rate of change of the momentum map, which mimicks formula \eqref{eq:Stochastic differential of J} in the continuous case.

\subsection{Mean-square Lagrange-d'Alembert partitioned Runge-Kutta methods}
\label{sec:Mean-square Lagrange-d'Alembert partitioned Runge-Kutta methods}

\subsubsection{Construction}
\label{sec:Construction}

Partitioned Runge-Kutta methods for deterministic forced Hamiltonian systems have been proposed in \cite{JayNonholonomic} and \cite{MarsdenWestVarInt}. A general class of stochastic mean-square Runge-Kutta methods for Stratonovich ordinary differential equations was introduced and analyzed in \cite{Burrage1996}, \cite{Burrage1998}, and \cite{Burrage2000}. These ideas were later used by Ma \& Ding \& Ding \cite{MaDing2012} and Ma \& Ding \cite{MaDing2015} to construct symplectic Runge-Kutta methods for stochastic Hamiltonian systems without forcing; see also \cite{HolmTyranowskiGalerkin}. Below we combine these ideas and introduce mean-square Lagrange-d'Alembert partitioned Runge-Kutta methods for stochastic forced Hamiltonian systems of the form \eqref{eq: Stochastic dissipative Hamiltonian system}.

\begin{define} An $s$-stage mean-square Lagrange-d'Alembert partitioned Runge-Kutta method for the system \eqref{eq: Stochastic dissipative Hamiltonian system} is given by

\begin{subequations}
\label{eq:SPRK for stochastic forced Hamiltonian systems}
\begin{align}
\label{eq:SPRK for stochastic forced Hamiltonian systems 1}
Q_i &= q_k + \Delta t \sum_{j=1}^s a_{ij} \frac{\partial H}{\partial p}(Q_j,P_j) + \sum_{r=1}^m \Delta W^r \sum_{j=1}^s b_{ij} \frac{\partial h_r}{\partial p}(Q_j,P_j), \quad \qquad i=1,\ldots,s,  \\
\label{eq:SPRK for stochastic forced Hamiltonian systems 2}
P_i &= p_k - \Delta t \sum_{j=1}^s \bar a_{ij} \frac{\partial H}{\partial q}(Q_j,P_j) - \sum_{r=1}^m \Delta W^r \sum_{j=1}^s  \bar b_{ij} \frac{\partial h_r}{\partial q}(Q_j,P_j) \nonumber \\
&\phantom{= p_k}+ \Delta t \sum_{j=1}^s \hat a_{ij} F(Q_j,P_j) + \sum_{r=1}^m \Delta W^r \sum_{j=1}^s \hat b_{ij} f_r(Q_j,P_j), \quad \qquad \qquad i=1,\ldots,s, \\
\label{eq:SPRK for stochastic forced Hamiltonian systems 3}
q_{k+1} &= q_k + \Delta t \sum_{i=1}^s \alpha_i \frac{\partial H}{\partial p}(Q_i,P_i) + \sum_{r=1}^m \Delta W^r \sum_{i=1}^s \beta_i \frac{\partial h_r}{\partial p}(Q_i,P_i),\\
\label{eq:SPRK for stochastic forced Hamiltonian systems 4}
p_{k+1} &= p_k - \Delta t \sum_{i=1}^s \alpha_i \frac{\partial H}{\partial q}(Q_i,P_i) - \sum_{r=1}^m \Delta W^r \sum_{i=1}^s \beta_i \frac{\partial h_r}{\partial q}(Q_i,P_i) \nonumber \\
&\phantom{= p_k}+\Delta t \sum_{i=1}^s \hat \alpha_i F(Q_i,P_i) + \sum_{r=1}^m \Delta W^r \sum_{i=1}^s \hat \beta_i f_r(Q_i,P_i),
\end{align}
\end{subequations}

\noindent
where $\Delta t$ is the time step, $\Delta W = (\Delta W^1, \ldots, \Delta W^m)$ are the increments of the Wiener process, $Q_i$ and $P_i$ for $i=1,\ldots,s$ are the position and momentum internal stages, respectively, and the coefficients of the method $a_{ij}$, $\bar a_{ij}$, $\hat a_{ij}$, $b_{ij}$, $\bar b_{ij}$, $\hat b_{ij}$, $\alpha_i$, $\hat \alpha_i$, $\beta_i$, and $\hat \beta_i$ satisfy the conditions

\begin{subequations}
\label{eq:Symplectic conditions for SPRK}
\begin{align}
\label{eq:Symplectic conditions for SPRK 1}
\alpha_i \bar a_{ij} + \alpha_j a_{ji} &= \alpha_i \alpha_j, \\
\label{eq:Symplectic conditions for SPRK 2}
\beta_i \bar b_{ij} + \beta_j b_{ji} &= \beta_i \beta_j, \\
\label{eq:Symplectic conditions for SPRK 3}
\beta_i \bar a_{ij} + \alpha_j  b_{ji} &= \beta_i \alpha_j, \\
\label{eq:Symplectic conditions for SPRK 4}
\alpha_i \bar b_{ij} + \beta_j a_{ji} &= \alpha_i \beta_j, \\
\label{eq:Symplectic conditions for SPRK 5}
\alpha_i \hat a_{ij} + \hat \alpha_j a_{ji} &= \alpha_i \hat \alpha_j, \\
\label{eq:Symplectic conditions for SPRK 6}
\alpha_i \hat b_{ij} + \hat \beta_j a_{ji} &= \alpha_i \hat \beta_j, \\
\label{eq:Symplectic conditions for SPRK 7}
\beta_i \hat a_{ij} + \hat \alpha_j  b_{ji} &= \beta_i \hat \alpha_j, \\
\label{eq:Symplectic conditions for SPRK 8}
\beta_i \hat b_{ij} + \hat \beta_j b_{ji} &= \beta_i \hat \beta_j,
\end{align}
\end{subequations}

\noindent
for $i,j=1,\ldots,s$.
\end{define}

The partitioned Runge-Kutta method \eqref{eq:SPRK for stochastic forced Hamiltonian systems} can be represented by the tableau

\begin{equation}
\label{eq: SPRK tableau}
\begin{array}{c|c|c|c|c|c|c}
& a & \bar a & \hat a & b & \bar b & \hat b \\
\hline
& \alpha^T & \alpha^T & \hat \alpha^T & \beta^T & \beta^T & \hat \beta^T
\end{array},
\end{equation}

\noindent
where $a=(a_{ij})_{i,j=1\ldots s}$, $\alpha = (\alpha_i)_{i=1\ldots s}$, etc. The set of equations \eqref{eq:SPRK for stochastic forced Hamiltonian systems} forms a one-step numerical scheme. Knowing $q_k$ and $p_k$ at time $t_k$, one can solve Equations \eqref{eq:SPRK for stochastic forced Hamiltonian systems 1}-\eqref{eq:SPRK for stochastic forced Hamiltonian systems 2} for the internal stages $Q_i$ and $P_i$, and then use \eqref{eq:SPRK for stochastic forced Hamiltonian systems 3}-\eqref{eq:SPRK for stochastic forced Hamiltonian systems 4} to determine $q_{k+1}$ and $p_{k+1}$ at time $t_{k+1}$. If given $q_k$ and $p_{k+1}$ instead, one can also solve \eqref{eq:SPRK for stochastic forced Hamiltonian systems} for the remaining variables $Q_i$, $P_i$, $q_{k+1}$ and $p_k$. Note that since we have only used $\Delta W^r = \int_{t_k}^{t_{k+1}}dW^r(t)$ in \eqref{eq:SPRK for stochastic forced Hamiltonian systems}, we can in general expect mean-square convergence of order 1.0 at most. To obtain mean-square convergence of higher order we would also need to include higher-order multiple Stratonovich integrals, e.g., to achieve convergence of order 1.5 we would need to include terms involving $\Delta Z^r = \int_{t_k}^{t_{k+1}}\int_{t_k}^{t}dW^r(\xi)\,dt$ (see \cite{Burrage2000}, \cite{MilsteinRepin2001}, \cite{MilsteinRepin}). Below we prove that the Runge-Kutta method \eqref{eq:SPRK for stochastic forced Hamiltonian systems} with the conditions \eqref{eq:Symplectic conditions for SPRK} is indeed a stochastic Lagrange-d'Alembert method of the form \eqref{eq:Equations generating the numerical flow of the Hamiltonian system}.

\begin{thm} 
\label{thm:PRK method as a Lagrange-d'Alembert variational integrator}
The $s$-stage mean-square partitioned Runge-Kutta method \eqref{eq:SPRK for stochastic forced Hamiltonian systems} with the conditions \eqref{eq:Symplectic conditions for SPRK} is a stochastic Lagrange-d'Alembert variational integrator of the form \eqref{eq:Equations generating the numerical flow of the Hamiltonian system} with the discrete Hamiltonian

\begin{equation}
\label{eq:Discrete Hamiltonian for SPRK}
H^+_d(q_k,p_{k+1}) = p_{k+1} q_{k+1} - \Delta t \sum_{i=1}^s \alpha_i \bigg(P_i \frac{\partial H}{\partial p}(Q_i,P_i)-H(Q_i,P_i) \bigg) -  \sum_{r=1}^m \Delta W^r \sum_{i=1}^s  \beta_i \bigg(P_i \frac{\partial h_r}{\partial p}(Q_i,P_i)-h_r(Q_i,P_i) \bigg),
\end{equation}

\noindent
and the discrete forces

\begin{align}
\label{eq:Discrete forces for SPRK}
F^-_d(q_k,p_{k+1}) &= \Delta t \sum_{i=1}^s \hat \alpha_i \bigg(\frac{\partial Q_i}{\partial q_k}\bigg)^T F(Q_i,P_i) + \sum_{r=1}^m \Delta W^r \sum_{i=1}^s \hat \beta_i \bigg( \frac{\partial Q_i}{\partial q_k} \bigg)^Tf_r(Q_i,P_i), \nonumber \\
F^+_d(q_k,p_{k+1}) &= \Delta t \sum_{i=1}^s \hat \alpha_i \bigg( \frac{\partial Q_i}{\partial p_{k+1}} \bigg)^T F(Q_i,P_i) + \sum_{r=1}^m \Delta W^r \sum_{i=1}^s \hat \beta_i \bigg( \frac{\partial Q_i}{\partial p_{k+1}} \bigg)^T f_r(Q_i,P_i),
\end{align}

\noindent
where $q_{k+1}$, $p_k$, $Q_i$, and $P_i$ satisfy the system of equations \eqref{eq:SPRK for stochastic forced Hamiltonian systems} and are understood as functions of $q_k$ and $p_{k+1}$.
\end{thm}

\begin{proof}
The proof involves straightforward, although rather lengthy and tedious algebraic manipulations. Therefore, for the clarity and brevity of the exposition, we only consider the one-dimensional noise case $m=1$ and point out the key steps of the derivations. Let us introduce the following shorthand notation:

\begin{align}
\label{eq:Shorthand notation}
&\dot Q_i \equiv \frac{\partial H}{\partial p}(Q_i,P_i), \quad \dot P_i \equiv -\frac{\partial H}{\partial q}(Q_i,P_i), \quad F_i \equiv F(Q_i,P_i), \nonumber \\
&\dot K_i \equiv \frac{\partial h}{\partial p}(Q_i,P_i),  \quad \dot G_i \equiv -\frac{\partial h}{\partial q}(Q_i,P_i), \quad f_i \equiv f(Q_i,P_i).
\end{align} 

\noindent
Differentiate each of the equations \eqref{eq:SPRK for stochastic forced Hamiltonian systems} with respect to $q_k$ and $p_{k+1}$ to express the Jacobians $\partial Q_i/\partial q_k$, $\partial P_i/\partial q_k$, $\partial q_{k+1}/\partial q_k$, $\partial p_k/\partial q_k$, and analogous Jacobians with respect to $p_{k+1}$, in terms of the derivatives of the terms \eqref{eq:Shorthand notation}. For instance, we have

\begin{equation}
\label{eq:The derivative of Pi with respect to p_k+1}
\frac{\partial P_i}{\partial p_{k+1}} = I + \Delta t \sum_{j=1}^s(\bar a_{ij}-\alpha_j)\frac{\partial \dot P_j}{\partial p_{k+1}} + \Delta W \sum_{j=1}^s(\bar b_{ij}-\beta_j)\frac{\partial \dot K_j}{\partial p_{k+1}} + \Delta t \sum_{j=1}^s(\hat a_{ij}-\hat \alpha_j)\frac{\partial F_j}{\partial p_{k+1}} + \Delta W \sum_{j=1}^s(\hat b_{ij}-\hat \beta_j)\frac{\partial f_j}{\partial p_{k+1}},
\end{equation}

\noindent
where $I$ denotes the $N\times N$ identity matrix. Let us now calculate the derivative of the discrete Hamiltonian \eqref{eq:Discrete Hamiltonian for SPRK} with respect to $p_{k+1}$. After substituting the Jacobians \eqref{eq:The derivative of Pi with respect to p_k+1} and using \eqref{eq:SPRK for stochastic forced Hamiltonian systems 4} to replace $p_{k+1}$, we obtain the expression

\begin{align}
\label{eq:Calculating the derivative of the discrete Hamiltonian - part 1}
D_2H^+_d(q_k,p_{k+1}) = q_{k+1} &+ \Delta t^2 \sum_{i,j=1}^s \alpha_i \hat \alpha_j \bigg( \frac{\partial \dot Q_i}{\partial p_{k+1}} \bigg)^T F_j + \Delta t \Delta W \sum_{i,j=1}^s \alpha_i \hat \beta_j \bigg( \frac{\partial \dot Q_i}{\partial p_{k+1}} \bigg)^T f_j \nonumber \\
&+ \Delta t \Delta W \sum_{i,j=1}^s \beta_i \hat \alpha_j \bigg( \frac{\partial \dot K_i}{\partial p_{k+1}} \bigg)^T F_j + \Delta W^2 \sum_{i,j=1}^s \beta_i \hat \beta_j \bigg( \frac{\partial \dot K_i}{\partial p_{k+1}} \bigg)^T f_j \nonumber \\
&+\Delta t \sum_{i=1}^s \alpha_i \bigg( \frac{\partial \dot Q_i}{\partial p_{k+1}} \bigg)^T (p_k-P_i) + \Delta W \sum_{i=1}^s \beta_i \bigg( \frac{\partial \dot K_i}{\partial p_{k+1}} \bigg)^T (p_k-P_i) \nonumber \\
&+\Delta t^2 \sum_{i,j=1}^s (\alpha_i \alpha_j - \alpha_j a_{ji} ) \bigg( \frac{\partial \dot Q_i}{\partial p_{k+1}} \bigg)^T \dot P_j + \Delta t \Delta W \sum_{i,j=1}^s (\alpha_i \beta_j - \beta_j a_{ji} ) \bigg( \frac{\partial \dot Q_i}{\partial p_{k+1}} \bigg)^T \dot G_j \nonumber \\
&+\Delta t \Delta W \sum_{i,j=1}^s (\beta_i \alpha_j - \alpha_j b_{ji} ) \bigg( \frac{\partial \dot K_i}{\partial p_{k+1}} \bigg)^T \dot P_j + \Delta W^2 \sum_{i,j=1}^s (\beta_i \beta_j - \beta_j b_{ji} ) \bigg( \frac{\partial \dot K_i}{\partial p_{k+1}} \bigg)^T \dot G_j.
\end{align}

\noindent
After using \eqref{eq:Symplectic conditions for SPRK 1}-\eqref{eq:Symplectic conditions for SPRK 4} in the last four terms (e.g., $\alpha_i \alpha_j - \alpha_j a_{ji} = \alpha_i \bar a_{ij}$), and substituting \eqref{eq:SPRK for stochastic forced Hamiltonian systems 2} for $P_i$, we get

\begin{align}
\label{eq:Calculating the derivative of the discrete Hamiltonian - part 2}
D_2H^+_d(q_k,p_{k+1}) = q_{k+1} &+\Delta t^2 \sum_{i,j=1}^s (\alpha_i \hat \alpha_j - \alpha_i \hat a_{ij} ) \bigg( \frac{\partial \dot Q_i}{\partial p_{k+1}} \bigg)^T F_j  + \Delta t \Delta W \sum_{i,j=1}^s (\alpha_i \hat \beta_j - \alpha_i \hat b_{ij} ) \bigg( \frac{\partial \dot Q_i}{\partial p_{k+1}} \bigg)^T f_j \nonumber \\
&+\Delta t \Delta W \sum_{i,j=1}^s (\beta_i \hat \alpha_j - \beta_i \hat a_{ij} ) \bigg( \frac{\partial \dot K_i}{\partial p_{k+1}} \bigg)^T F_j + \Delta W^2 \sum_{i,j=1}^s (\beta_i \hat \beta_j - \beta_i \hat b_{ij} ) \bigg( \frac{\partial \dot K_i}{\partial p_{k+1}} \bigg)^T f_j.
\end{align}

\noindent
By using the conditions \eqref{eq:Symplectic conditions for SPRK 5}-\eqref{eq:Symplectic conditions for SPRK 8} and collecting terms, we finally arrive at

\begin{align}
\label{eq:Calculating the derivative of the discrete Hamiltonian - part 3}
D_2H^+_d(q_k,p_{k+1}) = q_{k+1} + \Delta t \sum_{i=1}^s \hat \alpha_i \bigg( \frac{\partial Q_i}{\partial p_{k+1}} \bigg)^T F_i  + \Delta W \sum_{i=1}^s \hat \beta_i \bigg( \frac{\partial Q_i}{\partial p_{k+1}} \bigg)^T f_i = q_{k+1} + F^+_d(q_k,p_{k+1}).
\end{align}

\noindent
In a similar fashion we derive

\begin{align}
\label{eq:Calculating the derivative of the discrete Hamiltonian - part 4}
D_1H^+_d(q_k,p_{k+1}) = p_k + F^-_d(q_k,p_{k+1}),
\end{align}

\noindent
which completes the proof.\\
\end{proof}

\subsubsection{Convergence}
\label{sec:Convergence of strong methods}

Mean-square convergence concentrates on pathwise approximations of the exact solutions (see \cite{KloedenPlatenSDE}, \cite{MilsteinBook}). Let $\bar z(t) = (\bar q(t),\bar p(t))$ be the exact solution to \eqref{eq: Stochastic dissipative Hamiltonian system} with the initial conditions $q_0$ and $p_0$, and let $z_k = (q_k,p_k)$ denote the numerical solution at time $t_k$ obtained by applying \eqref{eq:SPRK for stochastic forced Hamiltonian systems} iteratively $k$ times with the constant time step $\Delta t$. The numerical solution is said to converge in the mean-square sense with global order $r$ if there exist $\delta>0$ and a constant $C>0$ such that for all $\Delta t \in (0,\delta)$ we have

\begin{equation}
\label{eq:Definition of mean-square convergence}
\sqrt{E(\| z_K-\bar z(T)\|^2 )} \leq C\Delta t^r,
\end{equation}

\noindent
where $T = K\Delta t$, as defined before, and $E$ denotes the expected value. In principle, in order to determine the mean-square order of convergence of the Lagrange-d'Alembert partitioned Runge-Kutta method \eqref{eq:SPRK for stochastic forced Hamiltonian systems} we need to calculate the power series expansions of $q_{k+1}$ and $p_{k+1}$ in terms of the powers of $\Delta t$ and $\Delta W^i$, and compare them to the Stratonovich-Taylor expansions for the exact solution $\bar q(t_k+\Delta t)$ and $\bar p(t_k+\Delta t)$ (see \cite{Burrage2000}, \cite{KloedenPlatenSDE}, \cite{MilsteinBook}). As mentioned in Section~\ref{sec:Construction}, the mean-square order of the method \eqref{eq:SPRK for stochastic forced Hamiltonian systems} cannot exceed 1.0. Below we provide the conditions that have to be satisfied by the coefficients of the method \eqref{eq:SPRK for stochastic forced Hamiltonian systems} in order for it to be convergent.

\begin{thm}
\label{thm:Mean-square convergence theorem}
Suppose that, in addition to conditions (H1)-(H3), the functions $H(q,p)$, $F(q,p)$, and $h_i(q,p)$, $f_i(q,p)$ for $i=1,\ldots,m$ have all the necessary partial derivatives. Let the coefficients of the method \eqref{eq:SPRK for stochastic forced Hamiltonian systems} satisfy the conditions

\begin{align}
\label{eq:SPRK order conditions}
&\sum_{i=1}^s \alpha_i = \sum_{i=1}^s \hat \alpha_i = \sum_{i=1}^s \beta_i = \sum_{i=1}^s \hat \beta_i=1, \nonumber \\
&\sum_{i,j=1}^s \beta_i b_{ij}=\sum_{i,j=1}^s \beta_i \bar b_{ij} = \sum_{i,j=1}^s \beta_i \hat b_{ij} = \sum_{i,j=1}^s \hat \beta_i b_{ij}=\sum_{i,j=1}^s \hat \beta_i \bar b_{ij} = \sum_{i,j=1}^s \hat \beta_i \hat b_{ij} = \frac{1}{2}.
\end{align}

\noindent
If the noise is commutative, that is, if the following conditions are satisfied

\begin{equation}
\label{eq: Commutation conditions}
\Gamma_{ij}=\Gamma_{ji}, \qquad \Lambda_{ij}=\Lambda_{ji}, \qquad \text{for all $i,j=1,\ldots,m$},
\end{equation}

\noindent
where the vectors $\Gamma_{ij}$ and $\Lambda_{ij}$ for each $i,j=1,\ldots,m$ are defined as

\begin{align}
\label{eq: Gamma and Lambda terms definition}
\Gamma_{ij} &=  \frac{\partial^2 h_j}{\partial p \partial q} \frac{\partial h_i}{\partial p} - \frac{\partial^2 h_j}{\partial p^2} \frac{\partial h_i}{\partial q}+\frac{\partial^2 h_j}{\partial p^2} f_i, \nonumber \\
 \Lambda_{ij} &= -\frac{\partial^2 h_j}{\partial q^2} \frac{\partial h_i}{\partial p} + \frac{\partial^2 h_j}{\partial q \partial p} \frac{\partial h_i}{\partial q} + \frac{\partial f_j}{\partial q} \frac{\partial h_i}{\partial p}-\frac{\partial f_j}{\partial p} \frac{\partial h_i}{\partial q}-\frac{\partial^2 h_j}{\partial q \partial p} f_i+\frac{\partial f_j}{\partial p} f_i,
\end{align}

\noindent
then the method \eqref{eq:SPRK for stochastic forced Hamiltonian systems} is convergent with mean-square order 1.0. If the noise is noncommutative, then the method \eqref{eq:SPRK for stochastic forced Hamiltonian systems} is convergent with mean-square order 0.5.
\end{thm}

\begin{proof}
General order conditions for stochastic non-partitioned Runge-Kutta methods have been analyzed in \cite{Burrage1998} and \cite{Burrage2000}. Conditions for mean-square convergence of order 1.0 for stochastic partitioned Runge-Kutta methods with a one-dimensional noise have been derived in \cite{MaDing2015}. However, the method \eqref{eq:SPRK for stochastic forced Hamiltonian systems} is more general, as we allow a multidimensional noise, and different coefficients are applied to the Hamiltonian and forcing terms, but the method of proof is similar to the proof of Theorem~2.1 in \cite{MaDing2015}, therefore we only present the main steps. To simplify the notation, denote $\alpha=(\alpha_1,\ldots,\alpha_s)^T$, $b=(b_{ij})_{i,j=1,\ldots,s}$, and similarly for the remaining coefficients of the method. Let also $e=(1,1,\ldots,1)^T$ be an $s$-dimensional vector. Then the conditions \eqref{eq:SPRK order conditions} can be written more compactly, e.g., $\alpha^T e =1$ or $\beta^T b e = 1/2$. We first determine power expansions of the internal stages $Q_i$ and $P_i$ in terms of the powers of $\Delta t$ and $\Delta W^i$. We plug in series expansions for $Q_i$ and $P_i$ in Equations \eqref{eq:SPRK for stochastic forced Hamiltonian systems 1}-\eqref{eq:SPRK for stochastic forced Hamiltonian systems 2}, and determine their coefficients by expanding the derivatives of the Hamiltonians and forcing terms into Taylor series around $(q_k,p_k)$. Then we plug in thus found series expansions into Equations \eqref{eq:SPRK for stochastic forced Hamiltonian systems 3}-\eqref{eq:SPRK for stochastic forced Hamiltonian systems 4}, and again expand the derivatives of the Hamiltonians and forcing terms into Taylor series around $(q_k,p_k)$. This way we obtain the series expansions of $q_{k+1}$ and $p_{k+1}$ as

\begin{align}
\label{eq:Series expansion for the stochastic PRK}
q_{k+1} &= q_k + (\alpha^T e)\frac{\partial H}{\partial p}\Delta t + (\beta^T e)\sum_{i=1}^m \frac{\partial h_i}{\partial p}\Delta W^i + \frac{1}{2} \sum_{i=1}^M \bar \Gamma_{ii} (\Delta W^i)^2 + \frac{1}{2}\sum_{i=1}^M \sum_{\substack{j=1 \\ j \not = i}}^M \bar \Gamma_{ij} \Delta W^i \Delta W^j + \ldots, \nonumber \\
p_{k+1}&= p_k - (\alpha^T e)\frac{\partial H}{\partial q}\Delta t + (\hat \alpha^T e)F\Delta t - (\beta^T e)\sum_{i=1}^m \frac{\partial h_i}{\partial q}\Delta W^i + (\hat \beta^T e)\sum_{i=1}^m f_i\Delta W^i\nonumber \\
 &\phantom{= p_k-} + \frac{1}{2} \sum_{i=1}^M \bar \Lambda_{ii} (\Delta W^i)^2 + \frac{1}{2}\sum_{i=1}^M \sum_{\substack{j=1 \\ j \not = i}}^M \bar \Lambda_{ij} \Delta W^i \Delta W^j +\ldots,
\end{align}

\noindent
where the vectors $\bar \Gamma_{ij}$ and $\bar \Lambda_{ij}$ for each $i,j=1,\ldots,m$ are defined as

\begin{align}
\label{eq: Gamma bar and Lambda bar terms definition}
\bar \Gamma_{ij} &=  2 (\beta^T b e)\frac{\partial^2 h_j}{\partial p \partial q} \frac{\partial h_i}{\partial p} - 2 (\beta^T \bar b e)\frac{\partial^2 h_j}{\partial p^2} \frac{\partial h_i}{\partial q} + 2 (\beta^T \hat b e) \frac{\partial^2 h_j}{\partial p^2} f_i, \nonumber \\
\bar \Lambda_{ij} &= -2 (\beta^T b e)\frac{\partial^2 h_j}{\partial q^2} \frac{\partial h_i}{\partial p} + 2 (\beta^T \bar b e)\frac{\partial^2 h_j}{\partial q \partial p} \frac{\partial h_i}{\partial q} + 2 (\hat \beta^T b e)\frac{\partial f_j}{\partial q} \frac{\partial h_i}{\partial p}-2 (\hat \beta^T \bar b e)\frac{\partial f_j}{\partial p} \frac{\partial h_i}{\partial q} \nonumber \\
                  &\phantom{=-}-2 (\beta^T \hat b e)\frac{\partial^2 h_j}{\partial q \partial p} f_i+2 (\hat \beta^T \hat b e)\frac{\partial f_j}{\partial p} f_i,
\end{align}

\noindent
and the forcing terms and the derivatives of the Hamiltonians are evaluated at $(q_k, p_k)$. Let $\bar q(t;q_k,p_k)$ and $\bar p(t;q_k,p_k)$ denote the exact solution of \eqref{eq: Stochastic dissipative Hamiltonian system} such that $\bar q(t_k;q_k,p_k)=q_k$ and $\bar p(t_k;q_k,p_k)=p_k$. Using \eqref{eq: Stochastic dissipative Hamiltonian system} we calculate the Stratonovich-Taylor expansions for $\bar q(t_{k+1};q_k,p_k)$ and $\bar p(t_{k+1};q_k,p_k)$ as (see \cite{KloedenPlatenSDE})

\begin{align}
\label{eq:Stratonovich-Taylor expansion with multidimensional noise}
\bar q(t_{k+1};q_k,p_k) &= q_k + \frac{\partial H}{\partial p}\Delta t + \sum_{i=1}^m \frac{\partial h_i}{\partial p}\Delta W^i + \frac{1}{2} \sum_{i=1}^m \Gamma_{ii} (\Delta W^i)^2 + \sum_{i=1}^m \sum_{\substack{j=1 \\ j \not = i}}^m \Gamma_{ij} J_{ij} + \ldots, \nonumber \\
\bar p(t_{k+1};q_k,p_k)&= p_k + \bigg(- \frac{\partial H}{\partial q}+F \bigg) \Delta t + \sum_{i=1}^m \bigg( -\frac{\partial h_i}{\partial q}+f_i \bigg)\Delta W^i + \frac{1}{2} \sum_{i=1}^m \Lambda_{ii} (\Delta W^i)^2 + \sum_{i=1}^m \sum_{\substack{j=1 \\ j \not = i}}^m \Lambda_{ij} J_{ij} +\ldots,
\end{align}

\noindent
where $J_{ij} = \int_{t_k}^{t_{k+1}}\int_{t_k}^{t}dW^i(\tau)\circ dW^j(t)$ denotes a double Stratonovich integral, $\Gamma_{ij}$ and $\Lambda_{ij}$ have been defined in \eqref{eq: Gamma and Lambda terms definition}, and the forcing terms and the derivatives of the Hamiltonians are again evaluated at $(q_k, p_k)$. Assuming the conditions \eqref{eq:SPRK order conditions} are satisfied, we have that $\bar \Gamma_{ij} = \Gamma_{ij}$ and $\bar \Lambda_{ij} = \Lambda_{ij}$, but comparing \eqref{eq:Series expansion for the stochastic PRK} and \eqref{eq:Stratonovich-Taylor expansion with multidimensional noise}, we find that in the general case of noncommutative noise not all first order terms agree, and therefore we only have the local error estimates 

\begin{align}
\label{eq:Estimation of the remainders for multidimensional noise}
E\big(q_{k+1}-\bar q(t_{k+1};q_k,p_k)\big) &= O(\Delta t^{\frac{3}{2}}), \qquad \sqrt{E\big( \|q_{k+1}-\bar q(t_{k+1};q_k,p_k)\|^2 \big)} = O(\Delta t), \nonumber \\
E\big(p_{k+1}-\bar p(t_{k+1};q_k,p_k)\big) &= O(\Delta t^{\frac{3}{2}}), \qquad \sqrt{E\big( \|p_{k+1}-\bar p(t_{k+1};q_k,p_k)\|^2 \big)} = O(\Delta t).
\end{align}

\noindent
Theorem~1.1 from \cite{MilsteinBook} then implies that the method \eqref{eq:SPRK for stochastic forced Hamiltonian systems} has mean-square order 0.5. However, if the noise is commutative, then using the property $J_{ij}+J_{ji}=\Delta W^i \Delta W^j$ (see \cite{KloedenPlatenSDE}, \cite{MilsteinBook}), one can easily show

\begin{equation}
\label{eq: Jij terms from the Stratonovich-Taylor expansion}
\sum_{i=1}^m \sum_{\substack{j=1 \\ j \not = i}}^m \Gamma_{ij} J_{ij} = \frac{1}{2}\sum_{i=1}^m \sum_{\substack{j=1 \\ j \not = i}}^m \Gamma_{ij} \Delta W^i \Delta W^j, \qquad \quad \sum_{i=1}^m \sum_{\substack{j=1 \\ j \not = i}}^m \Lambda_{ij} J_{ij} = \frac{1}{2}\sum_{i=1}^m \sum_{\substack{j=1 \\ j \not = i}}^m \Lambda_{ij} \Delta W^i \Delta W^j.
\end{equation} 

\noindent
In that case all first-order terms in the expansions \eqref{eq:Series expansion for the stochastic PRK} and \eqref{eq:Stratonovich-Taylor expansion with multidimensional noise} agree, and we have the local error estimates

\begin{align}
\label{eq:Estimation of the remainders for 1D noise}
E\big(q_{k+1}-\bar q(t_{k+1};q_k,p_k)\big) &= O(\Delta t^2), \qquad \sqrt{E\big( \|q_{k+1}-\bar q(t_{k+1};q_k,p_k)\|^2 \big)} = O(\Delta t^{\frac{3}{2}}), \nonumber \\
E\big(p_{k+1}-\bar p(t_{k+1};q_k,p_k)\big) &= O(\Delta t^2), \qquad \sqrt{E\big( \|p_{k+1}-\bar p(t_{k+1};q_k,p_k)\|^2 \big)} = O(\Delta t^{\frac{3}{2}}).
\end{align}

\noindent
Theorem~1.1 from \cite{MilsteinBook} then implies that the method \eqref{eq:SPRK for stochastic forced Hamiltonian systems} has mean-square order $1.0$.\\
\end{proof}

\noindent
In the case of a one-dimensional noise the commutation condition \eqref{eq: Commutation conditions} is trivially satisfied, therefore we have the following corollary.

\begin{corollary}
\label{thm:Mean-square convergence corollary}
Under the assumptions of Theorem~\ref{thm:Mean-square convergence theorem}, the method \eqref{eq:SPRK for stochastic forced Hamiltonian systems} is convergent with mean-square order 1.0 for systems driven by a one-dimensional noise.
\end{corollary}

\subsubsection{Examples}
\label{sec:Examples of strong methods}

In the construction of the integrator \eqref{eq:SPRK for stochastic forced Hamiltonian systems} we may choose the number of stages $s$. In the deterministic case, the higher the number of stages, the higher order of convergence can be achieved (see \cite{HLWGeometric}, \cite{HWODE1}, \cite{HWODE2}). In our case, however, as explained earlier, we cannot in general achieve mean-square order of convergence higher than 1.0, because we only used $\Delta W^r$ in \eqref{eq:SPRK for stochastic forced Hamiltonian systems}. Since the system \eqref{eq:SPRK for stochastic forced Hamiltonian systems 1}-\eqref{eq:SPRK for stochastic forced Hamiltonian systems 2} requires solving $2sN$ equations for $2sN$ variables, from the computational point of view it makes sense to only consider methods with low values of $s$. In this work we focus on the following classical numerical integration formulas (one can easily verify that the conditions \eqref{eq:Symplectic conditions for SPRK} and \eqref{eq:SPRK order conditions} are satisfied for the discussed methods).

\begin{enumerate}
	\item \emph{Stochastic midpoint method} \\ Using the midpoint rule we obtain a one-stage non-partitioned Runge-Kutta method represented by the tableau
	
	\begin{equation}
		\label{eq: Stochastic midpoint method tableau}
		\begin{array}{c|c|c|c|c|c|c}
			& \frac{1}{2} & \frac{1}{2} & \frac{1}{2} & \frac{1}{2} & \frac{1}{2} & \frac{1}{2} \\
			\hline
			& 1 & 1 & 1 & 1 & 1 & 1
		\end{array}.
	\end{equation}
	
	\noindent
	Noting that $Q_1=(q_k+q_{k+1})/2$ and $P_1=(p_k+p_{k+1})/2$, this method can be written as
	
	\begin{align}
	\label{eq:Stochastic midpoint method}
	q_{k+1} &= q_k + \frac{\partial H}{\partial p} \bigg(\frac{q_k+q_{k+1}}{2},\frac{p_k+p_{k+1}}{2} \bigg)\Delta t 
	               + \sum_{i=1}^m \frac{\partial h_i}{\partial p} \bigg(\frac{q_k+q_{k+1}}{2},\frac{p_k+p_{k+1}}{2} \bigg)\Delta W^i, \nonumber \\
	p_{k+1} &= p_k + \bigg[ -\frac{\partial H}{\partial q} \bigg(\frac{q_k+q_{k+1}}{2},\frac{p_k+p_{k+1}}{2} \bigg) + F\bigg(\frac{q_k+q_{k+1}}{2},\frac{p_k+p_{k+1}}{2} \bigg) \bigg] \Delta t \nonumber \\ 
	        &\phantom{= p_k}+ \sum_{i=1}^m \bigg[ -\frac{\partial h_i}{\partial q} \bigg(\frac{q_k+q_{k+1}}{2},\frac{p_k+p_{k+1}}{2} \bigg) + f_i\bigg(\frac{q_k+q_{k+1}}{2},\frac{p_k+p_{k+1}}{2} \bigg) \bigg]\Delta W^i.
	\end{align}
	
	\noindent
	The stochastic midpoint method was considered in \cite{MilsteinRepin} and \cite{MaDing2015} in the context of symplectic integrators for stochastic Hamiltonian systems without forcing; see also \cite{HolmTyranowskiGalerkin}. This example demonstrates that the stochastic midpoint method retains its geometric properties also for forced systems. It is an implicit method and in general one has to solve $2N$ equations for $2N$ unknowns. However, if the Hamiltonians are separable, that is, $H(q,p)=T_0(p)+U_0(q)$ and $h_i(q,p)=T_i(p)+U_i(q)$, then $q_{k+1}$ from the first equation can be substituted into the second one. In that case only $N$ nonlinear equations have to be solved for $p_{k+1}$.

	\item \emph{Stochastic St{\"o}rmer-Verlet method} \\ A generalization of the classical St{\"o}rmer-Verlet method can be obtained by choosing the tableau
	
	\begin{equation}
		\label{eq: Stochastic Stormer-Verlet method tableau}
		\begin{array}{c|cc|cc|cc|cc|cc|cc}
			& 0 & 0 & \frac{1}{2} & 0 & \frac{1}{2} & 0 & 0 & 0 & \frac{1}{2} & 0 & \frac{1}{2} & 0 \\

			& \frac{1}{2} & \frac{1}{2}	& \frac{1}{2} & 0	& \frac{1}{2} & 0 & \frac{1}{2} & \frac{1}{2}	& \frac{1}{2} & 0	& \frac{1}{2} & 0\\
			\hline
			& \frac{1}{2} & \frac{1}{2}	& \frac{1}{2} & \frac{1}{2}	& \frac{1}{2} & \frac{1}{2}	& \frac{1}{2} & \frac{1}{2}	& \frac{1}{2} & \frac{1}{2}	& \frac{1}{2} & \frac{1}{2}
		\end{array}.
	\end{equation}
	
 \noindent	
 Noting that $Q_1=q_k$, $Q_2=q_{k+1}$, and $P_1=P_2$, this method can be more efficiently written as
	
	\begin{align}
	\label{eq:Stochastic Stormer-Verlet method}
	P_1 &= p_k + \frac{1}{2} \bigg[ -\frac{\partial H}{\partial q} \big(q_k, P_1 \big) + F\big(q_k, P_1 \big) \bigg] \Delta t 
	           + \frac{1}{2} \sum_{i=1}^m \bigg[ -\frac{\partial h_i}{\partial q} \big(q_k, P_1 \big) + f_i\big(q_k, P_1 \big) \bigg] \Delta W^i, \nonumber \\
	q_{k+1}&= q_k + \frac{1}{2} \frac{\partial H}{\partial p} \big(q_k, P_1 \big)\Delta t
	              + \frac{1}{2} \frac{\partial H}{\partial p} \big(q_{k+1}, P_1 \big)\Delta t
	              + \frac{1}{2} \sum_{i=1}^m \frac{\partial h_i}{\partial p} \big(q_k, P_1 \big)\Delta W^i
	              + \frac{1}{2} \sum_{i=1}^m \frac{\partial h_i}{\partial p} \big(q_{k+1}, P_1 \big)\Delta W^i, \nonumber \\
	p_{k+1}&= P_1 + \frac{1}{2} \bigg[ -\frac{\partial H}{\partial q} \big(q_{k+1}, P_1 \big) + F\big(q_{k+1}, P_1 \big) \bigg] \Delta t 
	           + \frac{1}{2} \sum_{i=1}^m \bigg[ -\frac{\partial h_i}{\partial q} \big(q_{k+1}, P_1 \big) + f_i\big(q_{k+1}, P_1 \big) \bigg]\Delta W^i.
	\end{align}
	
	\noindent
	This method was considered in \cite{MaDing2015} in the context of symplectic integrators for stochastic Hamiltonian systems without forcing; see also \cite{HolmTyranowskiGalerkin}. It is particularly efficient, because the first equation can be solved separately from the second one, and the last equation is an explicit update. Moreover, if the Hamiltonians are separable, the second equation becomes explicit. If in addition the forcing terms $F$ and $f_i$ have special forms, then further improvements in efficiency are possible. For instance, if the forcing terms depend linearly on $p$, as is often the case in practical applications, then the first equation is a linear equation for $P_1$, and can be solved using linear solvers. In case the forcing terms are independent of $p$ altogether, then the whole method becomes fully explicit.

\item \emph{2-stage stochastic DIRK method} \\ In order to reduce the computational cost of solving nonlinear equations, diagonally implicit Runge-Kutta (DIRK) methods use lower-triangular tableaus (see \cite{HLWGeometric}, \cite{HWODE1}, \cite{HWODE2}). One can easily verify that the most general family of 2-stage stochastic DIRK methods that satisfy the conditions \eqref{eq:Symplectic conditions for SPRK} and \eqref{eq:SPRK order conditions} has a tableau of the form
	
	\begin{equation}
		\label{eq: Stochastic DIRK method tableau}
		\begin{array}{c|cc|cc|cc|cc|cc|cc}
			& \frac{\lambda}{2} & 0	& \frac{\lambda}{2} & 0	& \frac{\lambda}{2} & 0	& \frac{\lambda}{2} & 0	& \frac{\lambda}{2} & 0	& \frac{\lambda}{2} & 0 \\
			& \lambda	& \frac{1-\lambda}{2}	& \lambda	& \frac{1-\lambda}{2}	& \lambda	& \frac{1-\lambda}{2}	& \lambda	& \frac{1-\lambda}{2}	& \lambda	& \frac{1-\lambda}{2}	& \lambda	& \frac{1-\lambda}{2}	\\
			\hline
			& \lambda & 1-\lambda	& \lambda & 1-\lambda	& \lambda & 1-\lambda	& \lambda & 1-\lambda	& \lambda & 1-\lambda	& \lambda & 1-\lambda	
		\end{array},
	\end{equation}
	
\noindent
where $\lambda \in \mathbb{R}$ is an arbitrary parameter. One can check that for $\lambda=0$ and $\lambda=1$, this method reduces to the stochastic midpoint method \eqref{eq:Stochastic midpoint method}. For other choices of $\lambda$, one needs to solve equations \eqref{eq:SPRK for stochastic forced Hamiltonian systems 1} and \eqref{eq:SPRK for stochastic forced Hamiltonian systems 2}, first for $i=1$ ($2N$ equations) in order to calculate the internal stages $Q_1$ and $P_1$ ($2N$ variables), and then for $i=2$ ($2N$ equations) to find the internal stages $Q_2$ and $P_2$ ($2N$ variables). If the Hamiltonians are separable, then equations \eqref{eq:SPRK for stochastic forced Hamiltonian systems 1} can be substituted into equations \eqref{eq:SPRK for stochastic forced Hamiltonian systems 2}, and the problem is reduced to solving two systems of $N$ equations each.
	
\end{enumerate}

\noindent
Note that the methods \eqref{eq:Stochastic midpoint method}, \eqref{eq:Stochastic Stormer-Verlet method}, and \eqref{eq: Stochastic DIRK method tableau} are in general implicit. One can use the Implicit Function Theorem to show that for sufficiently small $\Delta t$ and $|\Delta W^i|$, the relevant nonlinear equations will have a solution. However, since the increments $\Delta W^i$ are unbounded, for some values of $\Delta W^i$ solutions might not exist. To avoid problems with numerical implementations, if necessary, one can replace $\Delta W^i$ in equations \eqref{eq:Stochastic midpoint method} and \eqref{eq:Stochastic Stormer-Verlet method} with the truncated random variables $\overline{\Delta W^i}$ defined as

\begin{align}
\label{eq:Truncated Wiener increments}
\overline{\Delta W^i} =
\begin{cases}
A, & \text{if $\Delta W^i > A$},\\  
\Delta W^i, & \text{if $|\Delta W^i| \leq A$}, \\
 -A, & \text{if $\Delta W^i < -A$},
\end{cases}
\end{align}

\noindent
where $A>0$ is suitably chosen for the considered problem. See \cite{Burrage2004} and \cite{MilsteinRepin} for more details regarding schemes with truncated random increments and their convergence.

\subsection{Weak Lagrange-d'Alembert Runge-Kutta methods}
\label{sec:Weak Lagrange-d'Alembert partitioned Runge-Kutta methods}

\subsubsection{Construction}
\label{sec:Weak Construction}

A general class of weak stochastic Runge-Kutta methods for Stratonovich ordinary differential equations was introduced and analyzed in \cite{Rossler2004} and \cite{Rossler2007}. These ideas were later used by Wang \& Hong \& Xu \cite{Wang2017} to construct weak symplectic Runge-Kutta methods for stochastic Hamiltonian systems without forcing. Below we combine these ideas and introduce weak Lagrange-d'Alembert Runge-Kutta methods for stochastic forced Hamiltonian systems of the form \eqref{eq: Stochastic dissipative Hamiltonian system}.

\begin{define} An $s$-stage weak Lagrange-d'Alembert Runge-Kutta method for the system \eqref{eq: Stochastic dissipative Hamiltonian system} is given by

\begin{subequations}
\label{eq:WPRK for stochastic forced Hamiltonian systems}
\begin{align}
\label{eq:WPRK for stochastic forced Hamiltonian systems 1}
Q^{(0)}_i &= q_k + \Delta t \sum_{j=1}^s a^{(0)}_{ij} \frac{\partial H}{\partial p}(Q^{(0)}_j,P^{(0)}_j) + \sum_{r=1}^m \hat I_r \sum_{j=1}^s b^{(0)}_{ij} \frac{\partial h_r}{\partial p}(Q^{(r)}_j,P^{(r)}_j), \quad \qquad i=1,\ldots,s,  \\
\label{eq:WPRK for stochastic forced Hamiltonian systems 2}
P^{(0)}_i &= p_k + \Delta t \sum_{j=1}^s a^{(0)}_{ij} \bigg[-\frac{\partial H}{\partial q}(Q^{(0)}_j,P^{(0)}_j)+F(Q^{(0)}_j,P^{(0)}_j)\bigg]  \nonumber \\
&\phantom{= p_k}\;+ \sum_{r=1}^m \hat I_r \sum_{j=1}^s  b^{(0)}_{ij} \bigg[-\frac{\partial h_r}{\partial q}(Q^{(r)}_j,P^{(r)}_j) + f_r(Q^{(r)}_j,P^{(r)}_j) \bigg], \qquad \qquad \qquad i=1,\ldots,s, \\
\label{eq:WPRK for stochastic forced Hamiltonian systems 3}
Q^{(l)}_i &= q_k + \Delta t \sum_{j=1}^s a^{(1)}_{ij} \frac{\partial H}{\partial p}(Q^{(0)}_j,P^{(0)}_j) + \hat I_l \sum_{j=1}^s b^{(1)}_{ij} \frac{\partial h_l}{\partial p}(Q^{(l)}_j,P^{(l)}_j) \nonumber \\
&\phantom{= q_k}\;+\sum_{\substack{r=1 \\ r \not = l}}^m \hat I_r \sum_{j=1}^s b^{(3)}_{ij} \frac{\partial h_r}{\partial p}(Q^{(r)}_j,P^{(r)}_j), \qquad \qquad \qquad \qquad \qquad \;\, i=1,\ldots,s, \quad l=1,\ldots,m,  \\
\label{eq:WPRK for stochastic forced Hamiltonian systems 4}
P^{(l)}_i &= p_k + \Delta t \sum_{j=1}^s a^{(1)}_{ij} \bigg[-\frac{\partial H}{\partial q}(Q^{(0)}_j,P^{(0)}_j)+F(Q^{(0)}_j,P^{(0)}_j)\bigg]  \nonumber \\
&\phantom{= p_k}\;+ \hat I_l \sum_{j=1}^s  b^{(1)}_{ij} \bigg[-\frac{\partial h_l}{\partial q}(Q^{(l)}_j,P^{(l)}_j) + f_l(Q^{(l)}_j,P^{(l)}_j) \bigg] \nonumber \\
&\phantom{= p_k}\;+ \sum_{\substack{r=1 \\ r \not = l}}^m \hat I_r \sum_{j=1}^s  b^{(3)}_{ij} \bigg[-\frac{\partial h_r}{\partial q}(Q^{(r)}_j,P^{(r)}_j) + f_r(Q^{(r)}_j,P^{(r)}_j) \bigg], \;\,\quad i=1,\ldots,s, \quad l=1,\ldots,m, \\
\label{eq:WPRK for stochastic forced Hamiltonian systems 5}
q_{k+1} &= q_k + \Delta t \sum_{i=1}^s \alpha_i \frac{\partial H}{\partial p}(Q^{(0)}_i,P^{(0)}_i) + \sum_{r=1}^m \hat I_r \sum_{i=1}^s \beta_i \frac{\partial h_r}{\partial p}(Q^{(r)}_i,P^{(r)}_i),\\
\label{eq:WPRK for stochastic forced Hamiltonian systems 6}
p_{k+1} &= p_k + \Delta t \sum_{i=1}^s \alpha_i \bigg[ -\frac{\partial H}{\partial q}(Q^{(0)}_i,P^{(0)}_i) + F(Q^{(0)}_i,P^{(0)}_i) \bigg] \nonumber \\
&\phantom{= p_k}\;+ \sum_{r=1}^m \hat I_r \sum_{i=1}^s \beta_i \bigg[ -\frac{\partial h_r}{\partial q}(Q^{(r)}_i,P^{(r)}_i) + f_r(Q^{(r)}_i,P^{(r)}_i) \bigg],
\end{align}
\end{subequations}

\noindent
where $\Delta t$ is the time step, $\hat I_1, \ldots, \hat I_m$ are independent three-point distributed random variables with $P(\hat I_r=\pm \sqrt{3 \Delta t})=1/6$ and $P(\hat I_r=0)=2/3$ , $Q^{(0)}_i$, $Q^{(l)}_i$, $P^{(0)}_i$,  and $P^{(l)}_i$ for $i=1,\ldots,s$ and $l=1,\ldots,m$ are the position and momentum internal stages, respectively, and the coefficients of the method $a^{(0)}_{ij}$, $a^{(1)}_{ij}$, $b^{(0)}_{ij}$, $b^{(1)}_{ij}$, $b^{(3)}_{ij}$, $\alpha_i$, $\beta_i$ satisfy the conditions

\begin{subequations}
\label{eq:Symplectic conditions for WPRK}
\begin{align}
\label{eq:Symplectic conditions for WPRK 1}
\alpha_i a^{(0)}_{ij} + \alpha_j a^{(0)}_{ji} &= \alpha_i \alpha_j, \\
\label{eq:Symplectic conditions for WPRK 2}
\alpha_i b^{(0)}_{ij} + \beta_j a^{(1)}_{ji} &= \alpha_i \beta_j, \\
\label{eq:Symplectic conditions for WPRK 3}
\beta_i b^{(1)}_{ij} + \beta_j b^{(1)}_{ji} &= \beta_i \beta_j, \\
\label{eq:Symplectic conditions for WPRK 4}
\beta_i b^{(3)}_{ij} + \beta_j b^{(3)}_{ji} &= \beta_i \beta_j,
\end{align}
\end{subequations}

\noindent
for $i,j=1,\ldots,s$.
\end{define}

The Runge-Kutta method \eqref{eq:WPRK for stochastic forced Hamiltonian systems} can be represented by the tableau

\begin{equation}
\label{eq: WPRK tableau}
\begin{array}{c|c|c|c}
& a^{(0)} & b^{(0)} \\
\hline
& a^{(1)} & b^{(1)} & b^{(3)} \\
\hline
& \alpha^T & \beta^T
\end{array},
\end{equation}

\noindent
where $a^{(0)}=(a^{(0)}_{ij})_{i,j=1\ldots s}$, $\alpha = (\alpha_i)_{i=1\ldots s}$, etc. The set of equations \eqref{eq:WPRK for stochastic forced Hamiltonian systems} forms a one-step numerical scheme. Knowing $q_k$ and $p_k$ at time $t_k$, one can solve Equations \eqref{eq:WPRK for stochastic forced Hamiltonian systems 1}-\eqref{eq:WPRK for stochastic forced Hamiltonian systems 4} for the internal stages $Q^{(0)}_i$, $Q^{(l)}_i$, $P^{(0)}_i$ and $P^{(l)}_i$, and then use \eqref{eq:WPRK for stochastic forced Hamiltonian systems 5}-\eqref{eq:WPRK for stochastic forced Hamiltonian systems 6} to determine $q_{k+1}$ and $p_{k+1}$ at time $t_{k+1}$. Depending on the choice of the coefficients, the method \eqref{eq:WPRK for stochastic forced Hamiltonian systems} is in general implicit. However, since the random variables $\hat I_l$ are bounded, one can show that for sufficiently small $\Delta t$, the relevant nonlinear equations will have a solution. Below we prove that the Runge-Kutta method \eqref{eq:WPRK for stochastic forced Hamiltonian systems} with the conditions \eqref{eq:Symplectic conditions for WPRK} is indeed a stochastic Lagrange-d'Alembert method of the form \eqref{eq:Equations generating the numerical flow of the Hamiltonian system}. 

\begin{thm} 
\label{thm:Weak RK method as a Lagrange-d'Alembert variational integrator}
The $s$-stage weak Runge-Kutta method \eqref{eq:WPRK for stochastic forced Hamiltonian systems} with the conditions \eqref{eq:Symplectic conditions for WPRK} is a stochastic Lagrange-d'Alembert variational integrator of the form \eqref{eq:Equations generating the numerical flow of the Hamiltonian system} with the discrete Hamiltonian

\begin{align}
\label{eq:Discrete Hamiltonian for WPRK}
H^+_d(q_k,p_{k+1}) = p_{k+1} q_{k+1} &- \Delta t \sum_{i=1}^s \alpha_i \bigg(P^{(0)}_i \frac{\partial H}{\partial p}(Q^{(0)}_i,P^{(0)}_i)-H(Q^{(0)}_i,P^{(0)}_i) \bigg) \nonumber \\
&-  \sum_{r=1}^m \hat I_r \sum_{i=1}^s  \beta_i \bigg(P^{(r)}_i \frac{\partial h_r}{\partial p}(Q^{(r)}_i,P^{(r)}_i)-h_r(Q^{(r)}_i,P^{(r)}_i) \bigg),
\end{align}

\noindent
and the discrete forces

\begin{align}
\label{eq:Discrete forces for WPRK}
F^-_d(q_k,p_{k+1}) &= \Delta t \sum_{i=1}^s \alpha_i \bigg(\frac{\partial Q^{(0)}_i}{\partial q_k}\bigg)^T F(Q^{(0)}_i,P^{(0)}_i) + \sum_{r=1}^m \hat I_r \sum_{i=1}^s \beta_i \bigg( \frac{\partial Q^{(r)}_i}{\partial q_k} \bigg)^T f_r(Q^{(r)}_i,P^{(r)}_i), \nonumber \\
F^+_d(q_k,p_{k+1}) &= \Delta t \sum_{i=1}^s \alpha_i \bigg( \frac{\partial Q^{(0)}_i}{\partial p_{k+1}} \bigg)^T F(Q^{(0)}_i,P^{(0)}_i) + \sum_{r=1}^m \hat I_r \sum_{i=1}^s \beta_i \bigg( \frac{\partial Q^{(r)}_i}{\partial p_{k+1}} \bigg)^T f_r(Q^{(r)}_i,P^{(r)}_i),
\end{align}

\noindent
where $q_{k+1}$, $p_k$, $Q^{(0)}_i$, $Q^{(r)}_i$, $P^{(0)}_i$, and $P^{(r)}_i$, satisfy the system of equations \eqref{eq:WPRK for stochastic forced Hamiltonian systems} and are understood as functions of $q_k$ and $p_{k+1}$.
\end{thm}

\begin{proof}
The proof is analogous to the proof of Theorem~\ref{thm:PRK method as a Lagrange-d'Alembert variational integrator}.
\end{proof}

\paragraph{Remark.} For stochastic Hamiltonian systems without forcing, i.e. $F\equiv 0$, $f_r\equiv 0$, the method \eqref{eq:WPRK for stochastic forced Hamiltonian systems} reduces to a weak symplectic Runge-Kutta method of the type introduced in \cite{Wang2017}. Therefore, in that case Theorem~\ref{thm:Weak RK method as a Lagrange-d'Alembert variational integrator} also provides a type-II generating function for such a family of methods, and consequently an alternative proof of their symplecticity.

\subsubsection{Convergence}
\label{sec:Weak Convergence}

Rather than precisely approximating each sample path, weak convergence concentrates on approximating the probability distribution and functionals of the exact solution (see \cite{KloedenPlatenSDE}, \cite{MilsteinBook}). Let $\bar z(t) = (\bar q(t),\bar p(t))$ be the exact solution to \eqref{eq: Stochastic dissipative Hamiltonian system} with the initial conditions $q_0$ and $p_0$, and let $z_k = (q_k,p_k)$ denote the numerical solution at time $t_k$ obtained by applying \eqref{eq:WPRK for stochastic forced Hamiltonian systems} iteratively $k$ times with the constant time step $\Delta t$. The numerical solution is said to converge weakly with weak global order $r$ if for each $\varphi \in C^{2(r+1)}_P(T^*Q,\mathbb{R})$ there exists $\delta>0$ and a constant $C>0$ such that for all $\Delta t \in (0,\delta)$ we have

\begin{equation}
\label{eq:Definition of weak convergence}
\big\| E\big( \varphi (z_K) \big) - E\big( \varphi (\bar z(T)) \big) \big\| \leq C\Delta t^r,
\end{equation}

\noindent
where $T=K \Delta t$, and $C^{\alpha}_P(T^*Q,\mathbb{R})$ denotes the space of all $\varphi \in C^{\alpha}(T^*Q,\mathbb{R})$ with polynomial growth, i.e., there exists a constant $A>0$ and $\gamma \in \mathbb{N}$ such that $|\partial^\beta_z \varphi(z)| \leq A(1+\|z\|^{2 \gamma})$ for all $z \in T^*Q$ and any partial derivative of order $\beta \leq \alpha$. Weak convergence of the Runge-Kutta methods of type \eqref{eq:WPRK for stochastic forced Hamiltonian systems} has been analyzed, and the relevant order conditions for the coefficients have been derived in \cite{Rossler2007}.

\subsubsection{Examples}
\label{sec:Weak Examples}

In \cite{Wang2017} a number of weak symplectic Runge-Kutta methods for stochastic Hamiltonian systems without forcing have been proposed. Since the symplecticity conditions derived in \cite{Wang2017} are equivalent to the conditions \eqref{eq:Symplectic conditions for WPRK}, these methods become Lagrange-d'Alembert integrators when applied to systems with forcing. In this work, we particularly focus on two methods, namely $SRKw1$ and $SRKw2$, as dubbed in \cite{Wang2017}.

\begin{enumerate}
	\item \emph{SRKw1} \\ The family of 1-stage $SRKw1$ methods is defined by the tableau
	
	\begin{equation}
		\label{eq: SRKw1 tableau}
		\begin{array}{c|c|c|c}
		& \frac{1}{2} & \lambda \\
		\hline
		& 1-\lambda & \frac{1}{2} & \frac{1}{2} \\
		\hline
		& 1 & 1
		\end{array},
	\end{equation}
	
	\noindent
	where $\lambda\in \mathbb{R}$ is an arbitrary parameter. This method is weakly convergent with order 1.0 (see \cite{Rossler2007}, \cite{Wang2017}). Since $b^{(1)}=b^{(3)}$, equations \eqref{eq:WPRK for stochastic forced Hamiltonian systems 3} and \eqref{eq:WPRK for stochastic forced Hamiltonian systems 4} imply that $Q^{(1)}_1=\ldots=Q^{(m)}_1$ and $P^{(1)}_1=\ldots=P^{(m)}_1$. Therefore, in general one has to solve the system \eqref{eq:WPRK for stochastic forced Hamiltonian systems 1}-\eqref{eq:WPRK for stochastic forced Hamiltonian systems 4} for the $4N$ variables $Q^{(0)}_1$, $P^{(0)}_1$, $Q^{(1)}_1$, and $P^{(1)}_1$. However, for several choices of the parameter $\lambda$ the computational cost can be reduced. If $\lambda=0$, then one can first solve the $2N$ equations \eqref{eq:WPRK for stochastic forced Hamiltonian systems 1}-\eqref{eq:WPRK for stochastic forced Hamiltonian systems 2} for the $2N$ variables $Q^{(0)}_1$, $P^{(0)}_1$, and then the $2N$ equations \eqref{eq:WPRK for stochastic forced Hamiltonian systems 3}-\eqref{eq:WPRK for stochastic forced Hamiltonian systems 4} for the remaining $2N$ variables $Q^{(1)}_1$, $P^{(1)}_1$. Moreover, if the Hamiltonians are separable, that is, $H(q,p)=T_0(p)+U_0(q)$ and $h_i(q,p)=T_i(p)+U_i(q)$, then equation \eqref{eq:WPRK for stochastic forced Hamiltonian systems 1} can be substituted into equation \eqref{eq:WPRK for stochastic forced Hamiltonian systems 2}, and equation \eqref{eq:WPRK for stochastic forced Hamiltonian systems 3} can be substitted into equation \eqref{eq:WPRK for stochastic forced Hamiltonian systems 4}, thus reducing the complexity to solving two systems of $N$ equations each. A similar situation occurs for $\lambda=1$. For $\lambda=\frac{1}{2}$ we further have $Q^{(0)}_1=Q^{(1)}_1=(q_k+q_{k+1})/2$ and $P^{(0)}_1=P^{(1)}_1=(p_k+p_{k+1})/2$, and the $SRKw1$ method takes the form of the stochastic midpoint method \eqref{eq:Stochastic midpoint method} with $\Delta W^i$ replaced by $\hat I_i$.
	
	\item \emph{SRKw2} \\ For systems driven by a single noise ($m=1$) we can consider methods with $b^{(3)} \equiv 0$. The family of 4-stage $SRKw2$ methods is defined by the tableau
	
	\begin{equation}
		\label{eq: SRKw2 tableau}
		\begin{array}{c|cccc|cccc}
		& \frac{1}{8} & 0           & 0           & 0           & \frac{5}{6}-\frac{\sqrt{3}}{3}  & -\frac{1}{2} & 0 & 0 \\
		& \frac{1}{4} & \frac{1}{8} & 0           & 0           & -\frac{1}{6}+\frac{\sqrt{3}}{3} & \frac{1}{2}  & 0 & 0 \\
		& \frac{1}{4} & \frac{1}{4} & \frac{1}{8} & 0           & \frac{1}{2}                     & \frac{1}{2}  & 0 & 0 \\
		& \frac{1}{4} & \frac{1}{4} & \frac{1}{4} & \frac{1}{8} & -\frac{1}{6}                    & \frac{1}{2}  & 0 & 0 \\
		\hline
		& -\frac{1}{6}+\frac{\sqrt{3}}{6} & \frac{1}{3}-\frac{\sqrt{3}}{6} & 0 & \frac{1}{3} & \frac{1}{4} & \frac{1}{4}-\frac{\sqrt{3}}{6} & 0 & 0 \\
		& \frac{1}{2} & 0 & 0 & 0 & \frac{1}{4}+\frac{\sqrt{3}}{6} & \frac{1}{4} & 0 & 0 \\
		& 0 & 0 & 0 & 0 & \lambda_1 & \lambda_2 & 0 & \lambda_3 \\
		& 0 & 0 & 0 & 0 & 0 & -\frac{1}{2} & 0 & 0 \\
		\hline
		& \frac{1}{4} & \frac{1}{4} & \frac{1}{4} & \frac{1}{4} & \frac{1}{2} & \frac{1}{2} & 0 & 0
		\end{array},
	\end{equation}
	
	\noindent
	where $\lambda_1, \lambda_2, \lambda_3 \in \mathbb{R}$ are arbitrary parameters. This method is weakly convergent with order 2.0 (see \cite{Rossler2007}, \cite{Wang2017}). Note that $\beta_3=\beta_4=0$, so the values of the internal stages $Q^{(1)}_3$, $Q^{(1)}_4$, $P^{(1)}_3$, and $P^{(1)}_4$ are not needed in \eqref{eq:WPRK for stochastic forced Hamiltonian systems 5} and \eqref{eq:WPRK for stochastic forced Hamiltonian systems 6} to calculate $q_{k+1}$ and $p_{k+1}$, respectively. Moreover, equations \eqref{eq:WPRK for stochastic forced Hamiltonian systems 3} and \eqref{eq:WPRK for stochastic forced Hamiltonian systems 4} for $i=3,4$ are explicit updates, therefore there is no need to solve for or calculate the values of these internal stages. In fact, the choice of the parameters $\lambda_1$, $\lambda_2$, and $\lambda_3$ has no effect on the values of $q_{k+1}$ and $p_{k+1}$, therefore we can set them to zero for convenience. Consequently, the system of equations \eqref{eq:WPRK for stochastic forced Hamiltonian systems 1} and \eqref{eq:WPRK for stochastic forced Hamiltonian systems 2} for $i=1, 2, 3, 4$, and equations \eqref{eq:WPRK for stochastic forced Hamiltonian systems 3} and \eqref{eq:WPRK for stochastic forced Hamiltonian systems 4} for $i=1,2$ ($12N$ equations) has to be solved for the internal stages $Q^{(0)}_1, \ldots, Q^{(0)}_4$, $P^{(0)}_1, \ldots, P^{(0)}_4$, $Q^{(1)}_1$, $Q^{(1)}_2$, $P^{(1)}_1$, and $P^{(1)}_2$ ($12N$ variables). If the Hamiltonians are separable, then equations \eqref{eq:WPRK for stochastic forced Hamiltonian systems 1} and \eqref{eq:WPRK for stochastic forced Hamiltonian systems 3} can be substituted into equations \eqref{eq:WPRK for stochastic forced Hamiltonian systems 2} and \eqref{eq:WPRK for stochastic forced Hamiltonian systems 4}, and the resulting system of $6N$ equations can be solved for $P^{(0)}_1, \ldots, P^{(0)}_4$, $P^{(1)}_1$, and $P^{(1)}_2$ ($6N$ variables).
	
\end{enumerate}

\subsection{Quasi-symplecticity}
\label{sec:Quasi-symplecticity of strong methods}

The idea of quasi-symplectic integrators has been proposed in \cite{MilsteinQuasiSymplectic} as an attempt to construct numerical methods that at least to some extent emulate the special time evolution of the symplectic and volume forms, as pointed out in Theorem~\ref{thm:Conformal symplecticity} and Theorem~\ref{thm:Phase space volume evolution}, respectively. The authors considered a special form of the stochastic forced Hamiltonian system, namely

\begin{align}
\label{eq:Special form of the stochastic form Hamiltonian system}
H(q,p) &= \frac{1}{2}p^TM^{-1}p+U(q), & F(q,p) &= -\Gamma p, \nonumber \\
h_i(q,p) &= -\sigma_i^T q, & f_i(q,p)&=0, \qquad\quad \text{for $i=1,\ldots,m$},
\end{align}

\noindent
where $M$ is an $N\times N$ constant positive definite matrix, $\Gamma$ is an $N\times N$ constant matrix, and $\sigma_i$ are constant vectors. The authors call a numerical integrator $F^+_{t_{k+1},t_k}:(q_k,p_k)\longrightarrow (q_{k+1},p_{k+1})$ quasi-symplectic if it satisfies the following two conditions when applied to the system \eqref{eq:Special form of the stochastic form Hamiltonian system}:

\begin{itemize}
	\item[(QS1)] it degenerates to a symplectic method when the forcing term vanishes, i.e., $\Gamma = 0$, 
	\item[(QS2)] the Jacobian
	
	\begin{equation}
	J \equiv \det DF^+_{t_{k+1},t_k}=\frac{D(q_{k+1},p_{k+1})}{D(q_k,p_k)}
	\end{equation}
	
	\noindent
	does not depend on $q_k$ and $p_k$.
\end{itemize}

\noindent
The condition (QS2) is natural, since the exact Jacobian \eqref{eq:Parameter b_t} does not depend on the phase space variables. Several quasi-symplectic numerical methods have been proposed and tested in \cite{MilsteinQuasiSymplectic}; see also \cite{MilsteinErgodic}. Below we demonstrate that the idea of quasi-symplecticity can be extended to more general systems than \eqref{eq:Special form of the stochastic form Hamiltonian system}.

The methods presented in Section~\ref{sec:Examples of strong methods} and Section~\ref{sec:Weak Examples} preserve the underlying variational structure of the general system \eqref{eq: Stochastic dissipative Hamiltonian system}, as has been shown in Theorem~\ref{thm:Discrete stochastic Lagrange-d'Alembert Principle in Phase Space}. These methods also naturally reduce to symplectic methods, when the forcing terms $F$ and $f_i$ vanish (see \cite{HolmTyranowskiGalerkin}, \cite{MaDing2012}, \cite{MaDing2015}, \cite{MilsteinRepin}, \cite{Wang2017}). Below we show that the St{\"o}rmer-Verlet method satisfies the condition (QS2) for a much broader class of systems than \eqref{eq:Special form of the stochastic form Hamiltonian system}.

\begin{thm}
\label{thm:Quasi-symplecticity of Stormer-Verlet}
Suppose that $H(q,p)$, $F(q,p)$, and $h_i(q,p)$, $f_i(q,p)$ for $i=1,\ldots,m$ satisfy conditions (H1)-(H3). If the Hamiltonians are separable, that is,

\begin{equation}
\label{eq:Separable Hamiltonians for the quasi-symplecticity of Stormer-Verlet}
H(q,p) = T_0(p)+U_0(q), \qquad\qquad h_i(q,p) = T_i(p)+U_i(q), \qquad\qquad i=1,\ldots, m,
\end{equation}

\noindent
and the forcing terms have the form

\begin{equation}
\label{eq:Linear forcing terms for the quasi-symplecticity of Stormer-Verlet}
F(q,p) = -\Gamma_0 p, \qquad\qquad f_i(q,p) = -\Gamma_i p, \qquad\qquad i=1,\ldots, m,
\end{equation}

\noindent
for constant $N \times N$ matrices $\Gamma_i$, then the Jacobian $J$ of the discrete flow $F^+_{t_{k+1},t_k}:(q_k,p_k)\longrightarrow (q_{k+1},p_{k+1})$ defined by the St{\"o}rmer-Verlet method \eqref{eq:Stochastic Stormer-Verlet method} does not depend on $q_k$ and $p_k$, and is almost surely equal to

\begin{equation}
J=\det\bigg(I+\gamma \Big(I-\frac{1}{2}\gamma\Big)^{-1}\bigg),
\end{equation}

\noindent
where $I$ is the $N \times N$ identity matrix, $\gamma = \Delta t \Gamma_0 + \sum_{i=1}^m \Delta W^i \Gamma_i$, and we assume that the matrix $I-\frac{1}{2}\gamma$ is almost surely invertible.
\end{thm}

\begin{proof}
With the separable Hamiltonians \eqref{eq:Separable Hamiltonians for the quasi-symplecticity of Stormer-Verlet} and the linear forcing terms \eqref{eq:Linear forcing terms for the quasi-symplecticity of Stormer-Verlet}, the first equation in \eqref{eq:Stochastic Stormer-Verlet method} is linear, and $P_1$ can be expressed as

\begin{equation}
\label{eq:P1 for quasi-symplecticity}
P_1 = \Big(I-\frac{1}{2}\gamma\Big)^{-1} \bigg(p_k - \frac{1}{2}\Delta t \frac{\partial U_0}{\partial q}(q_k) - \frac{1}{2} \sum_{i=1}^m\Delta W^i \frac{\partial U_i}{\partial q}(q_k) \bigg).
\end{equation}

\noindent
We then plug in $P_1$ into the second and third equations in \eqref{eq:Stochastic Stormer-Verlet method} to obtain expressions for $q_{k+1}$ and $p_{k+1}$ as functions of $q_k$ and $p_k$. Let us introduce the notation

\begin{align}
\eta &= I - \frac{1}{2}\gamma, & A &= \Delta t \frac{\partial^2 T_0}{\partial p^2}(P_1)+\sum_{i=1}^m \Delta W^i \frac{\partial^2 T_i}{\partial p^2}(P_1), \nonumber \\
B &= \Delta t \frac{\partial^2 U_0}{\partial q^2}(q_k)+\sum_{i=1}^m \Delta W^i \frac{\partial^2 U_i}{\partial q^2}(q_k), & C&= \Delta t \frac{\partial^2 U_0}{\partial q^2}(q_{k+1})+\sum_{i=1}^m \Delta W^i \frac{\partial^2 U_i}{\partial q^2}(q_{k+1}).
\end{align}

\noindent
Using this notation, the Jacobian $J$ of the mapping $(q_k,p_k)\longrightarrow (q_{k+1},p_{k+1})$ can be expressed as

\begin{equation}
\label{eq:Jacobian for Stormer-Verlet 1}
J = \begin{vmatrix} 
\frac{\partial q_{k+1}}{\partial q_k} & \frac{\partial q_{k+1}}{\partial p_k} \\
\frac{\partial p_{k+1}}{\partial q_k} & \frac{\partial p_{k+1}}{\partial p_k} 
\end{vmatrix}
=\begin{vmatrix} 
I - \frac{1}{2}A \eta^{-1} B &  A \eta^{-1} \\
-\frac{1}{2}(I+\gamma \eta^{-1})B-\frac{1}{2}C+\frac{1}{4} CA\eta^{-1}B & I-\frac{1}{2}CA\eta^{-1}+\gamma \eta^{-1}
\end{vmatrix}.
\end{equation}

\noindent
Let us transform this determinant into a block upper triangular form by performing basic linear manipulations on its columns and rows. First, multiply the upper and lower right blocks by $\frac{1}{2}B$ on the right, and add the results to the upper and lower left blocks, respectively. Then, multiply the upper left and right blocks by $\frac{1}{2}C$ on the left, and add the results to the lower left and right blocks, thus obtaining a block upper triangular form. Writing out these steps explicitly, we have

\begin{equation}
\label{eq:Jacobian for Stormer-Verlet 2}
J =\begin{vmatrix} 
I  &  A \eta^{-1} \\
-\frac{1}{2}C & I-\frac{1}{2}CA\eta^{-1}+\gamma \eta^{-1}
\end{vmatrix}
=
\begin{vmatrix} 
I  &  A \eta^{-1} \\
0 & I+\gamma \eta^{-1}
\end{vmatrix}
=\det(I+\gamma \eta^{-1}),
\end{equation}

\noindent
which completes the proof.
\end{proof}

\paragraph{Remark.} In case the matrix $\eta = I - \frac{1}{2}\gamma$ is not almost surely invertible, one can replace $\Delta W^i$ with the suitably chosen truncated increments \eqref{eq:Truncated Wiener increments}.

\section{Numerical experiments}
\label{sec:Numerical experiments}

In this section we present the results of our numerical experiments. We have tested the performance of the stochastic Lagrange-d'Alembert integrators presented in Section~\ref{sec:Stochastic Lagrange-d'Alembert variational integrators}, namely the midpoint method \eqref{eq:Stochastic midpoint method}, the St\"{o}rmer-Verlet method \eqref{eq:Stochastic Stormer-Verlet method}, the DIRK method \eqref{eq: Stochastic DIRK method tableau} with $\lambda=1/2$, the $SRKw1$ method \eqref{eq: SRKw1 tableau} with $\lambda=0$, and the $SRKw2$ method \eqref{eq: SRKw2 tableau}, and compared it to the performance of some popular general purpose non-geometric explicit stochastic integrators, namely the mean-square Heun method (\cite{BurragePhDThesis}, \cite{KloedenPlatenSDE}), the mean-square $R2$ and $E1$ methods (see \cite{Burrage1996}, \cite{Burrage1998}, \cite{Burrage2000}, \cite{BurragePhDThesis}), and the weak $RS1$ and $RS2$ methods (\cite{Rossler2007}). The Lagrange-d'Alembert integrators have demonstrated superior behavior in long-time simulations in all of the examples described below. In the case of the midpoint, St\"{o}rmer-Verlet, and DIRK methods, we used unbounded increments $\Delta W^i$, but observed no numerical issues. In principle, one should use truncated increments of the form \eqref{eq:Truncated Wiener increments}, but for the chosen parameters in the examples below, the probability of encountering a singularity was negligible. All computations have been performed in the Julia programming language with the help of the \emph{GeometricIntegrators.jl} library (see \cite{KrausGeometricIntegrators}).

\subsection{Long-time energy behavior}
\label{sec:Long-time energy behavior}

The Kubo oscillator is a stochastic Hamiltonian system with the Hamiltonians given by $H(q,p)=p^2/2+q^2/2$ and $h(q,p)=\beta(p^2/2+q^2/2)$, where $\beta$ is the noise intensity (see \cite{MilsteinRepin}). It is an example of an oscillator with a fluctuating frequency and it was first introduced in the context of the line-shape theory (see \cite{Anderson1954}, \cite{Kubo1954}), but later also found many other applications in connection with mechanical systems, turbulence, laser theory, wave propagation (see \cite{VanKampen1976} and the references therein), magnetic resonance spectroscopy, nonlinear spectroscopy (see \cite{MukamelBook1995} and the references therein), single molecule spectroscopy (\cite{Jung2003}), and stochastic resonance (\cite{ChaudhuriMicroscopic2009}, \cite{Chaudhuri2010}, \cite{ChaudhuriNonequilibrium2009}, \cite{Gitterman2004}). The Kubo oscillator serves as a prototype for multiplicative stochastic processes, and since its solutions can be calculated analytically, it is often used for validation of numerical algorithms (see, e.g., \cite{Fox1987}, \cite{MaDing2015}, \cite{MilsteinRepin}, \cite{SunWang2016}). Here we consider the damped Kubo oscillator with the forcing terms given by $F(q,p)=-\nu p$ and $f(q,p)=-\beta \nu p$, where $\nu$ is the damping coefficient. It is straightforward to verify that the exact solution is given by  

\begin{align}
\label{eq:Damped Kubo oscillator---exact solution}
\bar q(t)&= q_0 e^{-\frac{\nu}{2}(t+\beta W(t))} \cos \omega (t+\beta W(t)) + \frac{1}{\omega}(p_0+\frac{\nu}{2} q_0) e^{-\frac{\nu}{2}(t+\beta W(t))} \sin \omega (t+\beta W(t)), \nonumber \\
\bar p(t)&= p_0 e^{-\frac{\nu}{2}(t+\beta W(t))} \cos \omega (t+\beta W(t)) - \frac{1}{\omega}(q_0+\frac{\nu}{2} p_0) e^{-\frac{\nu}{2}(t+\beta W(t))} \sin \omega (t+\beta W(t)),
\end{align}

\noindent
where $q_0$ and $p_0$ are the initial conditions, the angular frequency is $\omega=\frac{1}{2}\sqrt{4-\nu^2}$, and we have assumed the underdamped case $0\leq \nu < 2$. Note that \eqref{eq:Damped Kubo oscillator---exact solution} is the solution of the deterministic damped harmonic oscillator with the time argument shifted by $\beta W(t)$. Given that $W(t)\sim N(0,t)$ is normally distributed, one can explicitly calculate the expected value of the Hamiltonian $H$ as a function of time as

\begin{align}
\label{eq:The expected value of the Hamiltonian for the Kubo oscillator}
E\Big(H\big(\bar q(t),\bar p(t)\big)\Big)=a e^{-\frac{\nu (2-\beta^2 \nu)}{2}t} + e^{-((2-\nu^2)\beta^2+\nu)t}\Big[ b \cos\big(2 (1-\beta^2 \nu)\omega t \big) + c \sin\big(2 (1-\beta^2 \nu)\omega t \big) \Big],
\end{align}

\noindent
where

\begin{align}
\label{eq:Coefficients in the formula for E(H)}
a=\frac{2 (p_0^2 + q_0^2 + \nu p_0 q_0)}{4 - \nu^2}, \qquad\qquad b=-\frac{\nu^2 (p_0^2 + q_0^2)+4 \nu p_0 q_0}{2 (4 - \nu^2)}, \qquad\qquad c=\frac{\nu (q_0^2-p_0^2)}{2 \sqrt{4 - \nu^2}}.
\end{align}

\noindent
Simulations with the initial conditions $q_0=2$, $p_0=0$, and the parameters $\beta=0.5$ and $\nu=0.001$ were carried out until the time $T=5000$ (approximately 800 periods of the oscillator in the absence of noise). In each case 50000 sample paths were generated. The numerical value of the mean Hamiltonian $E(H)$ as a function of time is depicted in Figure~\ref{fig: Hamiltonian for Kubo for strong integrators} and Figure~\ref{fig: Hamiltonian for Kubo for weak integrators} for the mean-square and weak integrators, respectively. We see that the Lagrange-d'Alembert integrators capture the exponential decay of $E(H)$ very accurately even when relatively large time steps $\Delta t$ are used. The explicit Heun and $R2$ methods fail to reproduce that behavior even for the significantly smaller time step. While the explicit $E1$, $RS1$, and $RS2$ methods capture the qualitative decay of $E(H)$, still much smaller time steps would be needed to reach the level of accuracy of the Lagrange-d'Alembert integrators, thus rendering them inefficient. The accuracy of the Monte Carlo approximation of $E(H)$ at each time step was controlled by estimating the relative error $\sigma(E(H))/E(H)$, where $\sigma(E(H))$ denotes the standard deviation of the mean. The maximum relative error for the St\"{o}rmer-Verlet method was $2.87\cdot 10^{-3}$, and for all other methods it did not exceed $5.26\cdot 10^{-4}$.

\begin{figure}
	\centering
		\includegraphics[width=\textwidth]{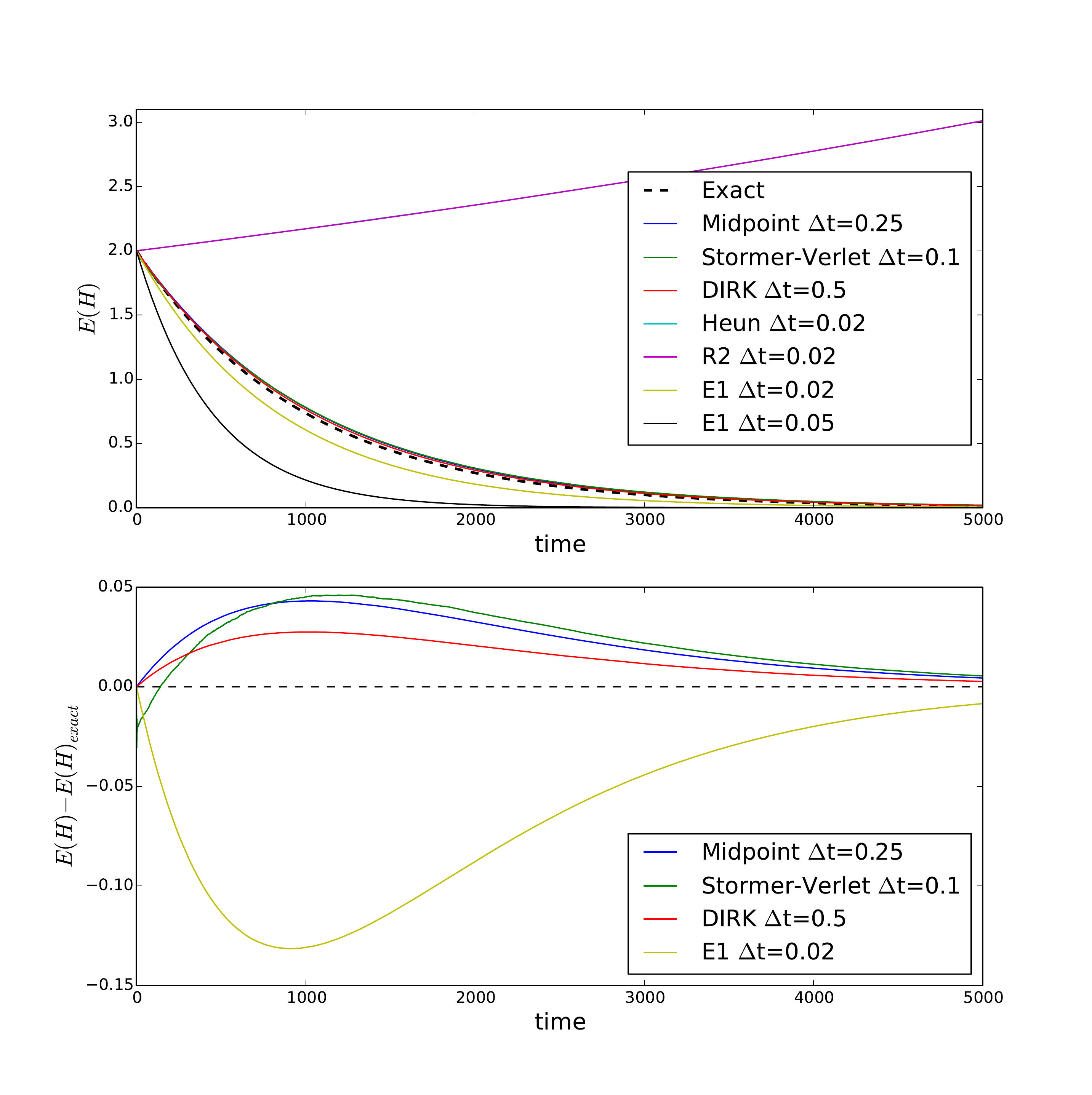}
		\caption{ \emph{Top:} The numerical value of the mean Hamiltonian $E(H)$ for the simulations of the damped Kubo oscillator with the initial conditions $q_0=2$, $p_0=0$, and the parameters $\beta=0.5$ and $\nu=0.001$ is shown for the solutions computed with the mean-square explicit Heun, $R2$, and $E1$ methods, and the mean-square Lagrange-d'Alembert methods presented in Section~\ref{sec:Examples of strong methods}. The Lagrange-d'Alembert integrators accurately capture the exponential decay of $E(H)$, whereas the explicit methods either fail to reproduce that behavior or do so inaccurately. Note that the plots for the Heun and $R2$ methods, as well as for the midpoint and St\"{o}rmer-Verlet methods, overlap very closely. \emph{Bottom:} The difference between the numerical value of the mean Hamiltonian $E(H)$ and the exact value \eqref{eq:The expected value of the Hamiltonian for the Kubo oscillator} is shown for the $E1$ method and the Lagrange-d'Alembert integrators. The stochastic DIRK method proves to be particularly accurate even when the time step $\Delta t=0.5$ is used.}
		\label{fig: Hamiltonian for Kubo for strong integrators}
\end{figure}

\begin{figure}
	\centering
		\includegraphics[width=\textwidth]{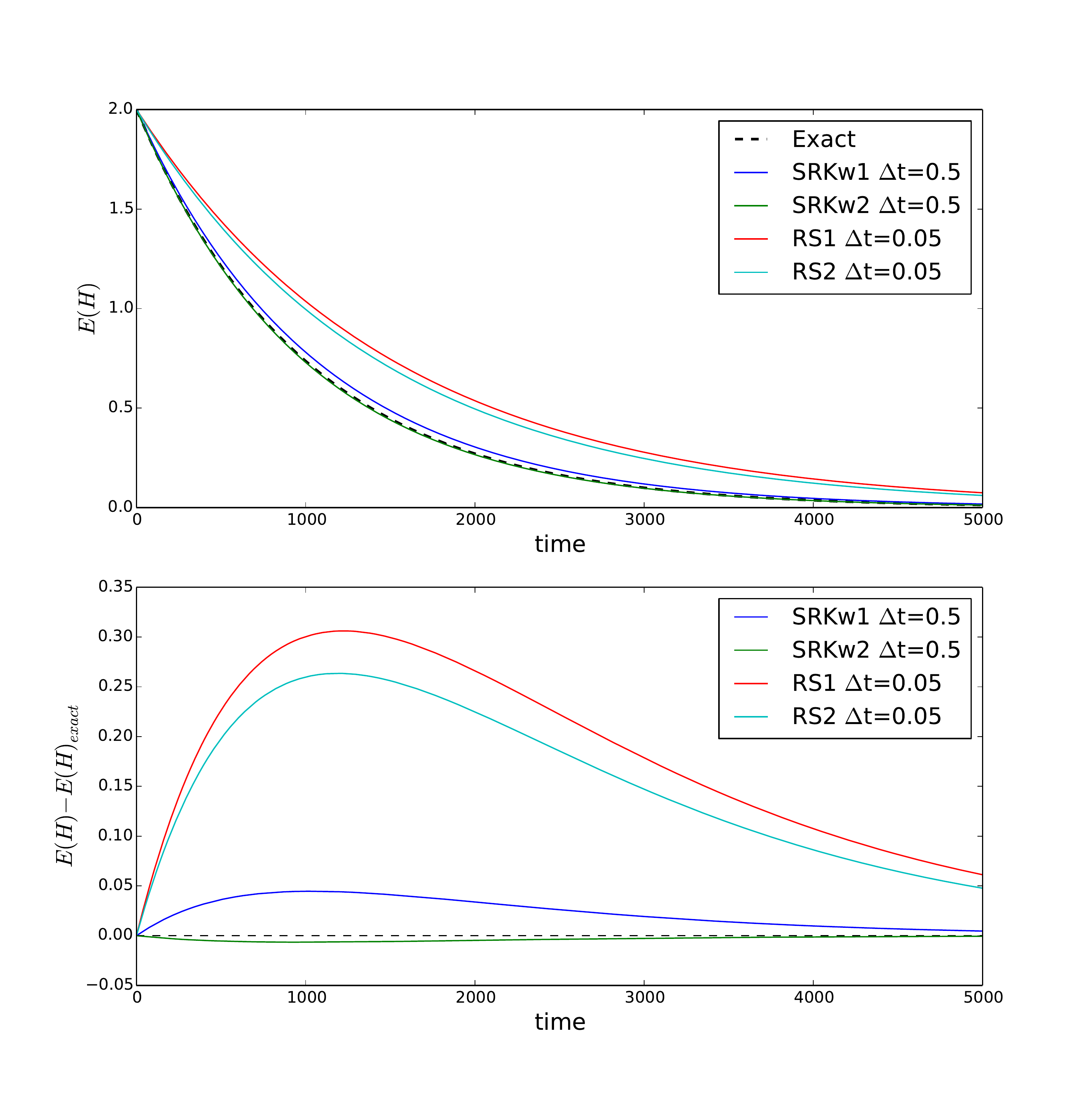}
		\caption{ \emph{Top:} The numerical value of the mean Hamiltonian $E(H)$ for the simulations of the damped Kubo oscillator with the initial conditions $q_0=2$, $p_0=0$, and the parameters $\beta=0.5$ and $\nu=0.001$ is shown for the solutions computed with the weak explicit $RS1$ and $RS2$ methods, and the weak Lagrange-d'Alembert methods presented in Section~\ref{sec:Weak Examples}. The Lagrange-d'Alembert integrators capture the exponential decay of $E(H)$ much more accurately then the explicit ones, even when much larger time steps are used. \emph{Bottom:} The difference between the numerical value of the mean Hamiltonian $E(H)$ and the exact value \eqref{eq:The expected value of the Hamiltonian for the Kubo oscillator} is shown instead. The $SRKw2$ method proves to be significantly more accurate then the others, even when the time step $\Delta t=0.5$ is used.}
		\label{fig: Hamiltonian for Kubo for weak integrators}
\end{figure}

\subsection{Ergodic limits}
\label{sec:Ergodic limits}

In many cases of practical interest the system \eqref{eq: Stochastic dissipative Hamiltonian system} is \emph{ergodic}, which means that

\begin{itemize}
	\item[(1)] it possesses a unique invariant measure represented by the probability density function $\rho_\infty(\xi, \zeta)$ with $(\xi,\zeta) \in T^*Q$, i.e. a stationary solution of the corresponding Fokker-Planck equation (see~\cite{GardinerStochastic})
	\item[(2)] for any function $\varphi: T^*Q \longrightarrow \mathbb{R}$ with polynomial growth at infinity, its ergodic limit, i.e. the expected value with respect to the invariant measure, can be calculated as the limit
	
	\begin{equation}
	\label{eq:Ergodict limit}
	\varphi^{\text{erg}}\equiv \int\int \varphi(\xi,\zeta) \rho_\infty(\xi,\zeta)\,d\xi d\zeta = \lim_{t\rightarrow +\infty} E\Big(\varphi\big(\bar q(t),\bar p(t)\big)\Big),
	\end{equation}
	
	\noindent
	where $(\bar q(t), \bar p(t))$ is an arbitrary solution of \eqref{eq: Stochastic dissipative Hamiltonian system} with arbitrary initial conditions.
\end{itemize}

\noindent
For more information about ergodic systems and ergodic numerical schemes see, e.g., \cite{BouRabeeOwhadiBoltzmannGibbs}, \cite{BouRabeeOwhadi2010}, \cite{Hong2017}, \cite{Mattingly2002}, \cite{Mattingly2010}, \cite{MilsteinErgodic}, \cite{Talay2002}. For many applications, it is interesting to compute the mean of a given function with respect to the invariant law of the diffusion, but the explicit form of the invariant measure is often not known. If the considered system is ergodic, then the ergodic limit can be approximated as 

\begin{equation}
	\label{eq:Ergodict limit approximation}
	\varphi^{\text{erg}} \approx E\Big(\varphi\big(\bar q(T),\bar p(T)\big)\Big)
\end{equation}

\noindent 
by choosing a sufficiently large time $T$. One can then use numerical integrators to approximate $\bar q(T)$ and $\bar p(T)$. However, formula \eqref{eq:Ergodict limit approximation} requires integration of the system over comparatively long time intervals, which poses a significant computational difficulty. Below we compare the performance of the geometric integrators introduced in Section~\ref{sec:Stochastic Lagrange-d'Alembert variational integrators} with the performance of explicit schemes. Note that we do not make any claims about the ergodicity of the used schemes and defer this issue to future work.

In recent years the analysis of nonlinear oscillators subjected to random excitations has been of significant interest, for instance in the context of stochastic resonance and stochastic bifurcation theory. The van der Pol oscillator is one of the most extensively studied systems in nonlinear dynamics and has a long history of being used in physical and biological sciences (see, e.g., \cite{Guckenheimer2002}). It possesses a trivial fixed point and a limit cycle attractor. Various stochastic extensions of the van der Pol oscillator have been considered to test the effect of external noises on its self-sustaining mechanism, the period and lifetime of its oscillations, and the attraction basins of its fixed point and limit cycle (see, e.g., \cite{DoiInoue1998}, \cite{Huang2013}, \cite{Leung1995}, \cite{LiZhu2018}, \cite{LiWu2019}, \cite{SchenkHoppe1996}, \cite{Spigler1985}, \cite{Tel1988}). A numerical study of such stochastic extensions requires long integration times and serves as an interesting testbed for numerical algorithms. Consider van der Pol's equation with additive noise (see \cite{MilsteinErgodic}), which is a stochastic forced Hamiltonian system of the form \eqref{eq: Stochastic dissipative Hamiltonian system} with

\begin{align}
\label{eq:Hamiltonian and forces for Van der Pol's equation}
H(q,p) &= \frac{1}{2}p^2 + \frac{1}{2}q^2, & F(q,p) &= \nu(1-q^2)p, \nonumber \\
h(q,p) &= -\sigma q, & f(q,p)&=0,
\end{align}

\noindent
where $\nu \geq 0$ and $\sigma \geq 0$ are parameters. The explicit form of the invariant measure for this system is unknown, however, it is interesting to compute the ergodic value of the energy. Note that the forcing term $F(q,p)$ is not globally Lipschitz, therefore this example also tests the Lagrange-d'Alembert integrators in the situation when the assumption (H3) from Section~\ref{sec:Stochastic Lagrange-d'Alembert principle} is not satisfied. Simulations with the initial conditions $q_0=1$, $p_0=1$, and the parameters $\sigma=0.05$ and $\nu=0.001$ were carried out until the time $T=5000$. In each case $10^6$ sample paths were generated. The numerical value of the mean Hamiltonian $E(H)$ as a function of time is depicted in Figure~\ref{fig: Hamiltonian for Van der Pol for DIRK} for the DIRK, Heun, and $E1$ methods. As the reference value we take $H^{\text{erg}}=2.3165$, which was calculated in \cite{MilsteinErgodic} using a second-order weak quasi-symplectic method at the time $T_{\text{ref}}=10000$ with the time step $\Delta t=0.05$ and $4 \times 10^6$ sample paths. We see that the DIRK method accurately reproduces the reference value even with the relatively large time step $\Delta t = 0.2$, while the Heun and $E1$ methods require the much smaller time step $\Delta t = 0.02$ to reach that level of accuracy. The situation is similar for the other Lagrange-d'Alembert and explicit integrators. Figure~\ref{fig: Hamiltonian for Van der Pol - close up for all} depicts the behavior of $E(H)$ near the reference value on the time interval $[4000,5000]$ for each of the tested integrators. The maximum relative Monte Carlo error $\sigma(E(H))/E(H)$ did not exceed $7.17\cdot 10^{-4}$ in any of the simulations.

\begin{figure}
	\centering
		\includegraphics[width=\textwidth]{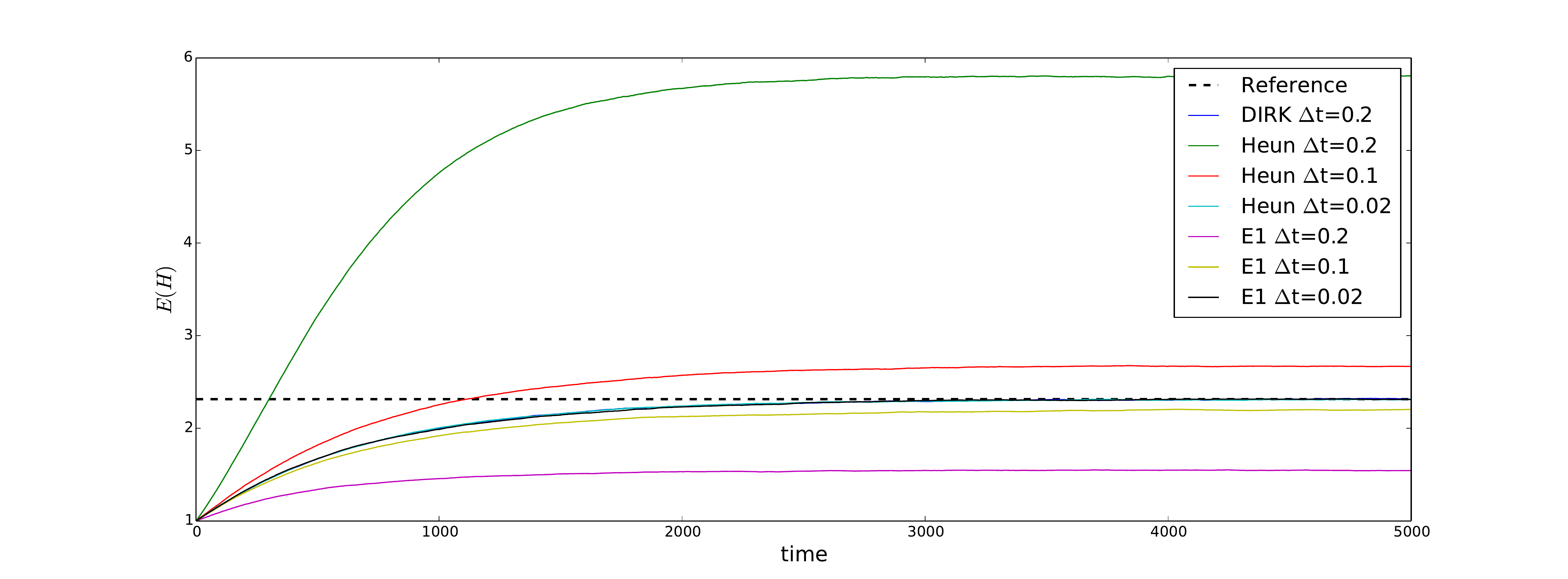}
		\caption{The numerical value of the mean Hamiltonian $E(H)$ as a function of time for the simulations of Van der Pol's equation with the initial conditions $q_0=1$, $p_0=1$, and the parameters $\sigma=0.05$ and $\nu=0.001$, is shown for the solutions computed with the DIRK, Heun, and $E1$ methods. The reference value $H^{\text{erg}}=2.3165$ was calculated in \cite{MilsteinErgodic}. The DIRK method accurately reproduces the reference value even with the relatively large time step $\Delta t = 0.2$, while the Heun and $E1$ methods require the much smaller time step $\Delta t = 0.02$ to reach that level of accuracy. For the clarity of the plot the other Lagrange-d'Alembert and explicit methods are not depicted, but they demonstrate similar behavior. Note that the plots for the DIRK method, and the Heun and $E1$ methods with $\Delta t=0.02$ overlap very closely.}
		\label{fig: Hamiltonian for Van der Pol for DIRK}
\end{figure}

\begin{figure}
	\centering
		\includegraphics[width=\textwidth]{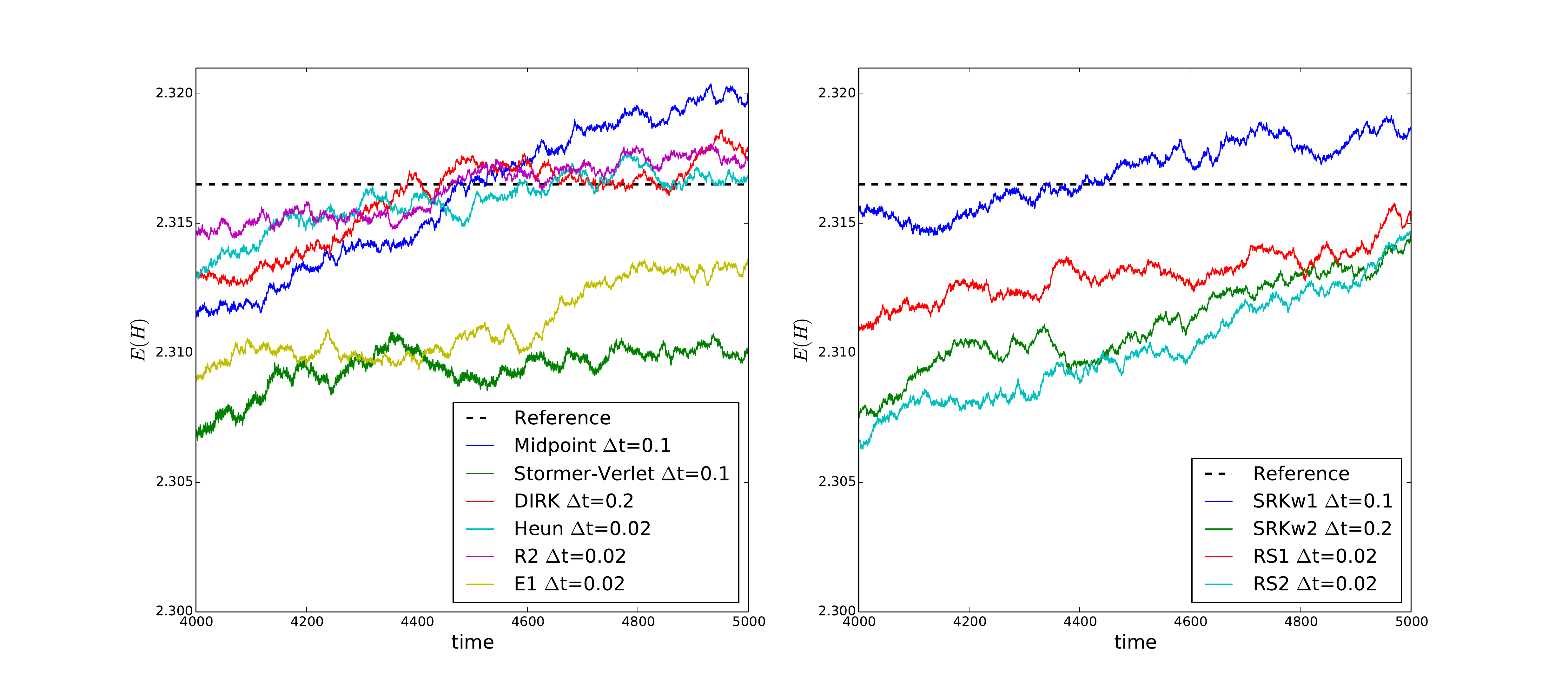}
		\caption{The numerical value of the mean Hamiltonian $E(H)$ for the simulations of Van der Pol's equation with the initial conditions $q_0=1$, $p_0=1$, and the parameters $\sigma=0.05$ and $\nu=0.001$, is shown on the time interval $[4000,5000]$ near the reference value for the solutions computed with the mean-square (\emph{Left}) and weak (\emph{Right}) methods. The reference value $H^{\text{erg}}=2.3165$ was calculated in \cite{MilsteinErgodic}.}
		\label{fig: Hamiltonian for Van der Pol - close up for all}
\end{figure}

\subsection{Vlasov equation}
\label{sec:Vlasov equation}

In recent years there has been a growing interest in applying geometric integration to particle-in-cell (PIC) simulations of the Vlasov equation in plasma physics. The results to date concern almost entirely collisionless cases (see \cite{Burby2017}, \cite{Evstatiev2013}, \cite{He2016}, \cite{Qin2016}, \cite{Shadwick2014}, \cite{Squire2012}, \cite{Stamm2014}, \cite{Xiao2015}). The first step towards a geometric description of collision operators, using the so-called metriplectic formulation, has been recently made in \cite{HirvijokiKrausMetriplectic}. Below we demonstrate that stochastic forced Hamiltonian systems provide an alternative structure-preserving description, and further consider two examples, namely the Lenard-Bernstein and the Lorentz collision operators, to test the long-time behavior of the stochastic Lagrange-d'Alembert integrators.

\subsubsection{Lenard-Bernstein collision operator}
\label{seq:Lenard-Bernstein collision operator}

The following two-dimensional Vlasov-Fokker-Planck equation 

\begin{equation}
\label{eq:Vlasov-Fokker-Planck equation}
\frac{\partial \rho}{\partial t}+v \frac{\partial \rho}{\partial x}-E(x)\frac{\partial \rho}{\partial v}=\nu \bigg( \mu \frac{\partial (v\rho)}{\partial v} + \frac{D^2}{2}\frac{\partial^2 \rho}{\partial v^2} \bigg)
\end{equation}

\noindent
has been studied in \cite{Kleiber2011} and \cite{Sonnendrucker2015} as a model for collisional kinetic plasmas, where $\rho=\rho(x,v,t)$ denotes the particle distribution function in the position-velocity phase space, $E(x)=-\phi'(x)$ is the external electric field with the electrostatic potential $\phi(x)$, and $\nu >0$, $\mu >0$, $D>0$ are real parameters. The right-hand side of \eqref{eq:Vlasov-Fokker-Planck equation} is the so-called Lenard-Bernstein collision operator, which models small-angle collisions and was originally used to study longitudinal plasma oscillations (see \cite{LenardBernstein1958}). A stochastic split particle-in-cell (PIC) method for the numerical simulation of \eqref{eq:Vlasov-Fokker-Planck equation} has been proposed in \cite{Sonnendrucker2015}, whereby the advection part is solved using the standard PIC method, and the diffusion part is modeled by a stochastic differential equation. Below we demonstrate a structure-preserving approach to solving \eqref{eq:Vlasov-Fokker-Planck equation}. When $\rho$ is interpreted as a probability density function, then \eqref{eq:Vlasov-Fokker-Planck equation} is the Fokker-Planck equation for the two-dimensional stochastic process $(X(t),V(t))$ whose evolution is governed by the stochastic differential equation (see \cite{GardinerStochastic}, \cite{KloedenPlatenSDE})

\begin{align}
\label{eq:SDE for the Vlasov equation}
d_t X = V \,dt, \qquad\qquad d_t V = (-E(X)-\nu \mu V) \,dt + \sqrt{\nu}D\circ dW(t),
\end{align}

\noindent
driven by the one-dimensional Wiener process $W(t)$. This equation is a stochastic forced Hamiltonian system \eqref{eq: Stochastic dissipative Hamiltonian system} with

\begin{align}
\label{eq:Hamiltonian and forces for the Vlasov equation}
H(X,V) &= \frac{1}{2}V^2 - \phi(X), & F(X,V) &= -\nu\mu V, \nonumber \\
h(X,V) &= -\sqrt{\nu}DX, & f(X,V)&=0.
\end{align}

\noindent
It can be easily verified that the stationary solution of \eqref{eq:Vlasov-Fokker-Planck equation} is given by the \emph{Gibbs measure}

\begin{equation}
\label{eq:Gibbs measure}
\rho_\infty(x,v) = \frac{1}{Z}e^{-\frac{2 \mu}{D^2}H(x,v)} = \frac{1}{Z}e^{\frac{2 \mu}{D^2}\phi(x)}e^{-\frac{\mu}{D^2}v^2},
\end{equation}

\noindent
where $Z$ is the normalizing constant such that $\int \int \rho_\infty(x,v) \, dv\,dx = 1$. Let us consider \eqref{eq:Vlasov-Fokker-Planck equation} on the domain $(x,v) \in [0,1]\times \mathbb{R}$ with periodic boundary conditions in $x$, and with the electrostatic potential

\begin{equation}
\label{eq:Electrostatic potential}
\phi(x) = -\frac{E_0}{4 \pi} \sin 4 \pi x,
\end{equation}

\noindent
where $E_0>0$ is the maximum magnitude of the electric field $E(x)=-\phi'(x)$. One can check that the system \eqref{eq:Hamiltonian and forces for the Vlasov equation} with the potential \eqref{eq:Electrostatic potential} is ergodic (see Theorem~3.2 in \cite{Mattingly2002}). As the initial condition, we take the probability distribution of the form

\begin{equation}
\label{eq:Initial conditions for the probability density}
\rho(x,v,0) = \rho_X(x) \rho_V(v) = (1+\epsilon \cos 2\pi x) \bigg( \frac{1}{1+a}\frac{1}{\sqrt{2 \pi}}e^{-\frac{1}{2}v^2} + \frac{a}{1+a} \frac{1}{\sqrt{2 \pi} \sigma}e^{-\frac{1}{2\sigma^2}(v-v_0)^2} \bigg),
\end{equation}

\noindent
where $\rho_X(x)$ for $\epsilon > 0$ describes a perturbation of the uniform distribution along the spatial direction $x$, and $\rho_V(v)$ for $a>0$ is the so called \emph{bump-on-tail} distribution in the velocity space, where the bump is centered at $v_0$ with the standard deviation $\sigma >0$. Simulations with the parameters $\nu=0.01$, $\mu=1$, $D=\sqrt{2}$, $E_0=3$, $\epsilon=0.25$, $a=0.5$, $v_0=4$, and $\sigma=0.5$ were carried out until the time $T=1000$. In each case $10^7$ sample paths were generated. The initial conditions $X_0$ and $V_0$ were randomly drawn from the probability distribution \eqref{eq:Initial conditions for the probability density} using rejection sampling (see Figure~\ref{fig: Initial probability density for the Vlasov equation}). The exact ergodic value $H^\text{erg}$ of the Hamiltonian can be calculated using the invariant probability density \eqref{eq:Gibbs measure} as

\begin{figure}
	\centering
		\includegraphics[width=\textwidth]{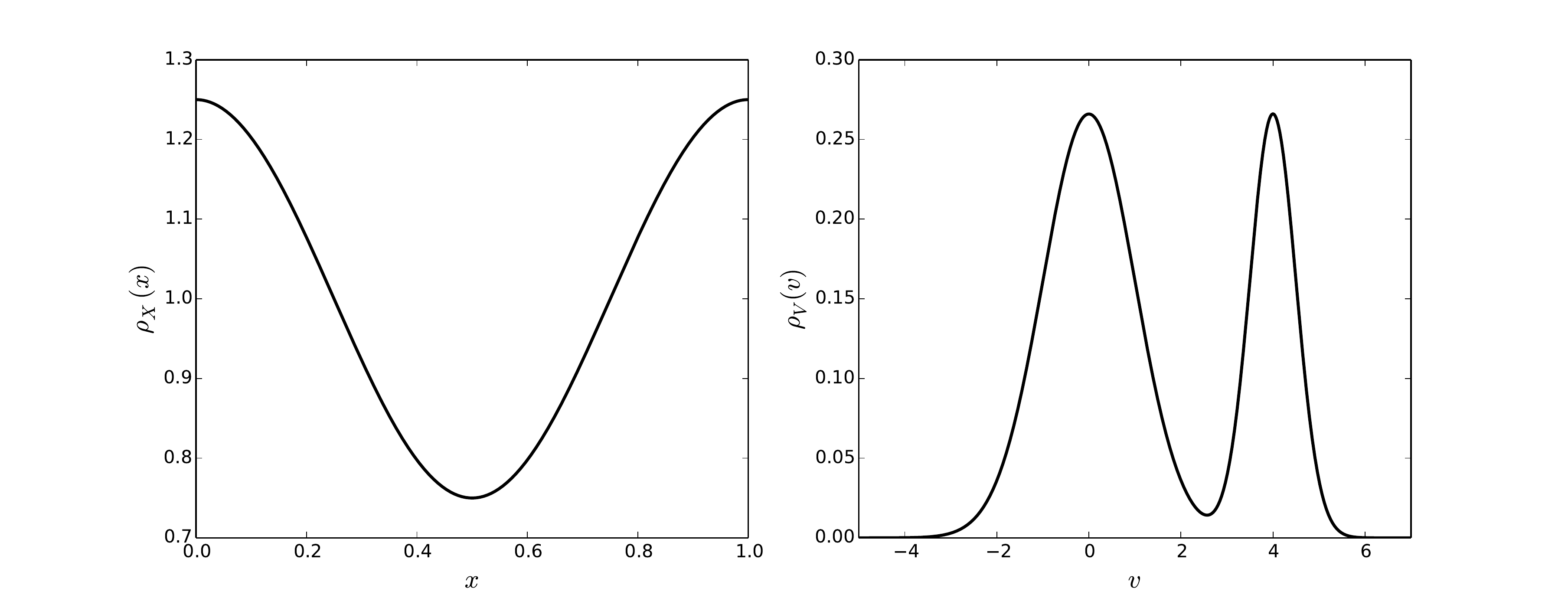}
		\caption{The initial probability density \eqref{eq:Initial conditions for the probability density} with the parameters $\epsilon=0.25$, $a=0.5$, $v_0=4$, and $\sigma=0.5$ for the simulations of the Vlasov equation with the Lenard-Bernstein collision operator.}
		\label{fig: Initial probability density for the Vlasov equation}
\end{figure}

\begin{equation}
\label{eq:Ergodic value of the Hamiltonian for the Vlasov equation}
H^\text{erg} = \int_0^1 \int_{-\infty}^\infty H(x,v) \rho_\infty(x,v) \,dv\,dx \approx 0.471705.
\end{equation}

\begin{figure}
	\centering
		\includegraphics[width=0.95\textwidth]{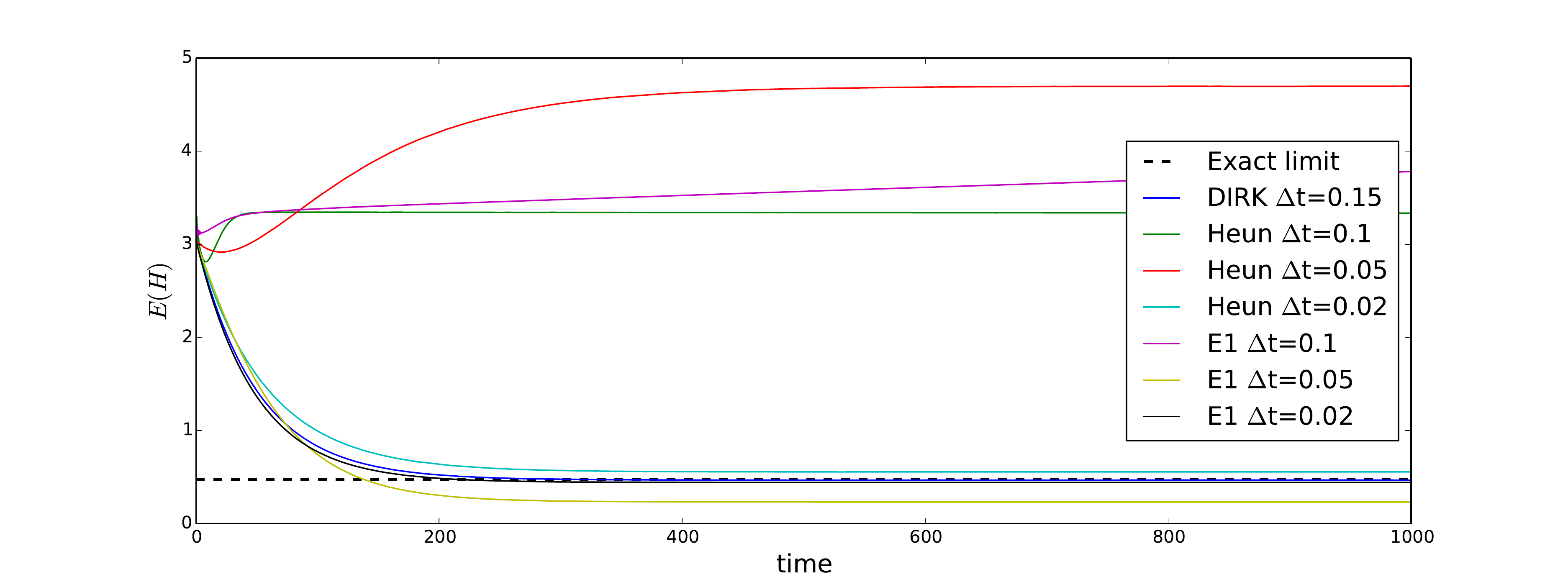}
		\caption{The numerical value of the mean Hamiltonian $E(H)$ as a function of time for the simulations of the Vlasov equation with the Lenard-Bernstein collision operator with the parameters $\nu=0.01$, $\mu=1$, $D=\sqrt{2}$, $E_0=3$, $\epsilon=0.25$, $a=0.5$, $v_0=4$, and $\sigma=0.5$, and the initial conditions $X_0$ and $V_0$ sampled from the distribution \eqref{eq:Initial conditions for the probability density}, is shown for the solutions computed with the DIRK, Heun, and $E1$ methods. The exact ergodic limit $H^{\text{erg}}\approx 0.471705$ was calculated in \eqref{eq:Ergodic value of the Hamiltonian for the Vlasov equation}. The DIRK method accurately reproduces the ergodic limit even with the relatively large time step $\Delta t = 0.15$, while the $E1$ method requires the much smaller time step $\Delta t = 0.02$ to reach a comparable level of accuracy. The Heun method yields a less accurate result even for $\Delta t = 0.02$.  For the clarity of the plot the other Lagrange-d'Alembert and explicit methods are not depicted, but they demonstrate similar behavior. Note that the plots for the DIRK method and the $E1$ method with $\Delta t=0.02$ overlap very closely.}
		\label{fig: Hamiltonian for Vlasov for DIRK}
\end{figure}
\begin{figure}
	\centering
		\includegraphics[width=0.9\textwidth]{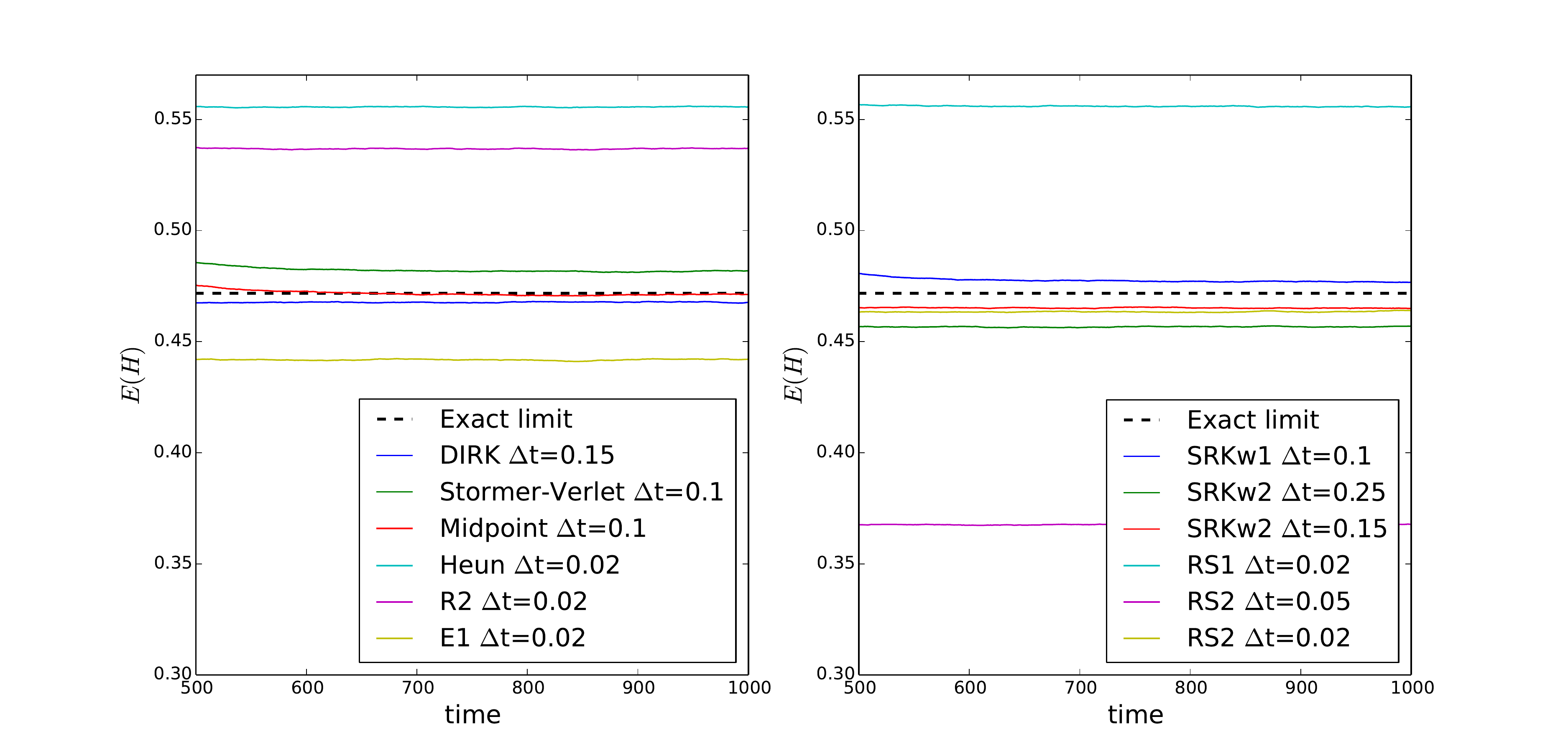}
		\caption{The numerical value of the mean Hamiltonian $E(H)$ for the simulations of the Vlasov equation with the Lenard-Bernstein collision operator with the parameters $\nu=0.01$, $\mu=1$, $D=\sqrt{2}$, $E_0=3$, $\epsilon=0.25$, $a=0.5$, $v_0=4$, and $\sigma=0.5$, and the initial conditions $X_0$ and $V_0$ sampled from the distribution \eqref{eq:Initial conditions for the probability density}, is shown on the time interval $[500,1000]$ near the exact ergodic limit for the solutions computed with the mean-square (\emph{Left}) and weak (\emph{Right}) methods. The exact ergodic limit $H^{\text{erg}}\approx 0.471705$ was calculated in \eqref{eq:Ergodic value of the Hamiltonian for the Vlasov equation}.}
		\label{fig: Hamiltonian for Vlasov - close up for all}
\end{figure}

\begin{figure}
	\centering
		\includegraphics[width=\textwidth]{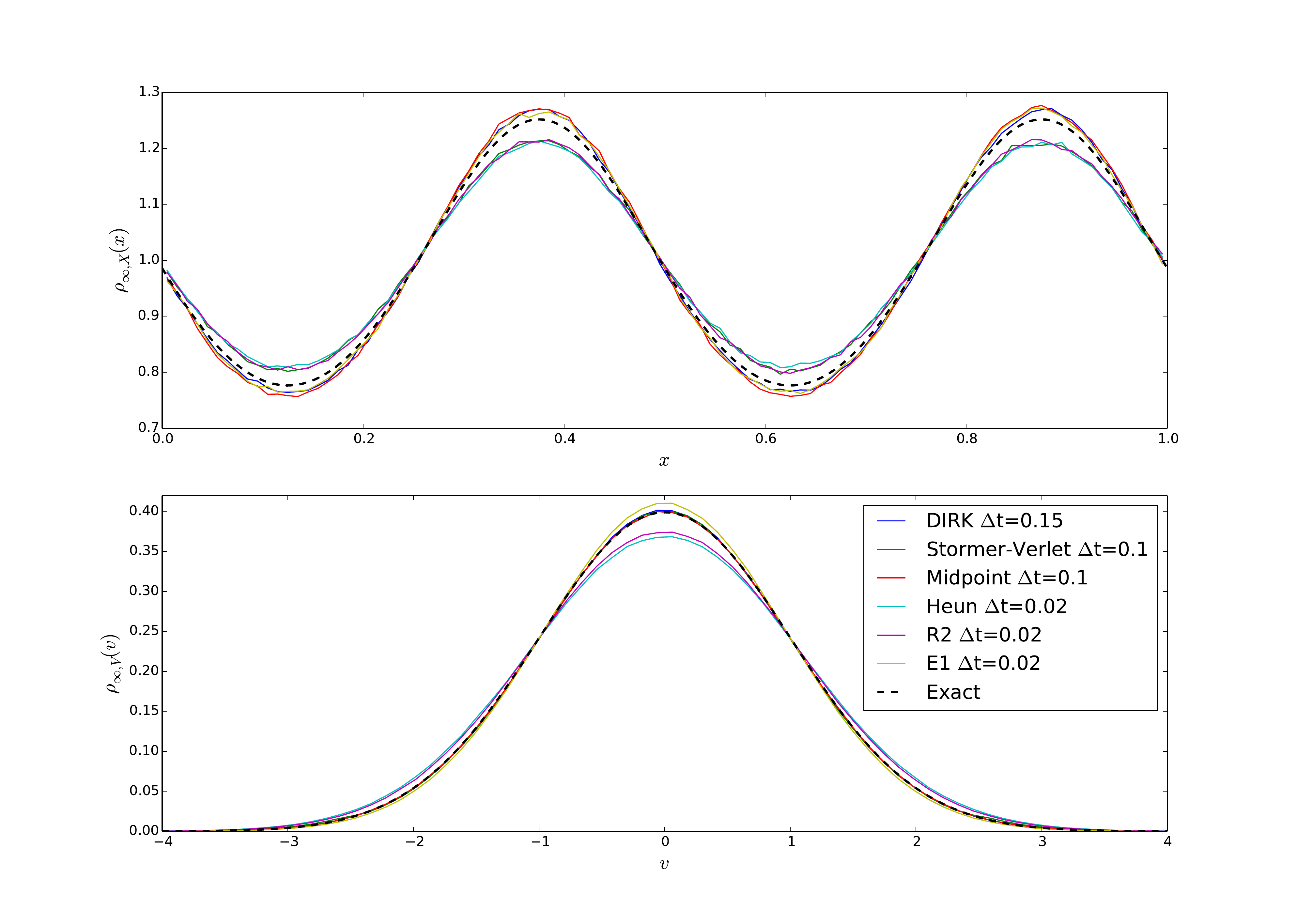}
		\caption{The numerical probability density at time $T=1000$ for the simulations of the Vlasov equation with the Lenard-Bernstein collision operator with the parameters $\nu=0.01$, $\mu=1$, $D=\sqrt{2}$, $E_0=3$, $\epsilon=0.25$, $a=0.5$, $v_0=4$, and $\sigma=0.5$, and the initial conditions $X_0$ and $V_0$ sampled from the distribution \eqref{eq:Initial conditions for the probability density}, is depicted for each of the mean-square integrators, and compared to the exact invariant measure \eqref{eq:Gibbs measure}. Note that in the top figure the plots for the DIRK, midpoint, and $E1$ methods, as well as the plots for the St\"{o}rmer-Verlet, Heun and $R2$ methods, overlap very closely. In the bottom figure the plots for the DIRK, St\"{o}rmer-Verlet, and midpoint methods overlap very closely with the exact solution. }
		\label{fig: Final density for Vlasov - mean-square integrators}
\end{figure}

\begin{figure}
	\centering
		\includegraphics[width=\textwidth]{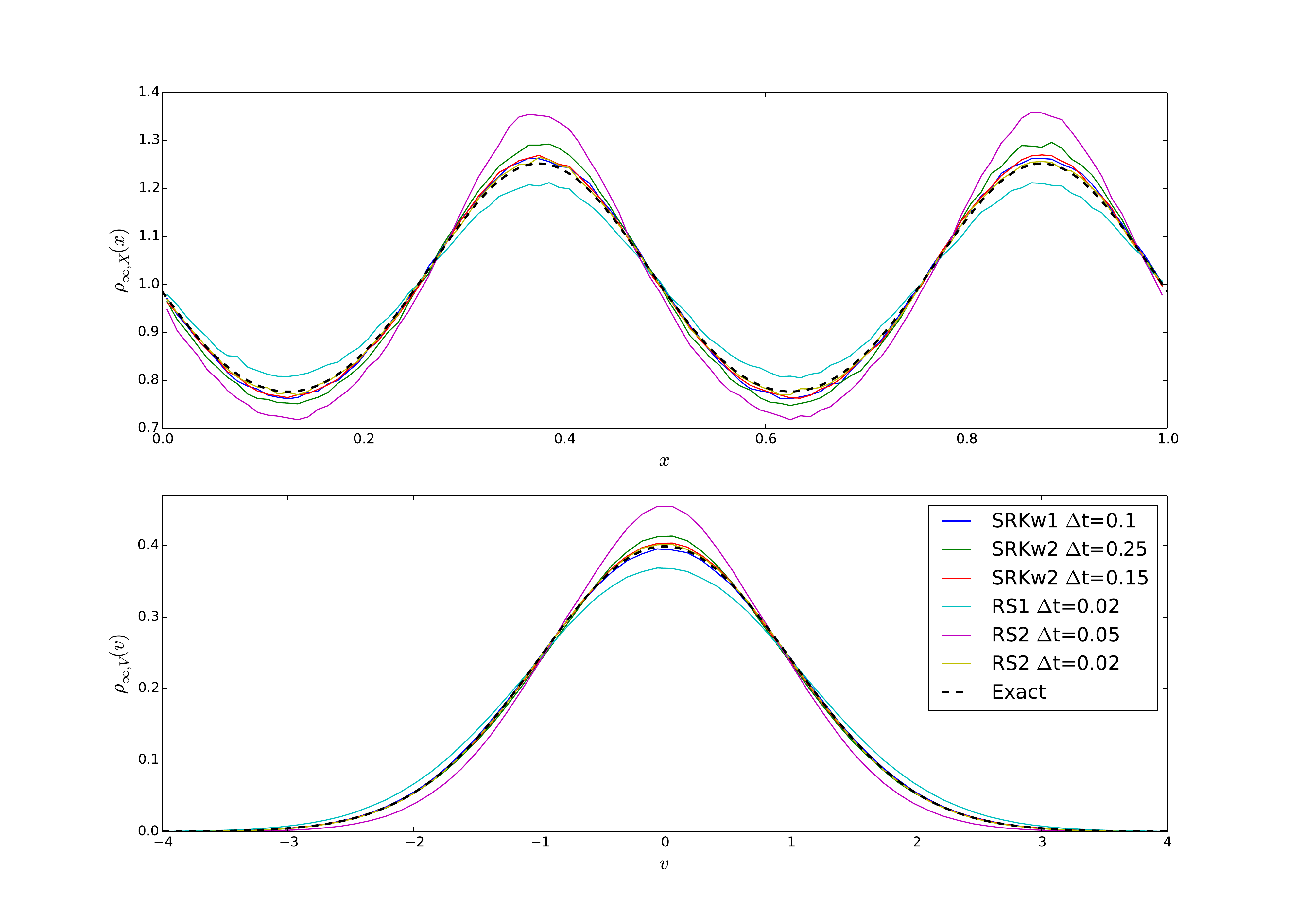}
		\caption{The numerical probability density at time $T=1000$ for the simulations of the Vlasov equation with the Lenard-Bernstein collision operator with the parameters $\nu=0.01$, $\mu=1$, $D=\sqrt{2}$, $E_0=3$, $\epsilon=0.25$, $a=0.5$, $v_0=4$, and $\sigma=0.5$, and the initial conditions $X_0$ and $V_0$ sampled from the distribution \eqref{eq:Initial conditions for the probability density}, is depicted for each of the weak integrators, and compared to the exact invariant measure \eqref{eq:Gibbs measure}. Note that the plots for the $SRKw1$ method, $SRKw2$ method with $\Delta t = 0.15$, and the $RS2$ method with $\Delta t = 0.02$ in the top figure, as well as the plots for the $SRKw1$ method, $SRKw2$ method with $\Delta t = 0.25$, $SRKw2$ method with $\Delta t = 0.15$, and the $RS2$ method with $\Delta t = 0.02$ in the bottom figure, overlap very closely with the exact solution.}
		\label{fig: Final density for Vlasov - weak integrators}
\end{figure}

\noindent
The numerical value of the mean Hamiltonian $E(H)$ as a function of time is depicted in Figure~\ref{fig: Hamiltonian for Vlasov for DIRK} for the DIRK, Heun, and $E1$ methods. We see that the DIRK method accurately reproduces the ergodic limit even with the relatively large time step $\Delta t = 0.15$, while the $E1$ method requires the much smaller time step $\Delta t = 0.02$ to reach a comparable level of accuracy. The Heun method yields a less accurate result even for $\Delta t = 0.02$. The situation is similar for the other Lagrange-d'Alembert and explicit integrators. Figure~\ref{fig: Hamiltonian for Vlasov - close up for all} depicts the behavior of $E(H)$ near the exact ergodic limit on the time interval $[500,1000]$ for each of the tested integrators. The maximum relative Monte Carlo error $\sigma(E(H))/E(H)$ did not exceed $8.24\cdot10^{-4}$ in any of the simulations. The numerical probability density at the final time $T=1000$ calculated with each of the mean-square and weak methods is depicted in comparison to the exact invariant measure \eqref{eq:Gibbs measure} in Figure~\ref{fig: Final density for Vlasov - mean-square integrators} and Figure~\ref{fig: Final density for Vlasov - weak integrators}, respectively.

\subsubsection{Lorentz collision operator}
\label{seq:Lorentz collision operator}

The following four-dimensional Vlasov-Fokker-Planck equation 

\begin{equation}
\label{eq:Vlasov-Fokker-Planck equation with Lorentz operator}
\frac{\partial \rho}{\partial t}+v_x \frac{\partial \rho}{\partial x}+v_y \frac{\partial \rho}{\partial y}-E_x(x,y)\frac{\partial \rho}{\partial v_x}-E_y(x,y)\frac{\partial \rho}{\partial v_y}=\nu \bigg( v_x \frac{\partial}{\partial v_y} - v_y \frac{\partial}{\partial v_x} \bigg)^2 \rho
\end{equation}

\noindent
has been used in \cite{Banks2016} to study the electron-ion collision effects on the damping of electron plasma waves, where $\rho=\rho(x,y,v_x,v_y,t)$ denotes the particle distribution function in the position-velocity phase space, $E_x(x,y)=-\frac{\partial \phi}{\partial x}(x,y)$ and $E_y(x,y)=-\frac{\partial \phi}{\partial y}(x,y)$ are the components of the external electric field with the electrostatic potential $\phi(x,y)$, and $\nu >0$ is a real parameter. The right-hand side of \eqref{eq:Vlasov-Fokker-Planck equation with Lorentz operator} is the so-called Lorentz collision operator, which models electron-ion interactions via pitch-angle scattering. The primary effect of this type of scattering is a change of the direction of the electron's velocity with negligible energy loss. More information about the Lorentz collision operator can be found in \cite{Karney1986}. Below we demonstrate a structure-preserving approach to solving \eqref{eq:Vlasov-Fokker-Planck equation with Lorentz operator}. When $\rho$ is interpreted as a probability density function, then \eqref{eq:Vlasov-Fokker-Planck equation with Lorentz operator} is the Fokker-Planck equation for the four-dimensional stochastic process $(X(t), Y(t), V_x(t), V_y(t))$ whose evolution is governed by the stochastic differential equation (see \cite{GardinerStochastic}, \cite{KloedenPlatenSDE})

\begin{align}
\label{eq:SDE for the Vlasov equation with Lorentz operator}
d_t X &= V_x \,dt, & d_t V_x &= -E_x(X,Y)\,dt + \sqrt{2\nu}V_y\circ dW(t), \nonumber \\
d_t Y &= V_y \,dt, & d_t V_y &= -E_y(X,Y)\,dt - \sqrt{2\nu}V_x\circ dW(t),
\end{align}

\noindent
driven by the one-dimensional Wiener process $W(t)$. This equation is a stochastic forced Hamiltonian system \eqref{eq: Stochastic dissipative Hamiltonian system} with

\begin{align}
\label{eq:Hamiltonian and forces for the Vlasov equation with Lorentz operator}
H(X,Y,V_x,V_y) &= \frac{1}{2}V_x^2 + \frac{1}{2}V_y^2 - \phi(X,Y), & F(X,Y,V_x,V_y) &= (0,0), \nonumber \\
h(X,Y,V_x,V_y) &= 0, & f(X,Y,V_x,V_y)&=\big(\sqrt{2 \nu} V_x,-\sqrt{2 \nu} V_y\big).
\end{align}

\noindent
Let us consider \eqref{eq:Vlasov-Fokker-Planck equation with Lorentz operator} on the domain $(x,y,v_x,v_y) \in [0,1]^2\times \mathbb{R}^2$ with periodic boundary conditions in $x$ and $y$, and with the electrostatic potential

\begin{equation}
\label{eq:Electrostatic potential for Lorentz operator}
\phi(x,y) = -\frac{E_0}{4 \pi} \sin 4 \pi x \, \sin 4 \pi y,
\end{equation}

\noindent
where $E_0>0$ is the maximum magnitude of the electric field $E(x,y)=-\nabla\phi(x,y)$. As the initial condition, we take the probability distribution of the form

\begin{equation}
\label{eq:Initial conditions for the probability density for Lorentz operator}
\rho(x,y,v_x,v_y,0) = \frac{1}{2 \pi}(1+\epsilon_1 \cos 2\pi x) (1+\epsilon_2 \cos 2\pi y) e^{-\frac{v_x^2+v_y^2}{2}} ,
\end{equation}

\noindent
where the parameters $\epsilon_1, \epsilon_2 > 0$ describe a perturbation of the uniform distribution along the spatial directions $x$ and $y$, and the velocity part is Maxwellian. The Lorentz collision operator by construction preserves the total energy of the plasma, that is,

\begin{equation}
\label{eq:Preservation of energy for Lorentz operator}
E(H) = \int_0^1 \int_0^1 \int_{-\infty}^\infty \int_{-\infty}^\infty H(x,y,v_x,v_y) \rho(x,y,v_x,v_y,t)\,dv_y dv_x dy dx = \text{const},
\end{equation}

\noindent
where $E(H) \equiv E\Big( H\big( X(t),Y(t),V_x(t),V_y(t) \big) \Big)$ for short (see \cite{Karney1986}). Moreover, in the stochastic description \eqref{eq:SDE for the Vlasov equation with Lorentz operator} the Hamiltonian is almost surely preserved for each sample path, which can be easily verified by calculating the stochastic differential

\begin{equation}
\label{eq:Stochastic differential dH}
dH = \frac{\partial H}{\partial X} \circ dX + \frac{\partial H}{\partial Y} \circ dY + \frac{\partial H}{\partial V_x} \circ dV_x + \frac{\partial H}{\partial V_y} \circ dV_y = 0,
\end{equation} 

\noindent
where the last equality follows from \eqref{eq:Hamiltonian and forces for the Vlasov equation with Lorentz operator} and \eqref{eq:SDE for the Vlasov equation with Lorentz operator}. 
\begin{figure}
	\centering
		\includegraphics[width=\textwidth]{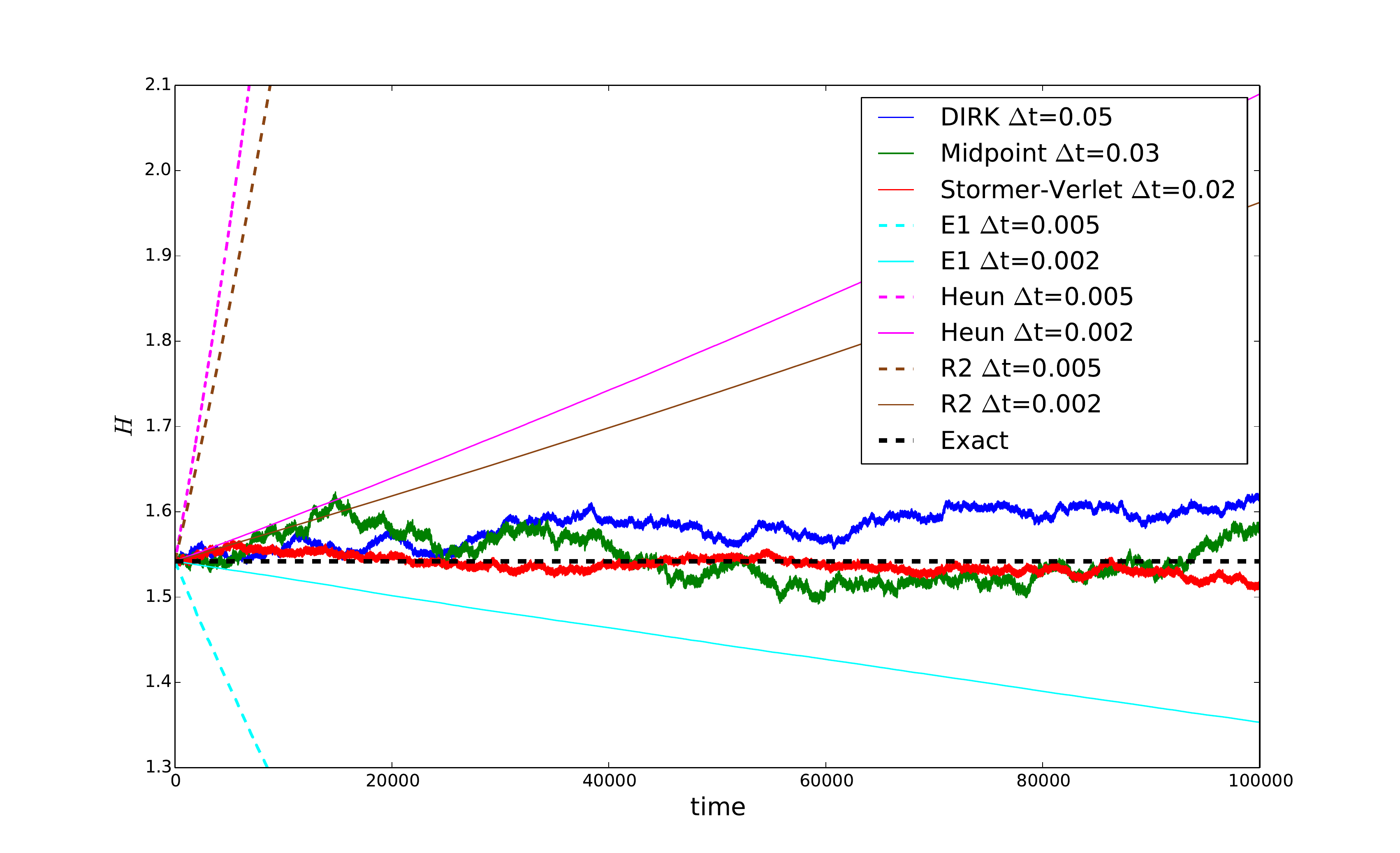}
		\caption{The numerical value of the Hamiltonian $H$ as a function of time for the simulations of the Vlasov equation with the Lorentz collision operator with the parameters $\nu=0.005$, $E_0=3$, $\epsilon_1=\epsilon_2=0.25$, and a single random initial condition sampled from the distribution \eqref{eq:Initial conditions for the probability density for Lorentz operator}, is shown for the solutions computed with the mean-square methods. For each integrator the same random initial condition and the same realization of the Brownian motion were used. The DIRK, midpoint, and St\"{o}rmer-Verlet methods accurately reproduce the conservation of energy over a long integration time, while the $E1$, $R2$, and Heun methods fail to do so even when significantly smaller time steps are used.}
		\label{fig: Energy for Vlasov with Lorentz operator - strong integrators}
\end{figure}
Mean-square integrators aim to approximate each sample path of the exact solution, and therefore they should also approximate the stronger energy preservation condition \eqref{eq:Stochastic differential dH}. In order to test the long-time performance of the mean-square integrators discussed in Section~\ref{sec:Examples of strong methods}, simulations with the parameters $\nu=0.005$, $E_0=3$, and $\epsilon_1=\epsilon_2=0.25$ were carried out for a single sample path until the time $T=100000$. For each integrator the same random initial condition, drawn from the probability distribution \eqref{eq:Initial conditions for the probability density for Lorentz operator}, and the same realization of the Brownian motion were used. The numerical value of the Hamiltonian $H$ as a function of time is depicted in Figure~\ref{fig: Energy for Vlasov with Lorentz operator - strong integrators}. Even with relatively large time steps the mean-square Lagrange-d'Alembert integrators preserve energy much more accurately than the non-geometric explicit methods. 
\begin{figure}
	\centering
		\includegraphics[width=\textwidth]{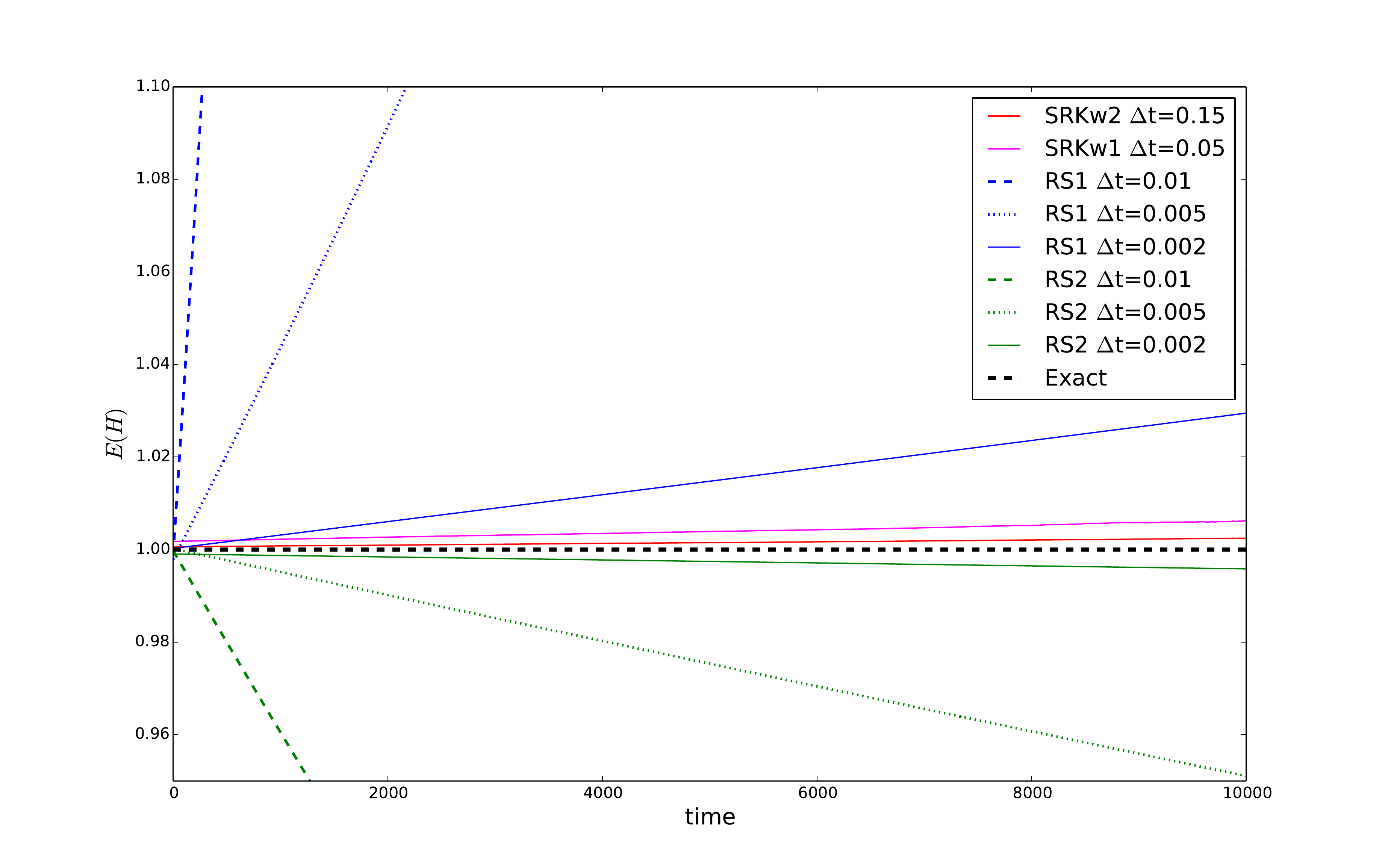}
		\caption{The numerical value of the mean Hamiltonian $E(H)$ as a function of time for the simulations of the Vlasov equation with the Lorentz collision operator with the parameters $\nu=0.005$, $E_0=3$, $\epsilon_1=\epsilon_2=0.25$, and the initial conditions sampled from the distribution \eqref{eq:Initial conditions for the probability density for Lorentz operator}, is shown for the solutions computed with the weak methods, and compared to the exact value $E(H)_{\text{exact}}=1$. The $SRKw1$ and $SRKw2$ methods accurately reproduce the conservation of mean energy over a long integration time even with relatively large time steps. The $RS2$ method requires the significantly smaller time step $\Delta t = 0.002$ to reach a comparable lever of accuracy, while the method $RS1$ remains less accurate even for $\Delta t = 0.002$.}
		\label{fig: Energy for Vlasov with Lorentz operator - weak integrators}
\end{figure}
On the other hand, weak integrators aim to approximate the probability distribution and functionals of the exact solutions rather than each sample path, therefore they may not preserve energy on each sample path, but nevertheless they should approximate the mean energy \eqref{eq:Preservation of energy for Lorentz operator}. In order to test the long-time performance of the weak integrators discussed in Section~\ref{sec:Weak Examples}, simulations with the same parameters as above were carried out until the time $T=10000$. In each case $10^6$ sample paths were generated. The initial conditions were randomly drawn from the probability distribution \eqref{eq:Initial conditions for the probability density for Lorentz operator}. The exact mean energy can be calculated by substituting \eqref{eq:Hamiltonian and forces for the Vlasov equation with Lorentz operator}, \eqref{eq:Electrostatic potential for Lorentz operator}, and \eqref{eq:Initial conditions for the probability density for Lorentz operator} into \eqref{eq:Preservation of energy for Lorentz operator}. For the chosen parameters, we have $E(H)_{\text{exact}}=1$.The numerical value of the mean Hamiltonian $E(H)$ as a function of time is depicted in Figure~\ref{fig: Energy for Vlasov with Lorentz operator - weak integrators} for the $SRKw1$, $SRKw2$, $RS1$, and $RS2$ methods. We see that the weak Lagrange-d'Alembert methods accurately reproduce the mean energy conservation even with the relatively large time steps, while the non-geometric methods require much smaller time steps to reach a comparable level of accuracy.  The maximum relative Monte Carlo error $\sigma(E(H))/E(H)$ did not exceed $0.001$ in any of the simulations.

\section{Summary and future work}
\label{sec:Summary}

In this paper we have presented a general framework for constructing a new class of stochastic variational integrators for stochastic forced Hamiltonian systems. We have extended the approach taken in \cite{HolmTyranowskiGalerkin} by considering the stochastic Lagrange-d'Alembert principle and constructing the corresponding structure-preserving schemes, which we have dubbed stochastic Lagrange-d'Alembert variational integrators. We have shown that in the presence of a symmetry such integrators satisfy a discrete version of Noether's theorem. We have further considered certain classes of mean-square and weak Runge-Kutta methods previously known in the literature, and determined the conditions under which such methods become Lagrange-d'Alembert integrators. We have finally pointed out several examples of low-stage Runge-Kutta methods of that type, and demonstrated their superior long-time numerical performance via numerical experiments. In particular, as one of the test cases we have considered the Vlasov-Fokker-Planck equation and proposed a new geometric approach to the simulation of collisional kinetic plasmas.

Our work can be extended in several ways. The mean-square partitioned Runge-Kutta methods introduced in Section~\ref{sec:Mean-square Lagrange-d'Alembert partitioned Runge-Kutta methods} only use the increments $\Delta W^r = \int_{t_k}^{t_{k+1}}dW^r(t)$, therefore their mean-square order of convergence cannot exceed 1.0 (see \cite{Burrage2000}, \cite{MilsteinRepin2001}, \cite{MilsteinRepin}). To obtain mean-square convergence of higher order one can extend the definitions of the discrete Hamiltonian \eqref{eq:Discrete Hamiltonian for SPRK} and the discrete forces \eqref{eq:Discrete forces for SPRK} to include higher-order multiple Stratonovich integrals, e.g., to achieve convergence of order 1.5 we would need to include terms involving $\Delta Z^r = \int_{t_k}^{t_{k+1}}\int_{t_k}^{t}dW^r(\xi)\,dt$; see \cite{HolmTyranowskiGalerkin} for an example how this can be done for unforced Hamiltonian systems. Another aspect worth a more detailed investigation is the issue of ergodicity of the Lagrange-d'Alembert methods. In Section~\ref{sec:Ergodic limits} and Section~\ref{sec:Vlasov equation} we have experimentally demonstrated the usefulness of our integrators in calculating the ergodic limits, but have not formally proved their ergodicity. It would be beneficial to determine under what conditions Lagrange-d'Alembert integrators can be ergodic in the sense discussed in, e.g., \cite{Mattingly2002}, \cite{Mattingly2010}, or \cite{Talay2002}, when applied to ergodic Hamiltonian systems. It would also be interesting to extend the idea of Lagrange-d'Alembert integrators to stochastic Hamiltonian systems that are both forced and constrained. Structure-preserving numerical methods for such systems would be of great interest in molecular dynamics (see \cite{BouRabeeConstrainedSVI}, \cite{Ciccotti2008}, \cite{VandenCiccotti2006}). Yet another direction of great practical significance would be a further study of the geometric approach to collisional kinetic plasmas presented in Section~\ref{sec:Vlasov equation} and application of more realistic collision operators that preserve the total energy and momentum, as well as an extension to the self-consistent Maxwell-Vlasov equations (see \cite{KrausPHD}, \cite{KrausGEMPIC}). Finally, one may extend the idea of variational integration to stochastic multisymplectic partial differential equations such as the stochastic Korteweg-de Vries, Camassa-Holm or Hunter-Saxton equations. Theoretical groundwork for such numerical schemes has been recently presented in \cite{HolmTyranowskiVirasoro}.

\section*{Acknowledgements}

We would like to thank Christopher Albert, Darryl Holm, Katharina Kormann, Omar Maj, Philip Morrison, Bruce Scott, Cesare Tronci, and Udo von Toussaint for useful comments and references. We are particularly endebted to Eric Sonnendr\"{u}cker for pointing out the connections between stochastic systems and the Vlasov equation. The study is a contribution to the Reduced Complexity Models grant number ZT-I-0010 funded by the Helmholtz Association of German Research Centers.


\end{document}